\documentclass{amsart}

\usepackage{amsthm,amssymb,latexsym,amsmath,color,mathrsfs,caption,enumitem}
\usepackage[all]{xy}
\usepackage{mathrsfs}
\usepackage{mathabx}
\usepackage[utf8]{inputenc}
\usepackage{multirow}
\usepackage[final,notref,notcite]{showkeys}
\usepackage{enumitem}
\usepackage{mathdots}
\usepackage{eucal}
\usepackage{marginnote}
\usepackage{hyperref}
\usepackage{dsfont}
\usepackage[right=3cm,left=3cm,top=2.2cm,bottom=2.2cm]{geometry}
\newcommand{\III}{{\mathds I}}
\newcommand{\VV}{{\mathds V}}
\newcommand{\Z}{{\mathds Z}}
\newcommand{\N}{{\mathds N}}
\newcommand{\C}{\mathds{C}}

\newcommand{\p}[1]{{\mathds{P}^{#1}}}
\newcommand{\pn}{{\mathds{P}^n}}

\newcommand{\op}[1]{{\mathcal O}_{\mathds{P}^{#1}}}
\newcommand{\opn}{{\mathcal O}_{\mathds{P}^n}}

\newcommand{\tpn}{{\rm T}{\mathds{P}^{n}}}
\newcommand{\tx}{{\rm T}{X}}
\newcommand{\txx}{{\rm T}_{\pn}\langle X\rangle}
\newcommand{\txprime}{{\rm T}'_{X}}

\newcommand{\SL}{\mathrm{SL}}

\def\redu{{\mathrm{red}}}
\def\tf{{\mathrm{tf}}}
\def\t{{\mathrm{t}}}
\def\diag{{\mathrm{diag}}}

\newcommand{\sing}{\operatorname{Sing}}
\newcommand{\supp}{\operatorname{Supp}}

\newcommand{\calc}{{\mathcal C}}
\newcommand{\cald}{{\mathcal D}}
\newcommand{\cale}{{\mathcal E}}
\newcommand{\calf}{{\mathcal F}}
\newcommand{\calg}{{\mathcal G}}

\newcommand{\cali}{{\mathcal I}}

\newcommand{\calk}{{\mathcal K}}
\newcommand{\call}{{\mathcal L}}
\newcommand{\calm}{{\mathcal M}}
\newcommand{\caln}{{\mathcal N}}
\newcommand{\calo}{{\mathcal O}}
\newcommand{\calp}{{\mathcal P}}
\newcommand{\calq}{{\mathcal Q}}
\newcommand{\calr}{{\mathcal R}}
\newcommand{\cals}{{\mathcal S}}
\newcommand{\calt}{{\mathcal T}}

\newcommand{\calv}{{\mathcal V}}

\newcommand{\inext}{{\mathcal E}{\it xt}}

\newcommand{\Ext}{\operatorname{Ext}}
\newcommand{\Hom}{\operatorname{Hom}}

\newcommand{\Tor}{\operatorname{\mathcal{T}or}}
\newcommand{\pdim}{\mathrm{pdim}}
\newcommand{\gpdim}{\mathrm{gpdim}}
\DeclareMathOperator{\End}{End}

\DeclareMathOperator{\coker}{coker}
\DeclareMathOperator{\GL}{GL}
\DeclareMathOperator{\im}{im}
\DeclareMathOperator{\codim}{{codim}}
\DeclareMathOperator{\depth}{{depth}}
\DeclareMathOperator{\rk}{{rk}}
\DeclareMathOperator{\cork}{{cork}}

\DeclareMathOperator{\lcm}{lcm}

\def\sD{{\mathscr{D}}}

\def\Ts{{\mathcal T}_{\sigma}}

\newcommand{\lra}{\longrightarrow}
\newcommand{\into}{\hookrightarrow}
\newcommand{\onto}{\twoheadrightarrow}

\newtheorem{theorem}{Theorem}[section]
\newtheorem{mthm}{Theorem}

\newtheorem{proposition}[theorem]{Proposition}

\newtheorem*{defn}{{\bf Definition}}

\newtheorem{lemma}[theorem]{Lemma}
\newtheorem{corollary}[theorem]{Corollary}

\theoremstyle{definition}
\newtheorem{remark}[theorem]{Remark}
\newtheorem{example}[theorem]{Example}
\newtheorem{definition}[theorem]{{\bf Definition}}

\title{Logarithmic sheaves of complete intersections}

\author[D. Faenzi]{Daniele Faenzi}
\address{Daniele Faenzi
Institut de Mathématiques de Bourgogne,
UMR 5584 CNRS,
Université de Bourgogne,
9 Avenue Alain Savary,
F-21000, Dijon, France}
\email{daniele.faenzi@u-bourgogne.fr}

\author[M. Jardim]{Marcos Jardim}
\address{Marcos Jardim. Universidade Estadual de Campinas \\ Departamento de Matem\'atica \\
Rua S\'ergio Buarque de Holanda, 651\\ 13083-970 Campinas-SP, Brazil}
\email{jardim@unicamp.br}

\author[J. Vallès]{Jean Vallès}
\address{Jean Vall\`es. Universit\'e de Pau et des Pays de l'Adour,
  LMAP-UMR CNRS 5142, 
  Avenue de l'Universit\'e - BP 1155 -
  64013 Pau Cedex, France}
 \email{jean.valles@univ-pau.fr}
\date{\today}

\makeatletter
\let\@wraptoccontribs\wraptoccontribs
\makeatother

\contrib[With an appendix by]{Alan Muniz}
\address{Alan Muniz.  Universidade Estadual de Campinas \\ Departamento de Matem\'atica \\
Rua S\'ergio Buarque de Holanda, 651\\ 13083-970 Campinas-SP, Brazil}
\email{anmuniz@ime.unicamp.br}

\keywords{Logarithmic sheaves, freeness, and local freeness. Complete intersections. Stable syzygy sheaves. Pencils of quadrics. Algebraic Foliations.}

\subjclass[2010]{AF404; 14J60; 14M10; 32S65}

\begin{document}
\sloppy

\maketitle
\begin{abstract}
We define logarithmic tangent sheaves associated with subvarieties of codimension larger than 1 in connection with Jacobian syzygies and distributions. We analyze the notions of local freeness, freeness, and stability of these sheaves. 

We carry out a complete study of logarithmic sheaves associated with pencils of quadrics and compute their projective dimension from the classical invariants such as the Segre symbol and new invariants (splitting type and degree vector) designed for the classification of irregular pencils. This leads to a complete classification of free (equivalently, locally free) pencils of quadrics.

Finally, we produce examples of locally free, non-free pencils of surfaces in \(\p 3\) of arbitrary degree \(k \ge 3\), answering (in the negative) a question of Calvo-Andrade, Cerveau, Giraldo, and Lins Neto about codimension foliations on \(\p 3\).
\end{abstract}

\tableofcontents

\section{Introduction}

Let $\kappa$ be an algebraically closed field of characteristic zero and consider a sequence $\sigma=(f_1,\dots,f_k)$ of homogeneous polynomials $f_i\in R = \kappa[x_0,\dots,x_n]$ of degree $d_i+1$, for some $0\le d_1\le\dots\le d_k$ and $k\le n$. Let $I_\sigma:=(f_1,\ldots,f_k)$ denote the ideal generated by $\sigma$, and $\VV(\sigma)$ be the associated closed subscheme in $\pn$. Consider the Jacobian matrix of  $\sigma$, namely:
$$ \nabla \sigma = \left( \begin{array}{c}
\nabla f_1 \\ \vdots \\ \nabla f_k
\end{array} \right)  \colon \opn^{\oplus n+1} \lra \oplus_{i=1}^k \opn(d_i)$$
The Jacobian scheme of $\sigma$ is defined by the maximal minors of this matrix; it may have several components of different dimensions and degrees. We will say that $\sigma$ has a \textit{singular divisor of degree $d$} if the Jacobian scheme contains a divisor of degree $d$.

The main focus of this paper, rather than the Jacobian scheme, is the sheaf
\[
\calt_\sigma:=\ker(\nabla\sigma)
\]
which we assume to have rank equal to $n-k+1$. By the Jacobian criterion, this is equivalent to assuming that $\sigma$ is \textit{algebraically independent}, that is there is no polynomial $G\in\kappa[x_1,\dots,x_k]$ such that $G(f_1,\dots,f_k)=0$, see \cite[Ch. I, 11.4]{Lefschetz}. More precisely, we will be mostly concerned with the case when $\sigma$ is a regular sequence so that $\VV(\sigma)$ is a complete intersection subscheme.

We call $\calt_\sigma$ the \emph{logarithmic tangent sheaf associated with $\sigma$}. This nomenclature for $\calt_\sigma$ is motivated by the following observation. Assuming that $\sigma$ is a regular sequence, set $X=\VV(\sigma)$ and recall that the Zariski tangent sheaf $\tx$ and the sheaf $\txprime$ supported at $\sing(X)$ fit into:
\[
0 \to \tx \to \tpn |_X \to \bigoplus_{i=1}^k \calo_X(d_i+1) \to \txprime \to 0.
\]
The \emph{sheaf $\txx$ of vector fields on $\pn$ tangent to $X$} is the kernel of the natural morphism $\tpn \to \bigoplus_{i=1}^k \calo_X(d_i+1)$, see \cite[Chapter 3]{Se}.
It turns out that $\calt_\sigma(1)$ is a subsheaf of rank $n-k+1$ of $\txx$.
More precisely (see Lemma \ref{five terms}), writing 
$\calv_\sigma=\bigoplus_{i=1}^k \cali_{X}(d_i+1) / \calo_{\pn}$
and denoting by $\calq_\sigma$ the cokernel of $\nabla\sigma$, we have:
\[
0 \to \calt_\sigma(1) \to \txx \to \calv_\sigma \to \calq_\sigma(1) \to \txprime \to 0.
\]

When the sequence $\sigma$ consists of a single polynomial $f$ (so that $k=1)$, then $\calv_\sigma=0$ so $\calt_\sigma(1) \simeq \txx$ is precisely the logarithmic tangent sheaf associated with the divisor $\VV(f)$, see for instance \cite{U-Y} or the celebrated \cite{Sa}.

Note that, for $k \ge 2$, the sheaf $\txx$ cannot be locally free. On the other hand, as we shall see, $\calt_\sigma$ may be locally free or even completely decomposable as a sum of line bundles. Hence, we propose the following three definitions, whose goal is to generalize the usual concept of a \emph{free divisor} introduced in \cite{Sa}. 

\begin{defn}\label{defn free}
An algebraically independent sequence $\sigma$ is said to be:
\begin{enumerate}
\item \emph{locally free} if the associated logarithmic tangent sheaf $\calt_\sigma$ is locally free.
\item \emph{free} if the logarithmic tangent sheaf $\calt_\sigma$ splits as a sum of line bundles.
\item \emph{strongly free} if every sequence $\sigma'$ such that $I_{\sigma'}=I_\sigma$ is free.
\end{enumerate}
\end{defn}

Clearly, every free algebraically independent sequence of length $k=1$ is also strongly free. On the other extreme case, we observe that if $\sigma$ is an algebraically independent sequence of length $k=n$, then $\calt_\sigma=\opn(e)$ for some negative integer $e$, since every rank 1 reflexive sheaf on $\pn$ is a line bundle. Therefore, every algebraically independent sequence $\sigma$ in $\kappa[x_0,\dots,x_n]$ of length $n$ is strongly free. We provide explicit examples of free and strongly free regular sequences of length 2 in $\kappa[x_0,x_1,x_2,x_3]$, see Example \ref{free not strong} and Example \ref{strong} below.

Regarding the middle range $1<k<n$, recall that it is notoriously hard to construct indecomposable locally free sheaves of rank $r$ on $\pn$ when $2\le r\le n-2$, cf. \cite[Problem 1 and subsequent paragraph]{H}. In fact, only two examples, up to twist and affine pullback, are known: the Horrocks--Mumford rank 2 bundle on $\p4$ and Horrocks' rank 3 bundle on $\p5$. Furthermore, the Hartshorne conjecture (cf. \cite[Problem 3 and subsequent paragraph]{H}) seems to indicate that every locally free sheaf of rank $r$ on $\pn$ with $3r < n$ splits as a sum of line bundles, as the vanishing locus of a general section of a sufficiently positive twist of a non-split bundle would be a good candidate to contradict the conjecture.
This would imply that locally free algebraically independent sequences of length $k$ in $\kappa[x_0,\dots,x_n]$ are free whenever $3k>2n+3$.

With these facts in mind, it seems natural to investigate algebraically independent sequences of length $2$ in $R=\kappa[x_0,\dots,x_n]$. Following two directions usually pursued in the literature concerning logarithmic sheaves for hypersurfaces, our goal is to find criteria to determine when, on the one hand, an algebraically independent sequence $\sigma=\{f_1,f_2\}$ is (locally) free, and, on the other hand, when the associated logarithmic tangent sheaf $\Ts$ is slope-stable.

We start in Section \ref{prelims} by setting up basic definitions and examples. Most importantly, we provide criteria to determine when an algebraically independent sequence in $R$ is locally free (see Lemma \ref{lf no codim 3}), as well as examples of free regular sequences that are not strongly free (Example \ref{free not strong}), and of a strongly free regular sequence (Example \ref{strong}). In Section \ref{sec:sqc-dist} we show that algebraically independent sequences of length $k$ induce codimension $k-1$ distributions on $\pn$, see Lemma \ref{lem:dist}; these distributions are in fact integrable, and, as it is shown in the Appendix, coincide with the rational foliations introduced by Cukierman, Pereira, and Vainsencher in \cite{CPV}. Yet another important notion in this paper is the compressibility of an algebraically independent $\sigma$ (that is, an algebraically independent sequence of length $k$ consisting only of forms of equal degree): $\sigma$ is said to be \emph{compressible} if there is a variable that occurs in none of the polynomials contained in $\sigma$; an algebraically independent sequence that is not compressible is said to be \emph{incompressible} (see Definition \ref{comp}).

Sections \ref{sec:stable-pencils} through \ref{locally free pencils} are dedicated to a thorough study of algebraically independent sequences consisting of 2 polynomials of degree 2, also referred to as \emph{pencils of quadrics}.

First, recall that the \emph{slope} of a torsion-free sheaf \(F\) of rank \(p>0\) on \(\pn\) of determinant \(\big(\bigwedge^p F\big)^{\vee \vee} \simeq \calo_\pn(e)\) is defined as \(\mu(F)=e/p\).
The sheaf \(F\) is said to be \emph{slope-(semi)stable} if any proper subsheaf \(K\) of \(E\) has slope \(\mu(K)<(\le)~\mu(E)\); $F$ is \emph{slope-polystable} if it is the direct sum of slope-stable sheaves with the same slope, and $F$ is \emph{slope-unstable} if it is not slope-semistable.
The following result is proved in Section \ref{sec:stable-pencils}.

\begin{mthm}\label{mthm:stable}
Let $\sigma$ be a pencil of quadrics in \(\pn\) and let $r_0$ be the maximal corank of the Hessian matrix of the quadrics in the pencil.
\begin{enumerate}[leftmargin=1cm, label=\arabic*. ]   
\item \label{thm c1=0} The singular divisor of $\sigma$ consists of two simple hyperplanes or a double hyperplane if and only if $c_1(\calt_\sigma)=0$ and this happens if and only if $\calt_\sigma \simeq \opn^{\oplus(n-1)}$.
\item \label{thm c1=-1} If $\sigma$ contains one and only one simple singular hyperplane, then $c_1(\calt_\sigma)=-1$ and $\calt_\sigma$ is slope-stable if and only if $\sigma$ is incompressible.
\item If $\sigma$ is compressible and contains no singular hyperplane, then $\calt_\sigma$ is slope-unstable.
\item If $\sigma$ is incompressible and contains no singular hyperplane, then
\begin{enumerate}[label=\roman*)]
\item \(\calt_\sigma\) is slope-stable when \(2r_0 < n+1\);
\item \(\calt_\sigma\) is strictly slope-semistable when \(2r_0 = n+1\);
\item \(\calt_\sigma\) is slope-unstable when \(2r_0 > n+1\).
\end{enumerate}
\end{enumerate}
\end{mthm}

The upshot is that, for the most interesting case  (namely that of incompressible pencils without double hyperplanes), stability depends only on the maximal corank \(r_0\) of the quadrics in the pencil.
By \cite{AL}, semistability of a pencil of quadrics in the sense of geometric invariant theory is equivalent to the fact that the discriminant of the pencil is non-zero (i.e. the pencil is regular) and has no root of multiplicity greater than $(n+1)/2$. So there are many GIT-unstable pencils $\sigma$ whose logarithmic sheaf $\calt_\sigma$ is still slope-semistable or even slope-stable, see Remark \ref{remark on GIT}.

Next, we look at freeness and local freeness of pencils of quadrics and, more generally, at the projective dimension of \(\calt_\sigma\), both in the local and in the graded senses.
This turns out to depend on more subtle invariants of the pencil.
To review them, note that the pencil of quadrics defined by \(\sigma\) gives a symmetric matrix \(\rho_\sigma\) of linear forms on \(\p1\), whose \textit{generic corank} \(r_1\) is the corank of the Hessian matrix of a generic quadric in the pencil.
Note that \(r_1=0\) if and only if \(\sigma\) contains smooth quadrics, we call \(\sigma\) \emph{regular} in this case and \emph{irregular} otherwise.
When \(\sigma\) is irregular, there are integers \(1\le c_1 \le \cdots \le c_{r_1-m}\), with $m=\dim(H^0(\calt_\sigma))$ such that 
the torsion-free part of \(\calc_\sigma=\coker(\rho_\sigma)\) is \(\bigoplus_{i=1}^{r_1} \calo_{\p 1}(c_i)\). We call \(\mathbf{c}=(c_1,\ldots,c_{r_1})\) the \emph{degree vector} of \(\sigma\).
If $\Lambda=\{\lambda_1,\ldots,\lambda_\ell\} \subset \p1$ is the support of the torsion part \(\calc_\t\) of \(\calc_\sigma\), then, for each \(j\in \{1,\ldots,\ell\}\),
denoting by  \(\lambda_j^{(a)}\) the \(a\)-tuple structure over \(\lambda_j\), the localization at \(\lambda_j\) of \(\calc_\t\) is
\(
\bigoplus_{i=1}^{s_j} \calo_{\lambda_j^{(a_{j,i})}}^{\oplus p_{j,i}},
\)
for some \(s_j\) and \((a_{j,i},p_{j,i} \mid i \in \{1,\ldots,s_j\})\).
These data are arranged into the \emph{Segre symbol} \(\Sigma=[\Sigma_1,\ldots,\Sigma_\ell]\), defined for all \(j \in \{1,\ldots,\ell\}\) by:
\[
\Sigma_j=(\underbrace{a_{j,1},\ldots,a_{j,1}}_{p_{j,1}},\ldots,\underbrace{a_{j,s_j},\ldots,a_{j,s_j}}_{p_{j,s_j}}), \qquad \mbox{with \(a_{j,1} > \cdots > a_{j,s_j}\)}.
\]

The data $(r_1,\Lambda,\Sigma)$ completely characterize an incompressible pencil of quadrics up to homography, a result that goes back to Kronecker, Segre and Weierstrass, see Section \ref{segre symbols} for further details. 

With these data, we describe the scheme-theoretic structure of the Jacobian scheme, when $\sigma$ is regular, as a union of nilpotent structures on pairwise disjoint linear spaces whose dimension and degree of nilpotency depend on $\Sigma$ and whose position depends on $\Lambda$.
If $\sigma$ is irregular, the Jacobian scheme contains an additional component which is a rational normal scroll of dimension \(r_1\) and degree \(c_1+\cdots+c_{r_1}\) that connects all the linear spaces, with a prescribed intersection along each space.

The upshot is that these invariants also characterize the projective and graded projective dimensions of $\Ts$, as it is described in the following two results, proved in Section \ref{ext sheaves} and Section \ref{sec:pdim}, respectively

\begin{mthm}
Let \(\sigma\) have Segre symbol \(\Sigma\). For \(q>0\),  \(\inext_{\pn}^q(\calt_\sigma,\calo_\pn) \ne 0\) if and only if there are \(j \in \{1,\ldots,\ell\}\) and \(k \in \{1,\ldots,s_j\}\) such that:
\[
q + p_{j,1}+\ldots+p_{j,k} = n-r_1-1,
\]
or \(r_1>0\) and \(q+r_1=n-2\).
In particular, we have:
\begin{enumerate}[label=\roman*)]
        \item if \(\sigma\) is regular and \(p = \min\{p_{j,1} \mid j \in \{1,\ldots,\ell\}\}\), then
    \(\pdim(\calt_\sigma)= n-p-1;
    \)
    \item if \(\sigma\) is irregular, then \(\pdim(\calt_\sigma)= n-r_1-2\).
\end{enumerate}
\end{mthm}

We have a rather different situation for the graded projective dimension \(\gpdim(\calt_\sigma)\), namely, the projective dimension of the module of global sections 
of \(\calt_\sigma\). This is summarized in the following result.

\begin{mthm}
For a regular pencil of quadrics \(\sigma\) in \(\pn\) we have \(\gpdim(\calt_\sigma)=n-2\) except if \(\sigma\) has Segre symbol \([1^p,1^q]\) for \(p \ge q \ge 1\) or \([(2^q,1^p)]\) with \(q \ge 1\). In both these cases
    \(\gpdim(\calt_\sigma)=n-q-1.\)

For an irregular pencil of quadrics \(\sigma\) of generic corank \(r_1\) we have \(\gpdim(\calt_\sigma)=n-1\)
except if \(\sigma\) has degree vector \((1,\ldots,1)\), in which case \(\gpdim(\calt_\sigma)=n-r_1-2.\)
\end{mthm}

With this in mind, after a careful analysis of pencils of quadrics in \(\p 3\), performed in Section \ref{another}, we come to the conclusion that freeness and local freeness are equivalent conditions for pencils of quadrics in \(\p n\) and we completely classify pencils satisfying such condition.

\begin{mthm}\label{mthm:free}
A pencil of quadrics $\sigma$ in $\pn$, $n\ge3$, is free if and only if $\calt_\sigma$ is locally free. More precisely, the only free pencils of quadrics are displayed in Table \ref{table of free}.
\end{mthm}

By contrast, we provide in Section \ref{locally free higher} a series of examples of locally free pencils of degree $k\ge3$ that are not free. This indicates that potentially interesting vector bundles may arise as logarithmic sheaves associated with regular sequences of higher degrees having deep singularities. To understand our following result, recall that a \emph{null correlation bundle} is defined as the cokernel of a non vanishing morphism $\op3(-1)\to\Omega^1_{\p3}(1)$; every slope-stable rank 2 locally free sheaf $N$ on $\p3$ with $c_1(N)=0$ and $c_2(N)=1$ arises in this way.

\begin{mthm} \label{main nonfree}
Fix $k\ge0$, and consider the pencil \(\sigma = (f,g)\) of degree $k+3$ with:
$$ f= x_0x_1^{k+2}+x_2^{k+3}+x_2^{k+2}x_3 ~~\text{  and  }~~
g = x_2x_3(x_2^{k+1}-x_1^{k+1}). $$
Then \(\calt_\sigma \simeq N(-k-2)\), where \(N\) is a null correlation bundle.
\end{mthm}

We complete this paper with an application of our results to the study of rational codimension one foliations on $\pn$, see Section \ref{sec:ratfols}. To be precise, recall that a \emph{rational 1-form} is a twisted 1-form given by the expression
$$ \omega = (d_1+1)f_1\cdot \mathrm{d}f_2 - (d_2+1)f_2\cdot \mathrm{d}f_1 ~ \in H^0(\Omega_{\pn}^1(d_1+d_2+2)), $$
where $f_i\in H^0(\opn(d_i+1))$ for $i = 1,2$ and $f_1,f_2$ have no common factors. Regarding $\omega$ as a morphism $\tpn\to\opn(d_1+d_1+2)$, we consider the kernel sheaf $\calk_\omega:=\ker\omega$.  We show in Section \ref{sec:ratfols} that the natural 1-1 correspondence $(f_1,f_2)\leftrightarrow\omega$ between regular sequences of length 2 and rational 1-forms is such that $\calk_\omega=\Ts(1)$, see Lemma \ref{rat-split}. 

This fact has two important consequences. First, we can invoke a result from the general theory of codimension one distributions on $\p3$, presented in \cite{CCJ1}, to obtain simple criteria to establish when $\Ts$ is slope-(semi)stable, see Corollary \ref{cor:stable}.

Second, we provide a negative answer to a problem posed by Calvo-Andrade, Cerveau, Giraldo, and Lins Neto, see \cite[Problem 2]{CACGLN}; namely, these authors asked whether the tangent sheaf of a codimension one foliation must split as a sum of line bundles whenever it is locally free. While Theorem \ref{mthm:free} implies that this claim is true for rational foliations of type $(2,2)$, Theorem \ref{main nonfree} says that for each $k\ge0$ there are rational foliations of type $(k+3,k+3)$ on $\p3$ whose tangent sheaf is a slope-stable locally free sheaf.


\subsection*{Acknowledgments}
We thank Eduardo Esteves for pointing us out to the results of \cite{AL}, Cleto Miranda Neto for useful discussions, and Alan Muniz for revising a preliminary version of the paper, as well as writing an Appendix that elucidates the relation between the new theory here proposed and the theory of holomorphic foliations on projective spaces.
This research is part of the CAPES/COFECUB project \emph{Moduli spaces in algebraic geometry and applications}, Capes reference number 88887.191919/2018-00. 
D.F. is partially supported by EIPHI Graduate School ANR-17-EURE-0002 and ANR Fano-HK ANR-20-CE40-0023.
M. J. is partially supported by the CNPQ grant number 305601/2022-9 and the FAPESP Thematic Project 2018/21391-1.
The authors also acknowledge the financial support from Coordenação de Aperfeiçoamento de Pessoal de Nível Superior - Brasil (CAPES) - Finance Code 001 and ANR BRIDGES ANR-21-CE40-0017.

We warmly thank the referee of the first version of the manuscript for pointing out several classical references and results that greatly helped improve this paper.

\section{Basic definitions and examples}\label{prelims}

Let \(\kappa\) be a field of characteristic 0, \(n \in \N\) and put \(R=\kappa[x_0,\ldots,x_n]\).
In this section, we will give some preliminary properties of logarithmic tangent sheaves associated with complete intersections.
Many properties remain valid for \(\kappa\)  of characteristic different from 2.

\subsection{General framework}

Let $\sigma:=\{f_1,\dots,f_k\}$ be an algebraically independent sequence of homogeneous polynomials in \(R\), with  $d_i=\deg(f_i)-1$ for $i \in \{1,\ldots,k\}$. Consider the associated subscheme $X=\VV(\sigma)\subset\pn$ and set  $\nabla\sigma$ for the associated Jacobian matrix. We write $\partial_j$ the partial derivative $\frac{\partial}{\partial x_j}$ for $j \in \{0,\ldots,n\}$ and 
$\nabla f=(\partial_0f, \ldots ,\partial_n f)$ the gradient of a homogeneous polynomial \(f \in R\).

We put $\calt_\sigma$ for the associated logarithmic tangent sheaf as defined in the Introduction, namely $\calt_\sigma$ is the kernel of $\nabla\sigma$.
In addition, we define the sheaves $\calm_\sigma:=\im(\nabla\sigma)$ and $\calq_\sigma:=\coker(\nabla\sigma)$. The sheaf $\calm_\sigma$ is torsion-free, it can be thought of as the natural extension to $\pn$ of the equisingular normal sheaf of $X=\VV(\sigma)$; furthermore, the sheaf $\calt_\sigma$ is reflexive because it is the kernel of a morphism between locally free sheaves.  
We have the fundamental exact sequences:
\begin{equation} \label{tn-sec}
0 \lra \calt_\sigma \lra \opn^{\oplus n+1} \stackrel{\nabla\sigma}{\lra} \calm_\sigma \lra 0,
\qquad 
0 \lra \calm_\sigma \lra \oplus_{i=1}^k \opn(d_i) \lra \calq_\sigma  \lra 0,
\end{equation}

The sheaf $\calt_\sigma$ can be viewed as the sheaf of syzygies of the Jacobian matrix $\nabla\sigma$.
Indeed, a syzygy of degree $a$ of $\nabla\sigma$ is a morphism $\nu:\opn(-a)\to\opn^{\oplus n+1}$ such that $\nabla\sigma\circ\nu=0$.
Write ${\rm Syz}_a(\nabla\sigma)$ for the vector space of syzygies of degree $a$ for the matrix $\nabla\sigma$.
Note that every syzygy of degree $a$ for $\nabla\sigma$ induces a section in $H^0(\calt_\sigma(a))$. Conversely, every non-trivial section in $H^0(\calt_\sigma(a))$ induces a syzygy of degree $a$ for $\nabla\sigma$, thus we obtain an isomorphism of vector spaces
\begin{equation}\label{iso-syz}
H^0(\calt_\sigma(a)) \simeq {\rm Syz}_a(\nabla\sigma).
\end{equation}

We define the \emph{Jacobian scheme} $\Xi_\sigma$ as the degeneracy locus of $\nabla\sigma$,
$$ \Xi_\sigma := \VV \left( \bigwedge^k \nabla\sigma \right).$$
This is the subscheme of $\pn$ defined by the common zeros of the $k\times k$ minors of $\nabla\sigma$.
The reduced structure $(\Xi_\sigma)_{\rm red}$ coincides with the support of the sheaf $\calq_\sigma$. Note that $(\Xi_\sigma)_{\rm red}$ may contain a hypersurface.

\smallskip

More precisely, the image of the exterior power morphism:
\begin{equation}\label{det-jc}
\bigwedge^k\nabla\sigma : \opn^{\oplus{\binom{n+1}{k}}} \to \opn(d_1+\cdots+d_k);
\end{equation}
is of the form  $\cali_{W_\sigma}(d_1+\cdots+d_k-l)$, where $l:=c_1(\calq_\sigma)$ is the degree of the divisorial part $\Xi_{\sigma}$, and $W_\sigma\subset\pn$ is a subscheme of codimension at least 2, possibly not pure. 
Let us illustrate this discussion with an explicit example.

\begin{example}\label{ex s=1 t=3}
Let $f=x_0x_1 + x_2x_3$ and $g=x_0x_1x_2x_3$, and note that neither $f$ and $g$ have common factors, nor do $\nabla f$ and $\nabla g$; however, one can check that  $(\Xi_\sigma)_{\rm red}$ contains a hypersurface.

Indeed, we have that:
$$ \nabla_{\sigma}:= \left( \begin{array}{cccc}
x_1  & x_0  & x_3  & x_2 \\
x_1x_2x_3  & x_0x_2x_3  & x_0x_1x_3  & x_0x_1x_2 
\end{array} \right). $$
The reduced structure of curve $C=\VV(f,g)$ is the union of the four lines $\VV(x_i,x_j)$ with $i=0,1$ and $j=2,3$.
Two of the $2\times2$ minors of $\nabla\sigma$ vanish identically, and we have that
$$ \bigwedge^2 \nabla\sigma = (x_0x_1-x_2x_3)\cdot 
( 0 ~~~~ x_0x_3 ~~~~ x_1x_2 ~~~~ x_1x_3 ~~~~ x_0x_2 ~~~~ 0 ). $$
It follows that $(\Xi_\sigma)_{\rm red}$ consists of the quadric $\VV(x_0x_1-x_2x_3)$, so that $l=c_1(\calq_\sigma)=2$ in this case, plus two lines $\VV(x_0,x_1)$ and $\VV(x_2,x_3)$; $W_\sigma$ is the union of these two skew lines. 
\end{example}

\begin{lemma} \label{lf no codim 3}
Let $\sigma$ be an algebraically independent sequence. Then:
\begin{enumerate}[label=\roman*)]
\item $\sigma$ is locally free if and only if $\calq_\sigma$ has no subsheaf of codimension \(\ge 3\);
\item if $\sigma$ is locally free,  $\Xi_\sigma$ has no irreducible component of codimension \( \ge 3\). 
\end{enumerate}
\end{lemma}

Note that $(\Xi_\sigma)_{\rm red}$ may have no irreducible component of codimension at least 3 even when $\calq_\sigma$ admits a subsheaf of codimension at least 3, see Section \ref{locally free higher}. This means that the converse of item \textit{(ii)} above does not hold in general.

\begin{proof}
Taking duals of \eqref{tn-sec} we obtain $\inext_{\pn}^{n-1}(\calt_\sigma,\opn)=0$ and, for $j\le n-1$:
$$ \inext_{\pn}^{j}(\calt_\sigma,\opn) \simeq \inext_{\pn}^{j+1}(\calm_\sigma,\opn) \simeq \inext_{\pn}^{j+2}(\calq_\sigma,\opn).$$

The sheaf $\calt_\sigma$ is locally free if and only if
$\inext_{\pn }^{j}(\calt_\sigma,\opn)=0$ for $j\ge1$, which is equivalent to requiring that $\inext_{\pn }^{j}(\calq_\sigma,\opn)=0$ for $j\ge3$. This gives the equivalence in the first claim. 

If $(\Xi_\sigma)_{\rm red}$ has an irreducible component $Y$ of codimension $j\ge3$, then, since $(\Xi_\sigma)_{\rm red}$ is the support of $\calq_\sigma$, it follows that $\calq_\sigma$ has a non-trivial subsheaf $\calv\into\calq_\sigma$ supported on $Y$, hence $\codim \calv=j$. The previous item then implies that $\calt_\sigma$ is not locally free.
\end{proof}

In particular, we observe that every algebraically independent sequence $\sigma$ on $\kappa[x_0,x_1,x_2]$ is locally free. Furthermore, as we pointed out in the introduction, every algebraically independent sequence of length $n$ in $\kappa[x_0,\dots,x_n]$ is free, since $\calt_\sigma=\opn(e)$ for some $e\in\Z$.


\subsection{Algebraically independent sequences and distributions}\label{sec:sqc-dist}

Recall that a \emph{codimension $r$ distribution on $\pn$} is a short exact sequence of the form
$$ \sD ~~ \colon ~~ 0\lra \calt_\sD\lra \tpn \lra \caln_\sD \lra 0 $$ 
where $\caln_\sD$ is a torsion-free sheaf of rank $r$ and $\calt_\sD$ is a reflexive sheaf of rank $n-r$, respectively called the \emph{normal} and \emph{tangent} sheaves of $\sD$. We refer to \cite[Section 2.1]{CCJ1} for further details on the general theory of distributions.

Let us point out how distributions are related to algebraically independent sequences. First, thinking of the Koszul complex attached to $\sigma$ we consider $\tilde \sigma=((d_1+1)f_1,\dots,(d_k+1)f_k)^{\rm t}$ and the \textit{Koszul syzygy} sheaf $\cals_\sigma:=\coker(\tilde \sigma)$, so:
\begin{equation} \label{cals}
0 \to \calo_{\pn}(-1) \to \bigoplus_{i=1}^k \opn(d_i)\to \cals_\sigma \to 0.
\end{equation}

Let $\eta:\opn(-1)\lra\opn^{\oplus n+1}$ be the Euler morphism, namely $\eta = (x_0,\ldots,x_n){}^\t$. The Euler relation gives $\eta\cdot\nabla f_i=(d_i+1)f_i$ for all $i \in \{1,\ldots,k\}$. This allows us to construct the following commutative diagram:
\begin{equation}\label{diag1}
\begin{split} \xymatrix@-2ex{ 
& & 0 \ar[d] & 0 \ar[d] \\
& & \opn(-1)\ar[d]^{\eta}\ar@{=}[r] & \opn(-1)\ar[d]^-{\tilde\sigma} \\
0\ar[r] & \calt_\sigma \ar[r]\ar@{=}[d] & \opn^{\oplus n+1} \ar[r]^-{\nabla\sigma}\ar[d] & \bigoplus_{i=1}^k \opn(d_i)\ar[d] \\
0\ar[r] & \calt_\sigma \ar[r] & \tpn(-1) \ar[r]\ar[d] & \cals_\sigma \ar[d] \\
& & 0 & 0 
} \end{split}
\end{equation}
Here we used that \(\kappa\) is of characteristic zero, or rather that the characteristic of $\kappa$ does not divide $d_i+1$ for all $i \in \{1,\ldots,k\}$.

Note that the image of $\tilde \sigma$ is contained in $\calm_\sigma$ and set
$\caln_\sigma$ for the cokernel of $\tilde\sigma$, corestricted to  $\calm_\sigma$.
The previous diagram gives:
\begin{equation}\label{diag2}
\begin{split} \xymatrix@-2ex{ 
& & 0 \ar[d] & 0 \ar[d] & \\
& & \opn(-1)\ar[d]^{\eta}\ar@{=}[r] & \opn(-1)\ar[d]^{\tilde \sigma} \\
0\ar[r] & \calt_\sigma \ar[r]\ar@{=}[d] & \opn^{\oplus n+1} \ar[r]^{\nabla\sigma}\ar[d] & \calm_\sigma \ar[d]\ar[r] & 0 \\
0\ar[r] & \calt_\sigma \ar[r] & \tpn(-1) \ar[r]\ar[d] & \caln_\sigma \ar[d] \ar[r] & 0\\
& & 0 & 0 
} \end{split}
\end{equation}
Furthermore, we have a second diagram featuring the cokernel sheaf $\calq_\sigma$:
\begin{equation}\label{diag3}
\begin{split} \xymatrix@-2ex{ 
& 0 \ar[d] & 0 \ar[d] & & \\
& \opn(-1)\ar[d]\ar@{=}[r] & \opn(-1)\ar[d]^{\tilde\sigma} & \\
0\ar[r] & \calm_\sigma \ar[r]\ar[d] & \oplus_{i=1}^k \opn(d_i) \ar[r]\ar[d] & \calq_\sigma \ar@{=}[d]\ar[r] & 0 \\
0\ar[r] & \caln_\sigma \ar[r]\ar[d] & \cals_\sigma \ar[r]\ar[d] & \calq_\sigma \ar[r] & 0\\
& 0 & 0 &  
} \end{split}
\end{equation}
It follows that the bottom line in diagram in display \eqref{diag2} defines, for $k\ge2$, a \emph{codimension $k-1$ distribution $\sD_\sigma$ on $\pn$}, given by the exact sequence
\begin{equation}\label{distribution}
\sD_\sigma ~\colon~ 0 \lra \calt_\sigma(1) \lra \tpn \lra \caln_\sigma(1) \lra 0.
\end{equation}

Summing up, we have proved the following statement.

\begin{lemma}\label{lem:dist}
Every algebraically independent sequence $\sigma$ of length $k$ on $n+1$ variables induces a codimension $k-1$ distribution $\sD_\sigma$ on $\pn$ such that $\calt_{\sD_\sigma}=\Ts(1)$.
\end{lemma}

It is shown in Proposition \ref{prop:log-fol} in the Appendix that the distributions $\sD_\sigma$ constructed above are integrable, and coincide with the rational foliations introduced by Cukierman, Pereira, and Vainsencher in \cite{CPV}.

However, not every codimension $k-1$ distribution on $\pn$ comes from an algebraically independent sequence via the construction above. For instance, given a codimension $k-1$ distribution $\sD$ on $\pn$, the monomorphism $\calt_\sD\into\tpn$ may not factor through $\opn(1)^{\oplus n+1}$. 


\subsection{Logarithmic tangent sheaf and deformations}

Let us point out the relationship between our sheaf and classical sheaves of tangent vector fields, in connection with locally trivial deformations of embeddings.

\subsubsection{Tangent vector fields along a complete intersection}

Given a regular sequence $\sigma$, the subscheme $X=\VV(\sigma) \subset \pn$ is a complete intersection whose ideal sheaf $\cali_X$ is generated by $\tilde{\sigma}^\vee : \oplus_{i=1}^k \calo_{\pn}(-d_i) \to \cali_X(1)$. In addition, we have the \emph{equisingular normal sheaf} $N'_{X/\pn}$, see \cite[§ 3.4.4]{Se}, which is defined as the quotient sheaf $\tpn|_X/\tx$, and therefore satisfies the following exact sequence
\[
0 \to \tx \to \tpn|_X \to N'_{X/\pn} \to 0;
\]
here, $\tx$ denotes the Zariski tangent sheaf of $X$. Note also that $N'_{X/\pn}$
is a subsheaf of the normal bundle $N_{X/\pn} \simeq \bigoplus_{i=1}^k \calo_X(d_i+1)$, so:
\[
N'_{X/\pn} \into \bigoplus_{i=1}^k \calo_X(d_i+1);
\]
the quotient of this monomorphism is denoted by $\txprime$, see \cite[§ 1.1.3]{Se}; it is supported at the singular locus of $X$. For further reference, we write its defining exact sequence:
\begin{equation} \label{txprime}
0 \to N'_{X/\pn} \to \bigoplus_{i=1}^k \calo_{X}(d_i+1) \to \txprime \to 0.    
\end{equation}


The \emph{sheaf of vector fields on $\pn$ tangent to $X$}, denoted by $\txx$, is defined as the kernel of the composed epimorphism $\tpn \onto \tpn|_X \onto N'_{X/\pn}$, yielding the exact sequence
\begin{equation} \label{txx}
0 \to \txx \to \tpn \to N'_{X/\pn} \to 0.
\end{equation}

The main motivation for introducing $\txx$ is given by \cite[Proposition 3.4.17]{Se}; namely, $H^1(\txx)$ and $H^2(\txx)$ are the tangent space and the obstruction space of the semi-universal space of locally trivial deformations of the embedding $X \hookrightarrow \pn$. Here we show that $\calt_\sigma(1)$ is a subsheaf of $\txx$, and, in addition, we describe, to a certain extent, the quotient sheaf $\txx/\calt_\sigma(1)$.

\medskip

Since the forms $f_1,\ldots,f_k$ generate the homogeneous ideal of $X$ in $\pn$, we may view $\tilde \sigma=((d_1+1)f_1,\dots,(d_k+1)f_k)^{\rm t}$ as a morphism 
$\calo_\pn \to \bigoplus_{i=1}^k \cali_{X}(d_i+1)$.
We define a torsion-free sheaf $\calv_\sigma=\coker(\tilde \sigma)$ fitting into:
\[
0 \to \calo_\pn \to \bigoplus_{i=1}^k \cali_{X}(d_i+1) \to \calv_\sigma \to 0.
\]

Note that, when $k=1$, we have $\calv_\sigma=0$, as $\cali_X(d_1+1)\simeq \calo_{\pn}$ in this case, so:
\[
\calt_\sigma(1) \simeq \txx, \qquad \mbox{for $k=1$}.
\]

For $k \ge 2$ the relationship between the two sheaves $\calt_\sigma(1)$ and
$\txx$ is expressed by the following lemma.

\begin{lemma} \label{five terms}
We have an exact sequence:
\begin{equation} \label{long one}
0 \to \calt_\sigma(1) \to \txx \to \calv_\sigma \to \calq_\sigma(1) \to \txprime \to 0.    
\end{equation}
\end{lemma}

\begin{proof}
We use the Koszul syzygy sheaf $\cals_\sigma$ of Subsection \ref{sec:sqc-dist} to write the following exact sequence relating $\cals_\sigma$ and $\calv_\sigma$:
\[
0 \to \calv_\sigma \to \cals_\sigma(1) \to \bigoplus_{i=1}^k \calo_{X}(d_i+1) \to 0.
\]
We get a commutative diagram:
\[
\begin{split} \xymatrix@-2.5ex{
&& 0 \ar[d] & 0 \ar[d] \\
&0 \ar[d] \ar[r] & \calo_\pn \ar^-{\tilde \sigma}[r] \ar^-\eta[d] & \bigoplus_{i=1}^k \cali_{X}(d_i+1) \ar[d]\\
0 \ar[r] & \calt_\sigma(1) \ar[d]\ar[r] & \calo_{\pn}^{\oplus(n+1)}(1)\ar[d] \ar[r]^-{\nabla\sigma(1)} & \bigoplus_{i=1}^k \calo_{\pn}(d_i+1) \ar[d]\\
0 \ar[r] & \txx \ar[r] & \tpn \ar[r]^-{\varpi} \ar[d] &   \bigoplus_{i=1}^k \calo_{X}(d_i+1) \ar[d] \\
&& 0 & 0
} \end{split}
\]
where $\varpi$ is given by the composition $\tpn\onto N'_{X/\pn} \into \bigoplus_{i=1}^k \calo_{X}(d_i+1)$. The exact sequence in display \eqref{long one} is then obtained via the snake lemma, since $\calv_\sigma:=\coker(\tilde \sigma)$, $\calq_\sigma:=\coker(\nabla\sigma)$ and $T'_X:=\coker(\varpi)$.
\end{proof}

\subsubsection{Tangent vector fields along hypersurfaces}

We look at the relationship between $\calt_\sigma$ and the tangent vector field to one of the hypersurfaces defining $\sigma$.

\begin{lemma}
We have:
\begin{equation} \label{intersection}
\calt_\sigma = \bigcap_{j=1}^k \calt_{f_j}.
\end{equation}
Further, for each $j\in\{1,\dots,k\}$, set  $Z_j=\sing\big(\VV(f_j)\big)$. Then there is an exact sequence:
\begin{equation}\label{sub tf sqc}
0 \to \calt_\sigma \to \calt_{f_j} \to \bigoplus_{i\in \{1,\ldots,k\} \setminus \{j\}} \calo_{\pn}(d_i) \to \calq_\sigma \to \calo_{Z_j}(d_j) \to 0.
\end{equation}
\end{lemma}

\begin{proof}
For any $j \in \{1,\ldots,n\}$, we have:
\[
\calt_{f_j} = \ker \Big(\nabla f_j \colon \calo_{\pn} ^{\oplus (n+1)} \to \calo_{\pn}(d_j) \Big).
\]

Therefore, since $\calt_\sigma$ is defined as kernel of the matrix obtained by stacking $\nabla(f_1), \ldots, \nabla(f_k)$, we get \eqref{intersection}.
Next, for any $j \in \{1,\ldots,k\}$, we have the following commutative diagram:
\begin{equation}\label{sub tf}
\begin{split} \xymatrix{ 
& & 0 \ar[d] & 0 \ar[d]  \\
0\ar[r] & \calt_\sigma\ar[r]\ar@{=}[d] & \calt_{f_j} \ar[d] \ar[r] & \bigoplus_{i\in \{1,\ldots,k\} \setminus \{j\}} \calo_{\pn}(d_i) \ar[d] \\
0\ar[r] & \calt_\sigma \ar[r] & \opn^{\oplus n+1} \ar[r]^-{\nabla\sigma}\ar[d]^{\nabla f_i} & \bigoplus_{i=1}^k \calo_{\pn}(d_i) \ar[d] \\
 &  & \cali_{Z_j}(d_j) \ar[r]\ar[d] & \opn(d_j) \ar[d] \\
& & 0 & 0
} \end{split}
\end{equation}
Since $Z_j$ is the Jacobian scheme of $f_j$, the completion of \eqref{sub tf} via the snake lemma leads to \eqref{sub tf sqc}.
\end{proof}

These observations will play an important role in the proof of Theorem \ref{high deg} below.


\subsection{Compressibility}

Next, we introduce the following definition.

\begin{definition}\label{comp}
We say that an algebraically independent sequence \(\sigma\) is \emph{compressible} if, up to a linear coordinate change, there is a variable that occurs in none of the forms \(f_1,\ldots,f_k\). An algebraically independent sequence that is not compressible is called \emph{incompressible}.
\end{definition}

\begin{lemma} \label{compressibility}
An algebraically independent sequence \(\sigma\) is compressible if and only if \(H^0(\calt_\sigma)\ne 0\).
\end{lemma}

\begin{proof}
The condition \(H^0(\calt_\sigma) \ne 0\) does not depend on the given choice of a system of coordinates. If none of the forms \(f_1,\ldots,f_k\) depends on a given variable, then all partial derivatives of \(f_1,\ldots,f_k\) with respect to this variable are zero. This means that $\nabla\sigma$ contains a column containing only 0; thus, the kernel sheaf \(\calt_\sigma\) contains a copy of \(\calo_\pn\).

Conversely, assume \(H^0(\calt_\sigma)\ne 0\).
For all \((i,j) \in \N^2\) with \(1 \le i \le k\) and \(0 \le j \le n\), we set:
\[
f_{i,j}= \frac{\partial f_i}{\partial x_j} \in R.
\]
Since the sheaf \(\calt_\sigma\) satisfies \(H^0(\calt_\sigma)\ne 0\), there is a non-zero vector \((b_0,\ldots,b_n) \in \kappa^{n+1}\) such that:
\[
b_0 f_{i,0} + \cdots + b_n f_{i,n} = 0, \qquad \mbox{for all \(1 \le i \le k\).}
\]
Then we define new coordinates \((x_0',\ldots,x_n')\) by choosing an invertible matrix \((a_{i,j})\) of size \(n+1\) with the condition \(a_{j,0}=b_j\) for all \(0 \le j \le n\) and putting:
\[
x_j = \sum_{\ell = 0}^n a_{j,\ell} x_\ell', \qquad \mbox{for all \(0 \le j \le n\).}
\]

Then, for all \mbox{for all \(1 \le i \le k\)}, we have:
\[
\frac{\partial f_i}{\partial x_0'} = \sum_{j=0}^n \frac{\partial x_j}{\partial x_0'}  f_{i,j} = \sum_{j=0}^n b_j  f_{i,j} = 0.
\]
Therefore, in the new coordinates \((x_0',\ldots,x_n')\), none of the forms appearing in \(\sigma\) depends on \(x_0'\).
\end{proof}

The \textit{compressibility} of $\sigma=\{f_1,\ldots,f_k\}$ is defined as the number of independent variables that can be removed from the polynomials \(f_i\), in a suitable coordinate system. In other words, \(\sigma\) has compressibility \(m\) if and only if \(h^0(\calt_\sigma)=m\); note that  \(0 \le m \le n-1\). We set \(\hat n:=n-m\); it indicates the minimal number of variables where $\sigma$ is defined.

\begin{lemma} \label{split compressible}
If $\sigma$ is a compressible algebraically independent sequence in $R$, then there is an incompressible sequence $\hat{\sigma}$ in $\kappa[x_0,\dots,x_{\hat{n}}]$ and a $\hat{n}$-dimensional linear space $L\subset\pn$ such that 
$\calt_\sigma=\opn^{\oplus(n-\hat{n})}\oplus\cale$ where $\cale|_L\simeq\calt_{\hat{\sigma}}$.  
\end{lemma}

\begin{proof}
Assume that $\sigma=(f_1,\dots,f_k)$ is compressible, so that \(m:=h^0(\calt_\sigma)>0\); set $\hat n:=n-h^0(\calt_\sigma)$. We get a monomorphism $\opn^{\oplus m}\into\calt_\sigma$ which we compose with $\calt_\sigma \into \opn^{\oplus (n+1)}$; one can then find an epimorphism $\opn^{\oplus (n+1)} \onto \opn^{\oplus m}$ such that the following composition
$$ \opn^{\oplus m}\into\calt_\sigma \into \opn^{\oplus (n+1)} \onto \opn^{\oplus m} $$
is the identity morphism; it follows that $\calt_\sigma=\opn^{\oplus(n-\hat{n})}\oplus\cale$, where the sheaf $\cale$ fits in the exact sequence
\begin{equation}\label{mu-map}
0 \to \cale \to \opn^{\oplus(\hat{n}+1)} \stackrel{\mu}{\to} \opn(d-1)^{\oplus k} ~~.    
\end{equation}
As we have seen in the proof of Lemma \eqref{compressibility}, there are new coordinates $(x_0':\dots:x_n')$ such that the variables $x_0',\dots,x_{m-1}'$ do not appear in the polynomials $f_i\in\sigma$.

This means that the first $m$ columns of the Jacobian matrix consist only of zeros and that the matrix $\mu$ in display \eqref{mu-map} is precisely the submatrix of trivial columns of $\nabla\sigma$. 

In addition, $\sigma$ can be regarded as a sequence in $\kappa[x_{m+1}',\dots,x'_{\hat{n}}]$, which we rename $\hat{\sigma}$. Setting $L=\VV(x_0',\cdots,x_{m-1}')$, we have that $\mu|_L=\nabla_{\hat{\sigma}}$, thus $\cale|_L\simeq\calt_{\hat{\sigma}}$.
\end{proof}

As an immediate consequence, we have:

\begin{corollary} \label{never}
The logarithmic tangent sheaf of a compressible algebraically independent sequence of length $k \le n-1$ is never slope-stable.
\end{corollary}

Recall that a coherent subsheaf \(\calf\) of a coherent sheaf \(\cale\) is \textit{saturated} if \(\cale/\calf\) is torsion-free. The following technical observation will be useful later on.

\begin{lemma} \label{sub of incompressible}
\label{double planes}
If $\sigma$ is an incompressible algebraically independent sequence, then every saturated subsheaf $\calk\subset\calt_\sigma$ satisfies $c_1(\calk)<0$.
In particular, when $k<n$ and $c_1(\calt_\sigma)=-1$, $\Ts$ is slope-stable if and only if $\sigma$ is incompressible.
\end{lemma}

\begin{proof}
Any saturated non-zero rank-$r$ subsheaf $\calk$ of $\calt_\sigma$ is also a saturated subsheaf of $\opn^{\oplus n+1}$, which is a slope polystable sheaf. It follows that $c_1(\calk)\le0$, and if $c_1(\calk)=0$, then  
\cite[Corollary 1.6.11]{HL} implies that $\calk=\opn^{\oplus r}$, so $\sigma$ is compressible.
This proves that $c_1(\calk) \le -1$ when $\sigma$ is incompressible.

For $k < n$, Corollary \ref{never} says that $\calt_\sigma$ is unstable as soon as $\sigma$ is compressible.
Conversely, let $\sigma$ be incompressible with $c_1(\calt_\sigma)=-1$ and consider a rank-$r$ subsheaf $\calk$ of $\calt_\sigma$, with 
$r \le n-k$. Then we just proved that $c_1(\calk) \le -1$ so:
$$ \frac{c_1(\calk)}{r} \le \frac{-1}{r} < \frac{-1}{n+1-k}.$$
This implies that $\calt_\sigma$ is slope-stable.
\end{proof}

\begin{example}\label{free not strong}
Here is an example of a free regular sequence that is not strongly free. Consider the following regular sequences in $R=\kappa[x_0,x_1,x_2,x_3]$
$$ \sigma:=(x_0x_1,g) ~~ {\rm and}~~ \sigma':=(x_0x_1,x_0^2x_1+g); $$ 
where $g$ is a polynomial of degree 3 depending only on $x_2$ and $x_3$; assume that $\partial_2g$ and $\partial_3g$ have no common factors, so that 
$\nabla g$ has no syzygies of degree $<2$. 

Clearly, $I_\sigma=I_{\sigma'}$. We argue that $\sigma$ is free, while $\sigma'$ is not. Indeed, their Jacobian matrices are given by:
$$ \nabla\sigma = \left( \begin{array}{cccc}
x_1 & x_0 & 0 & 0 \\
0 & 0 & \partial_2g & \partial_3g  
\end{array}\right) \quad {\rm and} \quad \nabla_{\sigma'} = \left( \begin{array}{cccc}
x_1 & x_0 & 0 & 0 \\
2x_0x_1 & x_0^2 & \partial_2g & \partial_3g  
\end{array}\right).
$$
Note that $\nabla\sigma$ has two independent syzygies, given by 
$$ \nu_1=(-x_0,x_1,0,0) ~~ {\rm and} ~~ \nu_2=(0,0,\partial_3g,-\partial_2g) $$ 
of degrees 1 and 2, respectively. Therefore, we have a monomorphism
$$ \nu \colon \op3(-1)\oplus\op3(-2)\into\calt_\sigma $$
whose cokernel, being a subsheaf of $\cali_{L}(1)\oplus\cali_{C}(2)$ with $L=\VV(x_0,x_1)$ and $C=\VV(\partial_2g,\partial_3g)$, must be torsion-free. It follows that $\nu$ must be an isomorphism, thus $\calt_\sigma\simeq \op3(-1)\oplus\op3(-2)$.

To see that $\calt_{\sigma'}$ does not split as a sum of line bundles, note that $\nu_2$ is also a syzygy for $\nabla_{\sigma'}$, thus $h^0(\calt_{\sigma'}(2))>0$. On the other hand, since $\nabla_{\sigma'}$ has no syzygy of degree $\le1$, we have that $h^0(\calt_{\sigma'}(1))=0$. 
In addition, the minors appearing in $\bigwedge^2\nabla_{\sigma'}$ have no common factor, thus $c_1(\calq_{\sigma'})=0$ and $c_1(\calt_{\sigma'})=-c_1(\calm_{\sigma'})=-3$.  Thus if 
$\calt_{\sigma'}=\op3(a)\oplus\op3(b)$ with $a\le b$, then $a+b=-3$, and $a,b\le-2$, which is impossible.

In fact, note that $\big(\Xi_{\sigma'}\big)_{\rm red}$ consists of the line $\VV(x_0,x_1)$ together with the following 0-dimensional schemes:
$$ \VV(x_0,\partial_2g,\partial_3g)  \qquad \mbox{and} \qquad \VV(x_1,\partial_2g,\partial_3g), $$
each of length equal to 4. Therefore, $(\Xi_{\sigma'})_{\rm red}$ contains at least two irreducible components of codimension 3; the second item of Lemma \ref{lf no codim 3} implies that $\calt_{\sigma'}$ is not locally free.
\end{example}

\begin{example}\label{strong}
We show that the regular sequence $\sigma=(x_0,x_3^2)$ in $R=\kappa[x_0,x_1,x_2,x_3]$ is a strongly free sequence consisting of polynomials of different degrees.
This example shows that, in general, $\det(\calt_\sigma)$ is not fixed and may change with the choice of generators for $I_\sigma$.

Any algebraically independent sequence $\sigma'$ such that $I_\sigma=I_{\sigma'}$ must be of the form $\sigma'=(\alpha x_0,x_0l+\beta x_3^2)$ for some linear form $l\in H^0(\op3(1))$ and $\alpha,\beta\in\kappa^*$. Setting $l=ax_0+bx_1+cx_2+dx_3$, the Jacobian matrix for $\sigma'$ is given by 
$$
\nabla_{\sigma'} = \left( \begin{array}{cccc}
\alpha & 0 & 0 & 0 \\
2ax_0 & bx_0 & cx_0 & dx_0+2\beta x_3  
\end{array}\right).
$$

If $c\ne0$, then 
$$ \nu_1=(0,-c,b,0) \qquad \mbox{and} \qquad \nu_2=(0,0,dx_0+2\beta x_3,-cx_0) $$ 
are independent syzygies of degrees 0 and 1, respectively. Following the argument in Example \ref{free not strong}, so we can conclude that $\calt_{\sigma'}=\op3\oplus\op3(-1)$.

When $c=0$ and $b\ne0$ then 
$$ \nu_1=(0,0,1,0)  ~~ {\rm and} ~~ \nu_2=(0,dx_0+2\beta x_3,0,-bx_0) $$
are independent syzygies of degree 0 and 1, respectively, so again we conclude that $\calt_{\sigma'}=\op3\oplus\op3(-1)$.

Finally, if $b=c=0$, then $\nu_1=(0,0,1,0)$ and $\nu_2=(0,1,0,0)$ are independent syzygies of degree 0, thus $\calt_{\sigma'}=\op3^{\oplus2}$ and $\sigma$ has compressibility $2$.

It is important to observe that even though $\calt_{\sigma'}$ always splits as a sum of line bundles, the splitting type is not always the same.
\end{example}

The following lemma settles the case of families of maximal compressibility.

\begin{lemma} \label{zero}
Let $\sigma$ be an algebraically independent sequence of homogeneous polynomials of degree $(d_1+1,\ldots, d_k+1)$, with $k \le n$. Then the following conditions are equivalent:
\begin{enumerate}[label=\roman*)]
\item \label{zero-i} $c_1(\calt_\sigma)=0$,
\item \label{zero-ii} $\calt_\sigma \simeq \calo_{\p n}^{\oplus(n+1-k)}$,
\item \label{zero-iii} $\sigma$ has a singular divisor of degree $d_1+\cdots+d_k$,
\item \label{zero-iv} there is a choice of linear coordinates of $\p n$ such that $\sigma$ depends only on $(x_0,\ldots,x_{k-1})$.
\end{enumerate}
\end{lemma}
\begin{proof}
We have \ref{zero-ii} $\Rightarrow$ \ref{zero-i}. Conversely, assuming \ref{zero-i}, 
again by  
\cite[Corollary 1.6.11]{HL} we get \ref{zero-ii}. Lemma \ref{split compressible} gives \ref{zero-ii} $\Leftrightarrow$ \ref{zero-iv}.
The degree $d$ of the singular divisor of $\sigma$ equals $c_1(\calq_\sigma)$ so:
\begin{equation}
\label{d vs c1}    c_1(\calt_\sigma) = d-\sum_{i=1}^k d_i.
\end{equation}
Hence, \ref{zero-i} implies \ref{zero-iii}. Conversely, assuming \ref{zero-iii} we get $c_1(\calt_\sigma) \ge 0$ and since  $c_1(\calt_\sigma) \le 0$, we get \ref{zero-i}.
\end{proof}

\subsection{Webs}\label{webs}

Fix integers \(d \ge 0\) and \(k \ge 1\) and let \(\sigma=(f_1,\ldots,f_k)\) be an algebraically independent sequence of forms of degree \(d+1\) in \(R=\kappa[x_0,\ldots,x_n]\); we call \(\sigma\) a \emph{\(k\)}-web in \(\pn\); a 2-web is usually called a \emph{pencil}. A \emph{\(k\)}-web is \emph{regular} if there is $(z_1,\ldots,z_k)  \in\kappa^k$ such that the hypersurface $\VV(\sum_{i=1}^k z_i f_i)$ is non-singular; a $k$-web that is not regular is called \emph{irregular}. Note that regular $k$-webs are incompressible.

In this section, we establish some basic properties of logarithmic tangent sheaves associated with $k$-webs, which will be useful later on. 

\subsubsection{Freeness of webs}

Here is the first fundamental fact.

\begin{lemma}\label{lem:k-web}
Let $\sigma$ be a $k$-web. If $\sigma$ is free, then it is strongly free. 
\end{lemma}
\begin{proof}
Let $\sigma'=(f_1',\dots,f_k')$ be another algebraically independent sequence such that $I_{\sigma'}=I_\sigma$;
one can check that there is a matrix $P\in \GL_k(\kappa)$ such that
$$ \left( \begin{array}{c} f_1' \\ \vdots \\ f_k' \end{array} \right) = P \left( \begin{array}{c} f_1 \\ \vdots \\ f_k \end{array} \right). $$
It follows that $\nabla_{\sigma'}=P\nabla_{\sigma}$, thus in fact $\calt_{\sigma'}\simeq\calt_{\sigma}$, from which the desired statement follows immediately.
\end{proof}

A particular case of the previous result leads to the simplest example of a strongly free algebraically independent sequence.

\begin{example} \label{linear}
Take a regular sequence $\sigma=(f_1,\dots,f_k)$ such that each $f_i$ is a linear polynomial; note that $\VV(\sigma)$ is a linear subspace of codimension $k$. The Jacobian matrix is a constant matrix of maximal rank, inducing a surjective morphism $\opn^{\oplus n+1}\to\opn^{\oplus k}$. It follows that $\calt_\sigma=\opn^{\oplus n+1-k}$, $\calm_\sigma=\opn^{\oplus k}$, and 
$\calq_\sigma=0$. 
\end{example}

\subsubsection{Webs versus algebraically independent sequences}

Let us point out how to associate a web to any algebraically independent sequence, keeping the logarithmic sheaf unchanged.
Let $\sigma=(f_1,\ldots,f_k)$ be an algebraically independent sequence, with $\deg(f_i)=d_i+1$, for some $d_1+1,\ldots,d_k+1 \in \N$. Let $e$ be the least common multiple of $d_1+1,\ldots,d_k+1$. 
For $i \in \{1,\ldots,k\}$, put $\ell_i=e/(d_i+1)$.
Set:
\[
\tau = (f_1^{\ell_1},\ldots,f_k^{\ell_k}).
\]
Note that $\tau$ is a web of degree $e$.
\begin{lemma}
We have:
\[
\calt_\tau = \calt_\sigma.
\]
\end{lemma}

\begin{proof}
For $i \in \{1,\ldots,k\}$, set $g_i=f_i^{\ell_i}$, so that $\tau = (g_1,\ldots,g_k)$.
By the chain rule, for each 
$i \in \{1,\ldots,k\}$ we have:
\[
\nabla(g_i) = \ell_i f_i^{\ell_i-1} \nabla(f_i).
\]
In other words, considering the morphism defined by the diagonal matrix:
\[
P=\mathrm{diag}(\ell_1f_1^{\ell_1-1}, \ldots,\ell_k f_k^{\ell_k-1}) \colon \bigoplus_{i=1}^k\calo_{\pn}(d_i) \to \calo_{\pn}(e)^{\oplus k},
\]
we get that:
\[
\nabla_\tau = P \circ \nabla\sigma.
\]
Since $\kappa$ is of characteristic zero, $P$ is injective, so $\calt_\tau = \ker(\nabla_\tau)= \ker(\nabla\sigma) = \calt_\sigma$.
\end{proof}


We complete this section with a characterization of the degeneracy locus of $k$-webs as the set of points that are singular for some hypersurface of the web.

\begin{lemma} \label{web deg locus}
Let $\sigma$ be a $k$-web. The reduced degeneracy locus $(\Xi_{\sigma})_{\rm red}$ of the Jacobian matrix $\nabla\sigma$ coincides with the union of the singular loci of the singular hypersurfaces contained in the web.
\end{lemma}
\begin{proof}
Set $\sigma=(f_1,\ldots,f_k)$. A point $x\in\pn$ belongs to $(\Xi_{\sigma})_{\rm red}$ if and only if the gradients of $f_i$ are linearly dependent, that is, there is $(z_1,\ldots,z_k) \in \kappa^{n+1} \setminus \{0\}$ such that:
$$ \nabla\left(\sum_{i=1}^k z_if_i\right)(x) = \sum_{i=1}^k z_i\nabla f_i(x) = 0. $$
But this is the same as saying that $x$ lies in the singular locus of $\VV(\sum_{i=1}^k z_if_i)$. 
\end{proof}


\section{Basic material on pencils of quadrics} \label{sec:stable-pencils}

In this section and up to \ref{locally free pencils}, \(\kappa\) is an algebraically closed of characteristic different from 2.
We consider pencils of quadrics, namely algebraically independent families $\sigma$ of two homogeneous polynomials $f_1,f_2$ of degree $2$ on $\p n$.
We develop here the basic material needed for Sections up to \ref{locally free pencils} as well, including a short mention of Segre symbols.

The classification of regular pencils of quadrics was given by Weierstrass in \cite{W} and interpreted geometrically by Segre \cite{Segi,Segii}. Weierstrass completed the work of Weierstrass by studying irregular pencils, see \cite{K68}. Segre again gave a geometric interpretation in \cite{Segiii}.
For a modern treatment, the reader may consult for instance \cite[Chapter XII]{G}, \cite[Section 10.6]{L},
\cite[§XIII.10]{HP} or the more recent paper \cite{FMS}.

In the whole section, we focus on algebraically independent sequences $\sigma=(f_1,f_2)$ such that $\deg(f_1)=\deg(f_2)=2$, to which we can associate the pencil of quadrics $Q_\lambda:=V\big(z_1 f_1 + z_2 f_2\big)$, where $\lambda=[z_1:z_2]\in\p1$.
Given a quadric hypersurface \(Q\) in \(\pn\), we denote by \(\rk(Q)\) the rank of the Hessian matrix of an equation of \(Q\), that is,  the rank of a quadratic form associated with \(Q\). We set \(\cork(Q)=n+1-\rk(Q)\). When non-empty, the singular locus of \(Q\) is a linear subspace of \(\pn\) of dimension \(\cork(Q)-1\). The quadric \(Q\) is a double plane if and only if \(\rk(Q)=1\).

\subsection{The classification of pencils of quadrics}

Considering the polarization (or Hessian) matrix of the quadrics in the pencil \(\sigma\) we obtain a \textit{pencil of symmetric matrices} of size \(n+1\):
\[
\rho_\sigma : \calo_{\p 1}(-1)^{\oplus (n+1)} \to \calo_{\p 1}^{\oplus (n+1)}.
\]
Conversely, from a pencil of symmetric matrices of size $n+1$, we recover a pencil of quadrics on $\p n$.
\begin{definition}
Given a pencil of quadrics $\sigma$ and $\lambda \in \p 1$ we set
$r(\lambda)=\cork(Q_\lambda)$. We put
\[r_1 := \min\big\{r(\lambda) \mid \lambda \in \p 1\big\}.
\]
We call \(r_1\) the \textit{generic corank} of \(\sigma\). 
We say that $\sigma$ (or $\rho_\sigma$) is \textit{regular} if $r_1=n+1$ and that $\sigma$ is \textit{irregular} otherwise. We say that $\sigma$ is \textit{completely irregular} if $r_1=r_0$.
\end{definition}

\subsubsection{The classification of pencils of symmetric matrices}
Let $\sigma$ be a pencil of quadrics in $\p n$ and put $m$ for the compressibility of $\sigma$.
In view of the classification of pencils of quadrics (cf. \cite[Chapter XII]{G}), there exist integers $p \le n$, $1 \le c_1 \le \cdots \le c_{r_1-m}$, a regular pencil $\bar \rho_\sigma$ of size $p+1$ and a suitable system of linear coordinates of $\p n$ such that $\rho_\sigma$ can be written in block diagonal form
\begin{equation} \label{bloc-decomposition}
\rho_\sigma = \diag(\zeta(c_1),\cdots,\zeta(c_{r_1-m}),\bar \rho_\sigma,0_m),  
\end{equation}
where $0_m$ is the zero matrix of size $m$ and, for any integer $c \ge 1$, 
\[
\zeta(c) = \left(\begin{array}{c|c}
    0 & \tau(c) \\
    \hline
    \tau(c)^\t & 0
\end{array}\right), 
\]
where 
$\tau(c) : \calo_{\p 1}(-1)^{\oplus c} \to \calo_{\p 1}^{\oplus c+1}$ reads as follows
\[\tau(c)=\begin{pmatrix}
    z_1 & 0 & \cdots & 0 \\
    z_2 & z_1 & 0 & \vdots \\    
    0 & z_2 & \ddots & 0\\
    \vdots & & \ddots & z_1 \\
    0 & \cdots & 0 & z_2
\end{pmatrix}.
\]

Here, we allow $p=-1$, in which case $\bar \rho_\sigma$ does not occur. For any integer \(c \in \N\) we let \(V_c\) be the irreducible representation \(\SL_2(\kappa)\) of weight \(c\), so \(V_0\) is the trivial representation of rank 1, \(V_1\) is the standard representation of rank 2, while \(V_c=S^c V_1\) has rank \(c+1\). By convention we set \(V_{-1}=0\). 
Then $\tau(c)$ and $\zeta(c)$ can be rewritten is  \(\SL_2(\kappa)\)-equivariant form:
\begin{gather} \label{SL2_zeta}
\nonumber 0 \to V_{c-1} \otimes \calo_{\p 1}(-1) \xrightarrow{\tau(c)} V_c \otimes 
\calo_{\p 1} \to \calo_{\p 1}(c) \to 0,\\
0 \to \calo_{\p 1}(-c-1) \to (V_{c-1} \oplus V_c) \otimes \calo_{\p 1}(-1) \xrightarrow{\zeta(c)} (V_{c-1} \oplus V_c) \otimes 
\calo_{\p 1}, \to \calo_{\p 1}(c) \to 0.
\end{gather}
and one has:
\[
\im(\zeta(c)) \simeq \calo_{\p 1}(-1)^{\oplus c} \oplus 
\calo_{\p 1}^{\oplus c}.
\]

\begin{definition}
The pencil of quadrics associated with \(\bar \rho_\sigma\) is called the \textit{regular part of \(\sigma\)}.
We put 
\[
v = \sum_{i=1}^{r_1-m} c_i, \qquad 
u = v + p+1.
\]
Recall that here $p+1$ is the size of $\bar \rho_\sigma$. We call \((u,v)\) the \textit{splitting type} of $\sigma$. Of course $u \ge v$. Also:
\[
\im(\rho_\sigma) \simeq \calo_{\p 1}(-1)^{\oplus u} \oplus \calo_{\p 1}^{\oplus v}.
\]
We call $(c_1,\ldots,c_{r_1-m})$ the \textit{degree vector} of $\sigma$. We observe that $\bar \rho_\sigma$ has size $u-v$. Looking at the size of the block decomposition \eqref{bloc-decomposition}, successively $2c_1+1,\ldots,2c_{r_1-m}+1,p+1,m$, we get:
\begin{equation} \label{n+1=}
n+1=2\sum_{i=1}^{r_1-m}c_i+r_1+p+1=u+v+r_1.
\end{equation}

\end{definition}

Note that the regular part of a pencil of quadrics may be empty: this happens for \(u=v\). 
We will call these pencils \emph{completely irregular} and treat them in detail a bit further on. 
Note also that the regular part of an irregular pencil may fail to be a pencil, when $\sigma$ is incompressible this happens for \(u=v+1\). 
By convention, a pencil of symmetric matrices of size 0 (the empty pencil) and a non-zero pencil of symmetric matrices of size 1 are regular.

\begin{lemma} \label{3r1}
An incompressible pencil of quadrics satisfies \(r_1 \le \frac{n+1}3\).
\end{lemma}

\begin{proof}
An incompressible pencil satisfies $m=0$ and, since \(c_i \ge 1 \) for $i \in \{1,\ldots,r_1\}$ we get
\(
v = \sum_{i=1}^{r_1} c_i \ge r_1.
\)
Therefore we have \(n+1=u+v + r_1\ge 2v+r_1\ge 3r_1.\)
\end{proof}

\subsubsection{Segre symbols} \label{segre symbols}

The Segre symbol is the key invariant of regular pencils. Indeed, the content of the Segre-Weierstrass theorem is that the set of singular quadrics of a regular pencil together with its Segre symbol classifies the regular pencil up to homography of \(\p1\) and \(\pn\). Again we refer to \cite[§XIII.10]{HP} or \cite{FMS}.
We give a short account of this theory here. 

Let \(\sigma\) be an incompressible pencil of quadrics. Following the notation introduced above, let \((u,v)\) be the splitting type of \(\sigma\), so that its generic corank \(r_1\) satisfies \(n+1=u+v+r_1\). 
Set \(\calc_\sigma := \coker(\rho_\sigma)\) and write
the long exact sequence:
\begin{equation} \label{K&C}
0 \to \ker(\rho_\sigma) \to \calo_{\p 1}(-1)^{\oplus (n+1)} \to 
\calo_{\p 1}^{\oplus (n+1)} \to \calc_\sigma \to 0.    
\end{equation}

The cokernel sheaf \(\calc_\sigma\) decomposes as \(\calc_\sigma \simeq \calc_\tf \oplus \calc_\t\), where \(\calc_\t\)  is its torsion part and \(\calc_\tf\) is its torsion-free part. One has:
\begin{equation} \label{reg part sqc}
\calc_\sigma \simeq \coker(\bar \rho_{\sigma}), \qquad 
h^0(\coker(\bar \rho_{\sigma}))=u-v.
\end{equation}

Let $\{\lambda_1,\ldots,\lambda_\ell\} \subset \p1$ be the support of the torsion sheaf \(\calc_\t\). We have:
\[
\calc_\t \simeq \bigoplus_{j=1}^\ell (\calc_\t)_{\lambda_j},
\]
where $(\calc_\t)_{\lambda_j}$ is the localization at \(\lambda_j\) of \(\calc_\t\), which in turn can be written in the following way:
\[
(\calc_\t)_{\lambda_j} \simeq 
\bigoplus_{i=1}^{s_j} \calo_{\lambda_j^{(a_{j,i})}}^{\oplus p_{j,i}}.
\] 

Here, we denoted by  \(\lambda_j^{(a_{j,i})}\) the subscheme  defined by the ideal \((g_j^{a_{j,i}})\) where \(g_j\) is a linear form vanishing at \(\lambda_j\) for each \(j \in \{1,\ldots,\ell\}\).
Up to choosing linear coordinates in $\p 1$ so that none of the subschemes $\lambda_1,\ldots,\lambda_\ell$ is at $\infty = (1:0)$, we may write $\lambda_j \in \kappa$ and $g_j=z_1-\lambda_j z_2$ for $j\in \{1,\ldots,\ell\}$. 

 For all $j \in \{1,\ldots,\ell\}$, the block of $\rho_\sigma$ presenting $(\calc_\t)_{\lambda_j}$ is itself in block diagonal form, with blocks of sizes \(a_{j,1},\ldots,a_{j,s_j}\), repeated 
\(p_{j,1},\ldots,p_{j,s_j}\) times, where the block of size \(a \in \N^*\) takes the form 
\[
\begin{pmatrix}
0 & 0 & \cdots  & 0 & z_2 & z_1-\lambda_j z_2 \\ 
0 & \cdots &0 & z_2 & z_1 -\lambda_j z_2& 0\\ 
\vdots & \vdots & \iddots & \iddots & 0 & \vdots  \\
0 & z_2 & z_1 -\lambda_j z_2& 0 & \iddots  & \iddots  \\ 
z_2 & z_1 -\lambda_j z_2& 0 & \iddots  & \iddots  \\ 
z_1 -\lambda_j z_2& 0 & \cdots & \iddots &  \\ 
\end{pmatrix}.
\]

The integers $a_{j,i}$ are arranged in the \emph{Segre symbol}. We write, for each \(j \in \{1,\ldots,\ell\}\):
\[
\Sigma_j=(\underbrace{a_{j,1},\ldots,a_{j,1}}_{p_{j,1}},\ldots,\underbrace{a_{j,s_j},\ldots,a_{j,s_j}}_{p_{j,s_j}}), \qquad \mbox{with \(a_{j,1} > \cdots > a_{j,s_j}\)}.
\]

The Segre symbol $\Sigma$ for a pencil of quadrics $\sigma$ is defined to be the multi-set \([\Sigma_1,\ldots,\Sigma_\ell]\). We will use exponential notation for parenthesized repeated entries; for instance, the Segre symbol \([(1,1,1),(3,3,1),2,2]\) in exponential notation reads \([1^3,(3^2,1),2,2]\).

Note that, as we are dealing with potentially irregular pencils \(\sigma\), we always refer to the Segre pencil of the regular part \(\bar \sigma\) of \(\sigma\). In case \(\sigma\) is completely irregular, its Segre symbol is \(\emptyset\) by convention. Note that, in contrast to the behavior for regular pencils, the Segre symbol of an irregular pencil may be of the form \([1^p]\), for some integer \(p\).

\bigskip

We introduce a filtration of the torsion part $\calc_\t$ of the cokernel of $\rho_\sigma$ for future use.
To this purpose, in order to simplify the notation we work at a single point $\lambda$ so we tacitly assume $\ell=1$ and we suppress $j$ from the indices.
The cokernel of this matrix is the structure sheaf of \(\calo_{\lambda^{(a)}}\), where \(\lambda^{(a)}\) is the \(a\)-tuple point of \(\p1\) which we may take to be defined by \((z_1^a)\).
Therefore, we have:
\[
\calc_\t \simeq \bigoplus_{i=1}^s \calo_{\lambda^{(a_i)}}^{\oplus p_i}.
\]

For each \(i \in \{1,\ldots,s\}\), we consider the injection \(\lambda^{(a_i)} \subset \lambda^{(a_s)}\). Concatenating the surjections \(\calo_{\lambda^{(a_s)}} \to \calo_{\lambda^{(a_i)}}\) we get an epimorphism:
\[
\calc_\t \twoheadrightarrow \bigoplus_{i=1}^s \calo_{\lambda^{(a_s)}}^{\oplus p_i}.
\]
For \(k \in \{1,\ldots,s\}\), put \(q_k = \sum_{i=1}^k p_i\).
From the above epimorphism, we get the exact sequence:
\[
0 \to \bigoplus_{i=1}^{s-1} \calo_{\lambda^{(a_i-a_s)}}^{\oplus p_i} \to \calc_\t \to \calo_{\lambda^{(a_s)}}^{\oplus q_s} \to 0.
\]
Iterating this procedure we obtain a natural filtration:
\[
0 = \cald^{(0)} \subset \cald^{(1)} \subset \cdots \subset \cald^{(s)} = \calc_\t,
\]
where, for all \(k \in \{1,\ldots,s\}\),  we have (with the convention \(a_{s+1}=0\)):
\[
\cald^{(k)} = \bigoplus_{i=1}^k \calo_{\lambda^{(a_i-a_{k+1})}}^{\oplus p_i}, \qquad \calc^{(k)} := \cald^{(k)}/\cald^{(k-1)} = \calo_{\lambda^{(a_k-a_{k+1})}}^{\oplus q_k}.
\]

\subsection{The Jacobian scheme of a pencil of quadrics}

Here we sketch the description of the Jacobian scheme of a pencil of quadrics, with special attention to the case of irregular pencils, by giving an outline of a construction, (essentially due to C. Segre, \cite{Segiii}) of a rational normal scroll swept by vertices of quadric cones in the pencil.
We give here essentially a set-theoretic description, as the scheme-theoretic description one only be needed later on.
We start with the following simple observation.

\begin{lemma} \label{they are disjoint}
Let \(\sigma\) be an incompressible pencil of quadrics and let \(\lambda,\mu \in \p1\) be distinct points such that \(Q_\lambda\) and \(Q_\mu\) are singular. Then the singular loci of \(Q_\lambda\) and \(Q_\mu\) are disjoint linear spaces of dimension \(\cork(Q_\lambda)-1\) and \(\cork(Q_\mu)-1\).
\end{lemma}

\begin{proof}
The singular loci of \(Q_\lambda\) and \(Q_\mu\) are defined by linear equations and the corank of \(Q_\lambda\) and \(Q_\mu\) is precisely the number of independent equations.
In addition, these two linear spaces are disjoint, as the coordinates of a point of \(\pn\) lying in the singular locus of two distinct quadrics \(f_1,f_2\) of the pencil would annihilate the derivatives of \(f_1\) and \(f_2\), so such derivatives would fail to span \(H^0(\calo_\pn(1))\). Thus we could choose coordinates so that one of the variables \(x_0,\ldots,x_n\) occurs neither in \(f_1\) nor in \(f_2\). However, this is excluded by the hypothesis that \(\sigma\) is incompressible.
\end{proof}

If the divisorial part of the Jacobian scheme of a pencil of quadrics $\sigma$ is non-empty then such part may have degree $1$ or $2$.
The following lemma gives an account of this extremal case.

\begin{lemma} \label{c1=0,1}
Let $\sigma$ be a pencil of quadrics with a singular divisor of degree $l\in\{1,2\}$. Then 
\begin{enumerate}[label = \roman*)]
\item \label{case 0} $l=2$ if and only if, in suitable coordinates of $\p n$, we have either $\sigma=(x_0^2,x_1^2)$ or $\sigma=(x_0^2,x_0x_1)$.
\item \label{case 1} $l=1$ if and only if $\sigma$ contains one and only one  double hyperplane or there is a choice of linear coordinates of $\p n$, such that 
$\sigma=(x_0x_1,x_0x_2)$.
\end{enumerate}
\end{lemma}
\begin{proof}
Let us prove \ref{case 0}. According to Lemma \ref{zero}, we have $c_1(\calt_\sigma)=0$ if and only $l=2$, if and only, in suitable coordinates, $\sigma$ depends only on $x_0$ and $x_1$. In such coordinates, the Segre symbols is either $[1,1]$ or $[2]$ and theses cases correspond respectively to $(x_0^2,x_1^2)$ and $(x_0^2,x_0x_1)$.

Let us move to \ref{case 1}. According to \eqref{d vs c1}, $l=1$ if and only if $\Xi_\sigma$ is a simple hyperplane. If $\sigma$ contains no double hyperplane, then $r_0 \le n-1$ and $\Xi_\sigma$ is the union of linear spaces of dimension at most $n-2$, parametrized by singular quadrics in the pencil. Hence there must be infinitely many such spaces (i.e. $\sigma$ is irregular) and they cannot be all of dimension $\le n-2$, i.e. $r_1 \ge n-1$. Hence $r_0=r_1=n-1$ and $\sigma$ is completely irregular, with compressibility $m=n-1$. Hence, by \eqref{n+1=}, $u=v=c_1=1$ and $\sigma$ takes the form $(x_0x_1,x_0x_2)$ in suitable coordinates.
\end{proof}

\subsubsection{The Jacobian scheme of an irregular pencil}
\label{Jacobian of irregular}

Assume now that \(\sigma\) is an incompressible irregular pencil and seek a set-theoretic description of degeneracy scheme \(\Xi_\sigma\).
The torsion-free part \(\calc_\tf\) defines a projective bundle \(Y=\p{}(\calc_\tf)\) and, since the vector bundle \(\calc_\tf\) is very ample, \(Y=\p{}(\calc_\tf)\) embeds via the linear system of the tautological relatively ample divisor \(h\) as a rational normal scroll of degree \(v\), spanning a linear space \(L \subset \pn\) of dimension \(n-u\).

\begin{lemma}\label{Jacobian union}
Let \(\sigma\) be an incompressible pencil of quadrics. Then \(\Xi_\sigma\) satisfies:
\begin{equation} 
(\Xi_\sigma)_{\redu} = ~ Y ~ \cup \bigcup_{\lambda \in \supp(\calc_\t)} {\p{r(\lambda)-1}},
\end{equation}
where the linear subspaces \(\{\p{r(\lambda)-1} \mid \lambda \in \supp(\calc_\t)\}\) are disjoint.
In particular:
\[
\dim(\Xi_\sigma) = \max(r_0-1,r_1), \qquad r_1=n+1-u-v.
\]
\end{lemma}

We develop the proof of this lemma, essentially known to Segre, for the sake of completeness; we take advantage to introduce some notation.

\begin{proof}
We look at the projectivization of the vector bundle \(\calc_\tf\)
and of the coherent sheaves \(\calc_\t\), \(\calc_\sigma\). The epimorphisms \(\calo_\p1^{\oplus(n+1)} \onto \calc_\sigma\),  \(\calo_\p1^{\oplus(n+1)} \onto \calc_\t\),  and \(\calo_\p1^{\oplus(n+1)} \onto \calc_\tf\) induce embeddings \(\p{}(\calc_\sigma) \hookrightarrow \p 1 \times \p n\), \(\p{}(\calc_\t) \hookrightarrow \p 1 \times \p n\) and \(Y \hookrightarrow \p 1 \times \p n\). Similarly, the epimorphism $\opn^{\oplus2}\onto\calq(-1)$ induces an embedding \(\p{}(\calq_\sigma(-1)) \hookrightarrow \p 1 \times \p n\).
The two subschemes \(\p{}(\calc_\sigma) \) and  \(\p{}(\calq_\sigma(-1))\) of \(\p 1 \times \p n\) are defined by the same bihomogeneous equations. Indeed, denoting by \(\lambda=(z_1:z_2)\) and \(x=(x_0:\ldots:x_n)\) the points of \(\p 1\) and \(\pn\) and recalling the notation \(f_{i,j}=\frac{\partial f_i(x)}{\partial x_j}\), we have:
\[
\p{}(\calq_\sigma(-1)) = \p{}(\calc_\sigma) = \{((x,\lambda) \in \p 1 \times \p n \mid f_{1,j}z_1+f_{2,j}z_2=0, \forall j = 0,\ldots,n\},
\]
which in turn gives a Kozsul complex (in the obvious notation):
\begin{equation} \label{koszul P1}
 \cdots \to \calo_{\p 1 \times \pn}(-1,-1)^{\oplus(n+1)} \to \calo_{\p 1 \times \pn} \to \calo_{\p{}(\calc_\sigma)}\to 0.
\end{equation}

We get thus a correspondence:
\begin{equation} \label{correspondence}
\begin{split} \xymatrix@-2ex{
&\p{}(\calc_\sigma) \ar_-{\varphi}[dl] \ar^-{\psi}[dr]\\
\p 1&& \pn} \end{split}
\end{equation}
where the map \(\varphi : \p{}(\calc_\sigma) \to \p 1\) is generically a \(\p {r_1-1}\)-bundle and \(\psi : \p {}(\calc_\sigma) \to  \Xi_\sigma \subset \pn\) is an isomorphism at the points where \(\calq_\sigma\) has rank 1. 

At each point \(\lambda\) of the support of the torsion part \(\calc_\t\) we have a skyscraper sheaf supported at \(\lambda\), whose rank we denote by \(r(\lambda)\). The surjection \(\calo_{\p 1}^{\oplus(n+1)} \to \calo_\lambda^{\oplus r(\lambda)}\) induces an embedding \(\p{r(\lambda)-1} \subset \Xi_\sigma \subset \pn\). We noticed in Lemma \ref{they are disjoint} that the linear spaces appearing as singular loci of distinct points of \(\supp(\calc_\t)\) are disjoint. This achieves the proof.
\end{proof}

\subsubsection{Completely irregular pencils}

Recall that the incompressible pencil \(\sigma\) is \textit{completely irregular} if has no regular part, which is to say, if \(u=v\). This is equivalent to the condition \(\calc_\t=0\), which in turn is tantamount to \(r_0=r_1\).

We take a closer look at the Jacobian scheme in this case.
Denote by \(F\) the divisor class of a fibre of \(Y \to \p 1\), so \(\calo_Y(F) \simeq \varphi^*(\calo_{\p 1}(1))\), in the notation of the diagram in display \eqref{correspondence}; write also \(H=c_1(\psi^*(\calo_{\pn}(1)))\). Note that \(\psi_*(\calo_Y(F)) \simeq \calq_\sigma\) and that the Koszul complex \eqref{koszul P1} is exact at \(\calo_{\p 1 \times \pn}\). Tensoring it with \(\calo_{\p 1 \times \pn}(1,0)\)
and applying \(\psi_*\), we get an exact sequence:
\[
\calo_{\pn}^{\oplus(n+1)}(-1) \to \calo_\pn^{\oplus 2 } \to \calo_Y(F) \to  0.    
\]

The rightmost morphism above agrees with the Jacobian matrix, so we have an exact sequence:
\begin{equation} \label{J4}
0 \to \calt_\sigma \to \calo_\pn^{\oplus(n+1)} \stackrel{\nabla\sigma}{\to} \calo_\pn(1)^{\oplus 2 } \to \calo_Y(H+F) \to  0.
\end{equation}

\begin{example}
Let $n=5$. Regular pencils give \(\calc_\sigma = \calc_\t\), a finite-length scheme with \(h^0(\calc_\sigma)=6\).  For irregular pencils, we have the following possibilities.
\[
\begin{small}
\begin{array}{c|c|c|c|c|c}
r_1 & \mbox{\((u,v)\) } & h^0(\calc_\t) & \calc_\tf & \mbox{Compressibility} & \mbox{Completely irregular} \\
\hline
\hline
1 & (5,0) & 5 & \calo_\p1 & \mbox{1} &\mbox{no}  \\
1 & (4,1) & 3 & \calo_\p1(1) &\mbox{0} &\mbox{no}\\
1 & (3,2) & 1 & \calo_\p1(2) & \mbox{0} &\mbox{no}\\
\hline
2 & (4,0) & 4 &  \calo_\p1^{\oplus 2} & \mbox{2} &\mbox{no}\\
2 & (3,1) & 2 & \calo_\p1 \oplus \calo_\p1(1) & \mbox{1} &\mbox{no}\\
2 & (2,2) & 0 &  \calo_\p1(1)^{\oplus 2} & \mbox{0} &\mbox{yes}\\
2 & (2,2) & 0 & \calo_\p1 \oplus \calo_\p1(2) & \mbox{1} &\mbox{yes}\\
\hline
3 & (3,0) & 3 &  \calo_\p1^{\oplus 3} & \mbox{3} &\mbox{no}\\
3 & (2,1) & 1 &  \calo_\p1^{\oplus 2} \oplus \calo_\p1(1) & \mbox{2} &\mbox{no}\\
\hline
4 & (2,0) & 2 &  \calo_\p1^{\oplus 4} & \mbox{4} &\mbox{no}\\
4 & (1,1) & 0 &  \calo_\p1^{\oplus 3} \oplus \calo_\p1(1)& \mbox{3} &\mbox{yes}
\end{array}
\end{small}
\]
Observe that there is only one incompressible, completely irregular pencil of quadrics in $\p5$. 

However, the lowest dimension in which incompressible, completely irregular pencils occur is $n=3$, the only such pencil being given by $(x_1x_2,x_0x_2)$. Note that this pencil is not a complete intersection.
The next case is for $n=4$, given by $(x_1x_3+x_2x_4,x_0x_3+x_1x_4)$.
\end{example}

\begin{lemma} \label{it is equivariant}
Let \(\sigma\) be an incompressible completely irregular pencil. Then there is an action of \(\SL_2(\kappa)\) on \(\pn\) for which the exact sequence in display \eqref{J4} is equivariant.
\end{lemma}

\begin{proof}
Let $c_1 \le \cdots \le c_{r_1}$ be the degree vector of $\sigma$. Let the group \(\SL_2(\kappa)\) act by homographies on \(\p 1\) regarded as the base of the pencil. 
Since the regular part $\bar \rho_\sigma$ vanishes, $\rho$ is the bloc-diagonal matrix 
\[\diag(\zeta(c_1),\ldots, \zeta(c_{r_1})).
\] 
Since each block of $\rho_\sigma$ is \(\SL_2(\kappa)\)-equivariant, so are \(\rho_\sigma\) and $\sigma$ itself. Hence  the Jacobian matrix of \(\sigma\) is also \(\SL_2(\kappa)\)-equivariant and this induces an \(\SL_2(\kappa)\)-action on its kernel and cokernel sheaves. 
\end{proof}


\section{Stability for pencils of quadrics} 

\label{subsec:nodouble}

In this section \(\kappa\) is algebraically closed of characteristic different from 2.
Our goal is to present the proof of Theorem \ref{mthm:stable}.
The first case is equivalent to $c_1(\calt_\sigma)=-1$ and, when $n \ge 3$, stability of $\calt_\sigma$ is equivalent to incompressibility by Lemma \ref{double planes}.
The second and third cases are equivalent to $c_1(\calt_\sigma)=0$ and in turn to the fact that $\calt_\sigma \simeq \calo_{\p n}^{\oplus(n-1)}$ by Lemma \ref{zero}.

This discussion proves the first two items \ref{thm c1=0} and \ref{thm c1=-1} of Theorem \ref{mthm:stable}.
The next step is to analyze pencils of quadrics containing no singular hyperplanes. Lemma \ref{compressibility} implies that compressible pencils of quadrics having no singular hyperplane have slope-unstable logarithmic sheaves, thus proving the third item of Theorem \ref{mthm:stable}.

We can finally address incompressible pencils of quadrics \(\sigma=(f_1,f_2)\) containing no singular hyperplanes. The result depends on the maximal corank of the quadrics \(Q_\lambda\), where for each \(\lambda=(z_1:z_2) \in \p1\) we write \(Q_\lambda = \VV(z_1 f_1 + z_2 f_2)\). 
To be precise, we prove the following result.

\begin{theorem} \label{stability-pencils}
Let \(n \ge 3\) and let \(\sigma\) be an incompressible pencil of quadrics in \(\pn\) which contains no singular hyperplane. Put \(r_0=\max(\cork(Q_\lambda) \mid \lambda \in \p 1)\). 
\begin{enumerate}[label=\roman*)]
    \item \label{<} If \(2r_0 < n+1\), then \(\calt_\sigma\) is slope-stable.
    \item \label{=} If \(2r_0 = n+1\), then \(\calt_\sigma\) is strictly slope-semistable.
    \item \label{>} If \(2r_0 > n+1\), then \(\calt_\sigma\) is slope-unstable.
\end{enumerate}
\end{theorem}

Since \(\sigma\) contains no singular hyperplane, we have \(\codim(\calq_\sigma) \ge 2\) so the slope of \(\calt_\sigma\) is:
\[
\mu(\calt_\sigma)=\frac{2}{1-n}.
\]
The proof of Theorem \ref{stability-pencils} will be divided into three parts. We start by establishing items \ref{=} and \ref{>} in Section \ref{koszul subs}. For the proof of item \ref{<}, we first consider regular pencils in Section \ref{stability regular}, leaving the case of irregular pencils for Section \ref{stability irregular}.

\subsection{Koszul subsheaves}\label{koszul subs}

As a preliminary step towards the proof of items \ref{=} and \ref{>}, of Theorem \ref{stability-pencils} we study the sheaves, that we call Koszul subsheaves, appearing in the Koszul complex of a linear space. We set \(\calr_M\) for the first Koszul syzygy sheaf of a linear subspace \(M\subset \pn\) of codimension \(r>1\), namely the sheaf fitting into the following short exact sequence:
\[
0 \to \calr_M \to \calo_\pn^{\oplus r }(-1) \to \cali_M \to 0.
\]

\begin{lemma}\label{calr_M}
The sheaf \(\calr_M\) is slope-stable.
\end{lemma}
\begin{proof}
Note that $\mu(\calr_M(1))=-1/(r-1)$. Moreover, by the argument in the proof of Lemma \ref{sub of incompressible}, any nonzero saturated subsheaf $\calf\into\calr_M(1)$ must have $c_1(\calf)\le-1$, since $h^0(\calr_M(1))=0$. If $\calf$ destabilizes $\calr_M(1)$, then providing a contradiction, as:
$$ \frac{c_1(\calf)}{\rk(\calf)} > \frac{-1}{r-1}, \qquad
\mbox{so} \qquad c_1(\calf) > - \frac{\rk(\calf)}{r-1} > -1. $$
\end{proof}

\subsubsection{Koszul subsheaves from singular quadrics}
A linear subspace of \(\Xi_\sigma\) is called \textit{maximal} if it is not strictly contained in another linear subspace of \(\Xi_\sigma\). The following technical lemma is quite useful.

\begin{lemma} \label{ii and iii}
Let \(\sigma\) be an incompressible pencil of quadrics, \(\varrho \ge 3\) be an integer and \(L \subset \Xi_\sigma\) be a maximal linear subspace of projective dimension \(\varrho-1\). Then there is a linear subspace \(M \subset \pn\) of codimension \(\varrho \) and a subscheme \(W \subset M\) such that \(\calt_\sigma\) fits into
\begin{equation} \label{4terms}
0 \to \calr_M(1) \to \calt_\sigma \to \calr_L(1) \to \cali_{W/M}(1) \to 0.
\end{equation}
\end{lemma}

\begin{proof}
Since the pencil \(\sigma\) is incompressible, the linear forms appearing in the Jacobian matrix of \(\sigma\) span \(H^0(\calo_\pn(1))\), hence the sheaf \(\calq_\sigma\) has rank 1 at each point of its support, in particular, this happens at each point of \(L\), so \(\calq_\sigma|_L\) is a line bundle on \(L\), namely there is \(e \in \Z\) such that \(\calq_\sigma|_L \simeq \calo_L(e)\). Since \(\calo_L(1)^{\oplus 2}\) surjects onto \(\calo_L(e)\) and \(\varrho \ge 3\), we conclude that \(e=1\).
 
The surjection \(\calq_\sigma \to \calo_L(1)\) allows to write the following commutative exact diagram:
\[
\xymatrix@-2ex{
0 \ar[r] & \calm_\sigma \ar@{.>}[d] \ar[r] & \calo_\pn(1)^{\oplus 2} \ar[r] \ar[d] & \calq_\sigma \ar[d] \ar[r] & 0 \\ 
0 \ar[r] & \cali_L(1) \ar[r] & \calo_\pn(1) \ar[d] \ar[r] & \calo_L(1) \ar[r] \ar[d] & 0 \\
 & & 0 & 0}
\]

Put \(\calf\) and \(\calg\) for the kernel and cokernel of the induced morphism \(\calm_\sigma \to \cali_L(1)\), respectively; in addition, let $\calq'$ denote the kernel of the epimorphism $\calq_\sigma\onto\calo_L(1)$. The snake lemma provides the following exact sequence
$$ 0 \to \calf \to \opn(1) \to \calq' \to \calg \to 0, $$
thus there is a subscheme \(W \subset \pn\) such that \(\calf \simeq \cali_W(1)\) and 
$$ \supp(\calg)\subset\supp(\calq')\subset\supp(\calq_\sigma) = (\Xi_\sigma)_{\rm red}.$$

Since \(\calm_\sigma\) is the image of the the Jacobian matrix \(\calo_\pn^{\oplus(n+1)} \to \calo_\pn(1)^{\oplus 2}\), we get a morphism \(\calo_\pn^{\oplus(n+1)} \to \cali_L(1)\), with cokernel \(\calg\). Therefore, either this morphism is surjective, or \(\calg\) is supported on a linear space strictly containing \(L\). However, this second possibility is excluded because $L$ is maximal.

Summing up, we obtain an epimorphism \(\calm_\sigma \onto \cali_L(1)\). Since \(L\) is cut by \(n+1-\varrho\) equations, the induced epimorphism \(\calo_\pn^{\oplus(n+1)} \onto \cali_L(1)\) factors through a morphism
\(\varphi:\calo_\pn^{\oplus(n+1-\varrho)} \to \cali_L(1)\), and we get a second diagram:
\begin{equation} \label{diagram M}    
\begin{split} \xymatrix@-2ex{
& & \calo_\pn^{\oplus \varrho} \ar[d] \ar[r]^{\varphi} & \cali_W(1) \ar[d]  \\
0 \ar[r] & \calt_\sigma \ar@{.>}[d] \ar[r] & \calo_\pn^{\oplus (n+1)} \ar[r] \ar[d] & \calm_\sigma \ar[d] \ar[r] & 0 \\ 
0 \ar[r] & \calr_L(1) \ar[r] & \calo_\pn^{\oplus (n+1-\varrho)} \ar[d] \ar[r] & \cali_L(1) \ar[r] \ar[d] & 0 \\
 & & 0 & 0} \end{split}
\end{equation}

The snake lemma then yields the exact sequence
$$ 0 \to \ker(\varphi) \to \calt_\sigma \to \calr_L(1) \to \coker(\varphi) \to 0 . $$
Note that the $\im(\varphi)$ coincides with
the ideal sheaf of a linear subspace \(M\ \subset \pn\) of codimension \(\varrho\) containing \(W\), twisted by $\opn(1)$; its cokernel is the ideal of \(W\) in \(M\), also twisted by $\opn(1)$. We, therefore, obtain the exact sequence in display \eqref{4terms}.
\end{proof}

\subsubsection{Destabilizing Kozsul subsheaves} \label{destab koszul}

Now we can prove items \ref{=} and \ref{>} of Theorem \ref{stability-pencils}. Indeed, we set \(\varrho=r_0\) and consider a quadric \(Q_\lambda\) in the pencil \(\sigma\) having \(\cork(Q_\lambda)=r_0\). 
The assumption \(r_0 \ge (n+1)/2\) forces \(r_0 \ge 3\) or \((n,r_0)=(3,2)\). 
The latter case follows from the full classification of pencils of quadrics in $\p3$ and their logarithmic tangent sheaves provided in Subsection \ref{another}. 
Hence we can assume \(r_0 \ge 3\), so the linear space \(L \subset \pn\) of dimension \(r_0-1\) appearing as the singular locus of \(Q_\lambda\) satisfies the hypotheses of Lemma \ref{ii and iii}, thus \(\calt_\sigma\) contains the Koszul subsheaf \(\calr_M(1)\) which has slope \(1/(1-r_0)\). The condition \(r_0>(n+1)/2\) implies:
\[
\mu(\calr_M(1)) = \frac{1}{1-r_0} > \frac{2}{1-n}=\mu(\calt).
\]

Finally, for item \ref{=}, we use the exact sequence in display \eqref{4terms}, which yields:
\begin{equation} \label{3terms}
0 \to \calr_M(1) \to \calt_\sigma \to \cale \to 0,    
\end{equation}
where \(\cale\) is the kernel of \(\calr_L(1) \onto \cali_{W/M}(1)\). Since \(M\) has codimension \(r_0 \ge 2\), the sheaves \(\cale\) and \(\calr_L(1)\) share the same slope, namely \(1/(1-r_0)\). 
This implies that any subsheaf $\calk$ of \(\cale\) with $\mu(\calk)\ge\mu(\cale)$ would destabilize \(\calr_L(1)\); since, by Lemma \ref{calr_M}, $\calr_L(1)$ is slope-stable, we conclude that \(\cale\) is also slope-stable. When \(r_0=(n+1)/2\), then \(\calr_L(1)\) also has slope equal to \(1/(1-r_0)\). Therefore, the exact sequence in display \eqref{3terms} shows that \(\calt\) is strictly slope-semistable; in addition, since $\calr_M$ is slope-stable, we can also conclude that $\calr_M(1)$ and $\cale$ are the factors of the Jordan--Holder filtration of $\calt_\sigma$.
\bigskip

\subsection{Proof of stability for regular pencils}\label{stability regular}

Let $\sigma$ be a regular pencil of quadrics containing no double hyperplane so that there are only finitely many points \(\lambda \in \p 1\) such that \(Q_\lambda\) is singular and at each such point the singular locus of \(Q_\lambda\) is a linear space of dimension \(\cork(Q_\lambda)-1\).
A regular pencil is incompressible so these spaces are disjoint by Lemma \ref{they are disjoint},
so \((\Xi_{\sigma})_{\rm red}\) is the union of finitely linear spaces of dimension at most \(r_0-1\).

In order to prove \ref{<} we assume, by contradiction, that \(\calt\) has a saturated destabilizing subsheaf \(\calk\) of rank \(p\), with \(1 \le p \le n-2\) with \(\big(\bigwedge^p \calk\big)^{\vee \vee} \simeq \calo_\pn(-e)\). Since $\sigma$ is incompressible, Lemma \ref{sub of incompressible} implies that \(e>0\). The condition that \(\calk\) destabilizes \(\calt\) amounts to:
\[
(n-1)e \le 2p.
\]
Since \(p \le n-2\), this gives $e=1$. Also, we get \(p \ge (n-1)/2\).

Choose a sufficiently general linear subspace \(M \subset \pn\) of dimension \(n-r_0\).
Since \(\dim(M)+\dim(\Xi_\sigma)=n-1\), we may assume that \(M\) is disjoint from the degeneracy locus \(\Xi_\sigma\) and that \(M\) meets transversely the locus where \(\calt / \calk\) is not locally free. The second assumption implies that \(\Tor_1(\calt/\calk,\calo_M)=0\), so we get a subsheaf \(\calk|_M \into \calt|_M\) which still destabilizes \(\calt|_M\). The first assumption yields \(\calq_\sigma|_M =0= \Tor^1(\calm_\sigma,\calo_M)\), so the restricted Jacobian matrix gives an exact sequence:
 \begin{equation} \label{restricted to M}
0 \to \calt|_M \to \calo_M^{\oplus (n+1)} \to \calo_M^{\oplus 2}(1) \to 0.
\end{equation}
The sheaf \(\calt|_M\) is locally free and, setting $q=n-1-p$ we get:
\[
\Big(\bigwedge^p \calt|_M\Big)(1) \simeq \Big(\bigwedge^q \calt^\vee|_M \Big)(-1).
\]
Since $\calk|_M$ is a subsheaf of $\calt|_M$, we obtain a monomorphism
$$ \calo_M(-1) \simeq \big(\bigwedge^p \calk|_M\big)^{\vee\vee} \into \big(\bigwedge^p \calt|_M\big)^{\vee \vee} \simeq \big(\bigwedge^p \calt|_M\big) $$ 
which in turn gives \(H^0\Big(\Big(\bigwedge^q \calt^\vee|_M \Big)(-1)\Big)\ne 0\). We need to prove that this is absurd.

In order to check this, we dualize the exact sequence in display \eqref{restricted to M} and take exterior powers to get a long exact sequence:
\[
0 \to \calo_M(-q-1)^{\oplus (q+1)} \to \calo_M(-q)^{\oplus q(n+1)} \to \cdots \to \calo_M^{\oplus{{n+1}\choose q}} \to \Big(\bigwedge^q \calt^\vee|_M \Big)(-1) \to 0.
\]

All of the terms in the sequence above, except for the rightmost one, are copies of \(\calo_M(-t)\) for some integer \(t\) with \(1 \le t \le q+1\).
In the range \(p \ge (n-1)/2\), we have \(q=n-1-p \le (n-1)/2\) so
\(q+1 \le (n+1)/2\). Now, the assumption \(r_0 <(n+1)/2\) guarantees \((n+1)/2 < n-r_0+1=\dim(M)+1\), thus $q<\dim(M)$.
Therefore \(H^*(\calo_M(-t))=0\) for all \(1 \le t \le q+1\) and hence \(H^0\Big(\Big(\bigwedge^q \calt^\vee|_M \Big)(-1)\Big)=0 \).
This is the contradiction we were looking for, thus proving \ref{<}.

This finishes the proof of Theorem \ref{stability-pencils} for regular pencils.

\subsection{Proof of stability for irregular pencils of quadrics}\label{stability irregular}

In order to prove Theorem \ref{mthm:stable}, it only remains for us to prove item \ref{<} for irregular incompressible pencils \(\sigma\) containing no double hyperplane.

By hypothesis, we have \(2r_0<n+1\). As in Subsection \ref{stability regular}, we assume by contradiction that \(\calt\) admits a saturated subsheaf \(\calk\) of rank \(p\) and again we get \(p \ge (n-1)/2\) and \((\bigwedge^p \calk)^{\vee \vee} \simeq \calo_\pn(-1)\).

Next, observe that, if \(r_1<r_0\), then by Lemma \ref{Jacobian union} we have \(\dim(\Xi_\sigma) \le r_0-1\). In this case, the proof of \ref{<} given in Subsection \ref{stability regular} goes through as again a sufficiently general linear subspace \(M \subset \pn\) of dimension \(n-r_0\) does not meet \(\Xi_\sigma\), while the rest of the argument is still valid for irregular pencils.

Therefore, we may assume until the end of the subsection that \(r_0=r_1\), which is to say, that the pencil is completely irregular, so \(u=v\), \(\calc_\sigma = \calc_\tf\) and \(\dim(\Xi_\sigma)=r_0\).

Next, we develop a lemma that will be used to check the indecomposability of $\calt_\sigma$ for a completely irregular pencil $\sigma$.
We note that there is an \(\SL_2(\kappa)\)-action on \(\calo_Y(H+f)\), namely the action by homographies on the basis of the scroll \(Y\).
Also, the \(\SL_2(\kappa)\)-action on \(\pn\) induces the isomorphism of \(\SL_2(\kappa)\)-modules: \begin{equation} \label{SL2}
H^0(\pn,\calo_\pn(1)) = \bigoplus_{i=1}^{r_1} (V_{c_i-1} \oplus V_{c_i}).
\end{equation}

The following lemma says that, if \(\sigma\) is a completely irregular pencil of quadrics, then the logarithmic sheaf \(\calt_\sigma\) is simple.

\begin{lemma} \label{its simple}
Let \(\sigma\) be a completely irregular incompressible pencil of quadrics with \(u \ge 2\).
Then \[\End_{\pn}(\calt_\sigma) \simeq \kappa.\]
\end{lemma}

\begin{proof}
We use the exact sequence in display \eqref{J4}. We first use its rightmost part, namely:
\begin{equation}
    \label{Nf}
0 \to \calm_\sigma \to \calo_\pn(1)^{\oplus 2} \to \calo_Y(H+F)\ \to 0.
\end{equation}

Since \(\calo_Y\) is a line bundle on the smooth irreducible variety \(Y\), we have \(\End_X(\calo_Y)\simeq \kappa\). Since the morphism \(\calo_\pn^{\oplus 2} \to \calo_Y(F) \) induces an isomorphism on global sections and \(H^p(\calo_Y(F))=0\) for \(p>0\), we get \(H^*(\calm_\sigma(-1))=0\).
Also, \(\Ext^p_{\pn}(\calo_Y(F+h),\calo_\pn(1))=0\) for \(p=0,1\) by Serre duality since \(\dim(Y)=n+1-2u<n-1\). Therefore applying \(\Hom_{\pn}(-,\calm_\sigma)\) and  \(\Hom_{\pn}(\calo_Y(H+F),-)\) to the exact sequence in display \eqref{Nf} we get:
\[
\End_{\pn}(\calm_\sigma) \simeq \Ext^1_{\pn}(\calo_Y(F+h),\calm_\sigma)\simeq 
\End_{\pn}(\calo_Y) \simeq  \kappa.
\]
Also, note that Serre duality gives \(\Ext^p_{\pn}(\calo_Y(F+h),\calo_\pn)=0\) for \(p=1,2\) as \(\dim(Y)=n+1-2u<n-2\). We deduce that \(\Ext^p_{\pn}(\calm_\sigma,\calo_\pn)=0\) for \(p=0,1\).

Next, we write the exact sequence:
\begin{equation} \label{Tf}
0 \to \calt_\sigma \to \calo_\pn ^{\oplus(n+1)} \to \calm_\sigma \to 0.     
\end{equation}

Applying \(\Hom_{\pn}(\calm_\sigma,-)\) to \eqref{Tf} and using that
\(\Ext^p_{\pn}(\calm_\sigma,\calo_\pn)=0\) for \(p=0,1\), we get:
\[
\Ext^1_{\pn}(\calm_\sigma,\calt_\sigma) \simeq \End_{\pn}(\calm_\sigma) \simeq \kappa.
\]
Finally, we apply \(\Hom_{\pn}(-,\calt_\sigma)\) to \eqref{Tf} and use that, since \(\sigma\) is incompressible, we have \(H^0(\calt_\sigma)=H^1(\calt_\sigma)=0\). This concludes the proof, since
\[
\End_{\pn}(\calt_\sigma) \simeq \Ext^1_{\pn}(\calm_\sigma,\calt_\sigma) \simeq \kappa.
\]
\end{proof}

Having established this lemma, let us continue the proof of Theorem \ref{stability-pencils}. This time, with a slight difference with respect to the proof of item \ref{<} given in Subsection \ref{stability regular}, we choose a general linear subspace \(M \subset \pn\) of dimension \(n-r_0-1\). In particular, we may assume that \(M\) does not meet \(\Xi_\sigma\) and that the \(p^{\textrm{th}}\) exterior power of \(\calk|_M \hookrightarrow \calt|_M\) gives a non-zero element of \(H^0(\bigwedge^{q} \calt^\vee|_M(-1))\), with \({q=n-1-p}\); this equality $q=n-1-p$ comes from duality of the sheaf $\calt$, which is of rank $n-1$.

If \(n\) is odd, we write \(n=2n_0+1\) and \(2r_0 < n+1\) gives \(r_0 \le n_0\), while \(p \ge (n-1)/2=n_0\) gives \(q = 2n_0-p \le n_0\). Since  \(\dim(M)=2n_0-r_0\ge n_0\) we get
\(q < \dim(M)\) unless \(\dim(M)=n_0\). If \(q < \dim(M)\), again the argument given in Subsection \ref{stability regular} remains valid, so we may assume, without loss of generality, that \(\dim(M)=n_0\). It then follows that $n_0=\dim(M)=2n_0-r_0$, thus $n_0=r_0$; since \(r_0=n+1-2u\), \(n_0\) is even, say \(n_0=2n_1\), and \(u=n_1+1\). Summing up, if \(n\) is odd, then:
\begin{equation} \label{odd} 
n=4u-3, \quad \dim(M)=p=q=r_0=2(u-1), \quad u \ge 2.    
\end{equation}
Similarly, if \(n=2n_0\), then we get that $r_0 \le n_0\le p$ and thus $q\le n_0-1\le\dim(M)$, so we may assume that \(\dim(M)=n_0-1\). It follows that $n_0=r_0=n+1-2u$ so $n_0$ is odd; setting $n_0=2n_1+1$, we obtain \(u=n_1+1\). Summarizing, if \(n\) is even, then:
\begin{equation} \label{even} 
n=4u-2, \quad \dim(M)=q=2(u-1), \quad r_0=p=2u-1,\quad u \ge 2. \end{equation}
In any case, the sheaves \(\calk\) and \(\call=\calt/\calk\) are slope-stable.

Having established these numerical constraints, we proceed with the next step of the proof, which requires looking at the exact sequence in display \eqref{J4}. Working in the linear span \(L=\p{n+1-u} \subset \pn\) of \(Y\), we write an exact commutative diagram
\[
\xymatrix@-2ex{
& 0 \ar[d] & 0 \ar[d] \\
& \cali_{Y/L}(1)^{\oplus 2} \ar[d] \ar@{=}[r]& \cali_{Y/L}(1)^{\oplus 2} \ar[d]\\
0 \ar[r] & \calm_L \ar[r] \ar[d]  & \calo_L(1)^{\oplus 2} \ar[r] \ar[d] & \calo_Y(H+F) \ar@{=}[d] \ar[r] & 0 \\
0 \ar[r] & \calo_Y(H-F) \ar[d] \ar[r] & \calo_Y(h)^{\oplus 2} \ar[d] \ar[r] & \calo_Y(H+F) \ar[r] & 0 \\
& 0 & 0 
}
\]
Here, the sheaf \(\calm_L\), defined by the middle row, can be thought of as the normal sheaf associated with the pencil of quadrics restricted to the smaller space \(L\).
Using the leftmost column of the previous diagram and the fact that the morphism \(\calo_\pn(1)^{\oplus 2} \to \calo_Y(H+F)\) in the exact sequence in display \eqref{J4} factors through \(\calo_L(1)^{\oplus 2} \to \calo_Y(H+F)\), we get an exact sequence:
\[
0 \to \cali_{L/\pn}(1)^{\oplus 2} \to \calm_\sigma \to \calm_L \to 0.
\]
Using this exact sequence and the one in display \eqref{J4} we get a second exact commutative diagram:
\[
\xymatrix@-3ex{
& 0 \ar[d] & 0 \ar[d] & 0 \ar[d] \\
0 \ar[r] & \calr_L(1)^{\oplus 2} \ar[d] \ar[r]& \calo_{\pn}^{\oplus 2u} \ar[r] \ar[d] &  \cali_{L/\pn}(1)^{\oplus 2n}\ar[d] \ar[r] & 0\\
0 \ar[r] & \calt \ar[r] \ar[d]  & \calo_\pn^{\oplus (n+1)} \ar[r] \ar[d] & \calm_\sigma \ar[d] \ar[r] & 0 \\
0 \ar[r] & \calg \ar[d] \ar[r] & \calo_\pn^{\oplus (n+1-2u)} \ar[d] \ar[r] & \calm_L \ar[r] \ar[d] & 0 \\
& 0 & 0 & 0
}
\]
Here the sheaf \(\calg\), defined with the bottom row, has \(c_1(\calg)=0\) and \(\rk(\calg)=n+1-2u\), while \(\calr_L(1)\) is the Koszul syzygy of \(L\) which we already proved to be stable of slope \(1/(1-u)\).

In view of the leftmost column of the previous diagram and of the slope-semistability of \(\calr_L(1)^{\oplus 2}\), the inclusion \(\calk \hookrightarrow \calt\) must descend to an inclusion \(\calk \hookrightarrow \calg\).
We get thus two injections with isomorphic cokernel:
\begin{equation} \label{distinguished hyperplane}
\calk \hookrightarrow \calg, \qquad \calr_L(1)^{\oplus 2} \hookrightarrow \call.    
\end{equation}
Denote by \(\calp\) this common cokernel sheaf, so \(\calp \simeq \calg/\calk \simeq \call/\calr_L(1)^{\oplus 2}\).
Then, note that independently on whether \(n\) is even or odd, we get:
\[
\rk(\calk)=p=n+1-2u=\rk(\calg).
\]

Therefore, since \((\bigwedge^p \calk)^{\vee \vee} \simeq \calo_\pn(-1)\), we get a hyperplane \(H \subset \pn\) as support \(\calp\).
Note that the hyperplane \(H\) is determined uniquely by \(\calt\). Indeed, if \(\calk' \hookrightarrow \calt\) is an embedding of any saturated destabilizing subsheaf of \(\calt\), then the induced morphism \(\calk' \to \call\) is either zero or an isomorphism, since \(\calk'\) and \(\call\) are stable of the same slope. If \(\calk' \to \call\) is zero then \(\calk' \to \calt\) factors through \(\calk\) so it determines the same hyperplane \(H\). If 
\(\calk' \to \call\) is an isomorphism, then \(\calt\) is decomposable, which is absurd by Lemma \ref{its simple}.

Now, recall from Lemma \ref{it is equivariant} that \(\calt_\sigma\) is equivariant for a natural \(\SL_2(\kappa)\)-action on \(\pn\). So the hyperplane \(H\) must be fixed by this action, in other words, it must correspond to a trivial summand \(V_0\) in the decomposition in display \eqref{SL2}.

Set \(t\) for the number of indices \(i \in \{1,\ldots,n+1-2u\}\) such that \(c_i=1\), so:
\begin{equation} \label{inequality t}
u = t+\sum_{i=t+1}^{n+1-2u} c_i \ge 2(n-2u-t+1) + t.    
\end{equation}

If \(t \ge 3\), then we can equip \(Y\) (and consequently \(\calt_\sigma\)) with a further \(\SL_2(\kappa)\)-action by letting \(\SL_2(\kappa)\) operate as \(V_{t-1} \otimes \calo_{\p 1}(1)\) on the summands of \(\calc_\sigma\) of the form \(\calo_{\p 1}(1)\).  Again we obtain that \(H^0(\calo_\pn(1))\) contains no copy of \(V_0\). In all these cases the \(\SL_2(\kappa)\)-fixed hyperplane \(H\) cannot exist and we conclude that \(\calt_\sigma\) is stable.

Finally if \(t \le 2\) then using \eqref{inequality t}, depending on whether \(n\) is odd or even, we get from \eqref{odd} or \eqref{even} that \(u \le 2\) or \(u \le 1\), which leaves the only case \(n=5\), \(u=2\), \(c_1=c_2=1\). This last case corresponds to the pencil of quadrics \(\sigma = (x_1x_5+x_3x_4,x_2x_4+x_0x_5)\). For this explicit pencil, direct computation shows that \(H^0\Big(\big(\bigwedge^2 \calt_\sigma\big)^{\vee}(-1)\Big)=0\) so that \(\calt_\sigma\) is stable.

\begin{remark} \label{remark on GIT}
One may check that GIT-semistability of $\sigma$ implies slope-semistability of $\calt_\sigma$. Indeed,  \cite[Theorem 3.1]{AL} says that GIT-semistability of $\sigma$ amounts to $\sigma$ being regular with $\calc_\sigma$ supported at $\lambda_1,\ldots,\lambda_\ell$ with the condition that for all $j \in \{1,\ldots,\ell\}$, $\lambda_j$ has multiplicity at most $(n+1)/2$ as a root of $\det(\rho_\sigma)$. In terms of the Segre symbol (see the introduction or the next paragraph), this means that $\sum_{i=1}^{s_j} a_{j,i} p_{j,i} \le (n+1)/2$. This implies $\sum_{i=1}^{s_j} p_{j,i} \le (n+1)/2$, which amounts to $r({\lambda_j}) \le (n+1)/2$, so $\calt_\sigma$ is slope-semistable, as $r_0=\max \{\sum_{i=1}^{s_j} p_{j,i} \mid j \in \{1,\ldots,\ell\}\}$.
However, the converse implication fails as one can see reverting the argument or considering that there are irregular pencils $\sigma$ having a slope-semistable sheaf $\calt_\sigma$.
\end{remark}

\section{Projective dimension for pencils of quadrics} \label{pdim}

Also in this section, \(\kappa\) is an algebraically closed field of characteristic different from 2.
Let $\sigma$ be an incompressible pencil of quadrics and adopt the notation from the previous section.
Note that \(p_{j,i} \le n\) for all indices \((i,j)\). The splitting type \((u,v)\) satisfies:
\begin{equation} \label{just sum}
\sum_{j=1}^\ell \sum_{i=1}^{s_j} a_{j,i}p_{j,i} = h^0(\calc_\t) = u-v,
\end{equation}
according to the exact sequence in display \eqref{reg part sqc}.

\subsection{Ext sheaves} \label{ext sheaves}

The main result of this section provides necessary and sufficient conditions in terms of \(\Sigma\) for the Ext sheaves $\inext_{\pn}^q(\calt_\sigma,\calo_\pn)$ to be non-trivial.

\begin{theorem} \label{extheorem}
For a pencil of quadrics \(\sigma\) with Segre symbol \(\Sigma\) and for \(q>0\), we have \(\inext_{\pn}^q(\calt_\sigma,\calo_\pn) \ne 0\) if and only if there are \(j \in \{1,\ldots,\ell\}\) and \(k \in \{1,\ldots,s_j\}\) such that:
\begin{equation} \label{exts}
q + p_{j,1}+\ldots+p_{j,k} = n-r_1-1,
\end{equation}
or \(r_1>0\) and \(q+r_1=n-2\).
\end{theorem}

\begin{proof}
We prove the theorem under the assumption that \(\sigma\) is incompressible, see the end of the proof for compressible pencils.
Since the question is local and \(\calt_\sigma\) is free of rank \(n-1\) locally around any point outside \(\Xi_\sigma\), it is enough to prove the claim on charts containing a single primary component of \(\Xi_\sigma\). In view of Lemma \ref{Jacobian union}, these components are supported either at disjoint linear subspaces associated with distinct points in the support of \(\calc_\t\), or at the rational normal scroll \(Y\). Note that the points of the support of \(\calc_\t\) correspond to the parenthesized pieces of the Segre symbol. Also, the proof at the points of \(Y\) is similar if the support of \(\calc_\t\) contains one point or many.
So we may assume, without loss of generality, that \(\ell=1\) and simplify the notation to \(\Sigma = (a_1^{p_1},\ldots,a_s^{p_s})\) with 
\(a_1 > \cdots > a_s > 0\).

As we did in the proof of item \ref{<} of Theorem \ref{stability-pencils} (see Section \ref{stability regular}), we observe that the sheaf \(\calq_\sigma\) is a line bundle supported at \(\Xi_\sigma\). Therefore, given \(q > 0\), we have
\(\inext_{\pn}^{q+2}(\calq_\sigma,\calo_\pn) \ne 0\) if and only if
\(
\inext_{\pn}^{q+2}(\calo_{\Xi_\sigma},\calo_\pn) \ne 0,
\)
so this in turn is equivalent to 
\(
\inext_{\pn}^{q}(\calt_\sigma,\calo_\pn) \ne 0.
\)

Let us analyze \(\Xi_\sigma\) more in detail and recall the notation of Subsection \ref{Jacobian of irregular}.
To make the proof more transparent, we carry it out first under the assumption that \(\sigma\) is regular, hence \(\rho=\rho_\sigma\), \(r_1=0\), \((u,v)=(n+1,0)\), \(\calc_\tf=0\) and \(\calc_\t=\calc_\sigma\).
For any \(k \in \{1,\ldots,s\}\), projectivizing the surjection \(\cald^{(k)} \twoheadrightarrow \calc^{(k)}\) we get a closed embedding \(\p{}(\calc^{(k)}) \hookrightarrow \p{}(\cald^{(k)})\) of schemes of the same dimension, with a residual subscheme of \(\p{}( \calc^{(k)})\) in \(\p{}(\cald^{(k)})\) which is isomorphic to \(\p{}(\cald^{(k-1)})\) and has thus strictly smaller dimension -- by convention, \(\p{}(\cald^{(0)}) = \emptyset\).
Note that, for each \(k \in \{1,\ldots,s\}\), the exact sequence 
\[
 0 \to \cali_{\p{}(\calc^{(k)})/\p{}(\cald^{(k)})} \to \calo_{\p{}(\cald^{(k)})} \to \calo_{\p{}(\calc^{(k)})} \to 0
\]
induces an exact sequence :
\[
 0 \to H^0(\cali_{\p{}(\calc^{(k)})/\p{}(\cald^{(k)})}) \to H^0(\calo_{\p{}(\cald^{(k)})}) \to H^0(\calo_{\p{}(\calc^{(k)})}) \to 0.
\]
Thus, since  \(\cali_{\p{}(\calc^{(k)})/\p{}(\cald^{(k)})}\) is supported at \(\p{}(\cald^{(k-1)})\), we have an isomorphism \(H^0(\cali_{\p{}(\calc^{(k)})/\p{}(\cald^{(k)})}) \simeq H^0(\calo_{\p{}(\cald^{(k-1)})})\), which in turn implies that 
\(\cali_{\p{}(\calc^{(k)})/\p{}(\cald^{(k)})}\) is isomorphic to 
\(\calo_{\p{}(\cald^{(k-1)})}\).

Recalling the correspondence \eqref{correspondence}, we send this filtration to \(\pn\) and define, for each \(k \in \{1,\ldots,s\}\) the subschemes  \(\Xi^{(k)}=\psi(\p{}(\cald^{(k)})\times_{\p1}\p{}(\calc_\sigma)) \subset \pn\) and \(\Upsilon^{(k)}=\psi(\p{}(\calc^{(k)}) \times_{\p1}\p{}(\calc_\sigma)) \subset \pn\).
Since \(\psi\) is an embedding on the fibres of \(\varphi\), this gives  \(\Upsilon^{(k)} \subset \Xi^{(k)}\) for each \(k \in \{1,\ldots,s\}\) and 
finally a stratification:
\begin{align}
\label{filtration Xi}& 0 = \calo_{\Xi^{(0)}} \subset 
\calo_{\Xi^{(1)}} \subset \cdots \subset \calo_{\Xi^{(s-1)}} \subset 
\calo_{\Xi^{(s)}}=\calo_{\Xi_\sigma}, &&\mbox{with}:   \\
\label{quotients Upsilon} & \calo_{\Xi^{(k)}}/\calo_{\Xi^{(k-1)}} \simeq \calo_{\Upsilon^{(k)}},
\end{align}
for all each \(k \in \{1,\ldots,s\}\), with \(\Xi^{(1)}=\Upsilon^{(1)}\).
Each component \(\Upsilon^{(k)}\) is the projectivization of a trivial bundle of rank \(q_k\) over a subscheme of length \(a_k\) in \(\p1\). As such, it is equidimensional and Cohen--Macaulay of codimension \(n-q_k+1\) and therefore satisfies: 
\begin{equation} \label{Ext Upsilon}
\inext_{\pn}^{q+2}(\calo_{\Upsilon^{(k)}},\calo_\pn) \simeq \left\{
\begin{array}{ll}
    0, & \mbox{if \(q+2 \ne n-q_k+1\)}, \\
    \omega_{\Upsilon^{(k)}}(n+1), & \mbox{if \(q+2=n-q_k+1\)}.
\end{array}\right.    
\end{equation}
So this sheaf is non-zero if and only if \(q=n-p_1-\cdots-p_k-1\), which in turn agrees with \eqref{exts}.

We apply \(\inext_{\pn}^{*}(-,\calo_\pn)\) to the filtration \eqref{filtration Xi}. To compute this we use \eqref{quotients Upsilon} and induction on \(k \in \{0,\ldots,s\}\).
Since for all \(k \in \{1,\ldots,s\}\) the sheaves \eqref{Ext Upsilon} are line bundles supported on subschemes sharing no common component, the boundary morphisms induced by applying  \(\inext_{\pn}^{*}(-,\calo_\pn)\) to \eqref{filtration Xi} are all zero. We deduce that \(\inext_{\pn}^{q+2}(\calo_{\Xi_\sigma},\calo_\pn) \ne 0\) if and only if there is \(k \in \{1,\ldots,s\}\) such that \eqref{exts} is satisfied. This concludes the proof when \(\sigma\) is regular.
\bigskip

Now let us assume that \(\sigma\) is irregular.
Restricting \(\calc_\tf\) to each of the subschemes \(\lambda^{(a_1)},\ldots,\lambda^{(a_s)}\), we obtain the sheaves:
\begin{equation} \label{hatd}
\hat \cald^{(k)} = \cald^{(k)} \oplus \calc_\tf(-a_{k+1}), \qquad \hat \calc^{(k)} = \hat \cald^{(k)}/\hat \cald^{(k-1)} = \calo_{\lambda^{(a_k-a_{k+1})}}^{\oplus (q_k+r_1)},
\end{equation}
with the filtration:
\begin{equation} \label{filtration irregular}
\calc_\tf(-a_1) = \hat \cald^{(0)} \subset \hat \cald^{(1)} \subset \cdots \subset \hat \cald^{(s)} = \calc_\sigma.
\end{equation}

Again, for any \(k \in \{1,\ldots,s\}\), 
we get an  embedding \(\p{}(\hat \calc^{(k)}) \hookrightarrow \p{}(\hat \cald^{(k)})\) of schemes of the same dimension. The residual subscheme is \(\p{}(\hat \cald^{(k-1)})\) has strictly smaller dimension -- this time \(\p{}(\hat \cald^{(0)}) = Y\).
The component \(Y\) of \(\Xi_\sigma\) is a rational normal scroll over \(\p1\).
We denote by \(F\) the divisor class of a fiber of the scroll map \(Y \to \p1\).

Using the diagram \eqref{correspondence} we define, for each \(k \in \{1,\ldots,s\}\) the subschemes  \(\hat \Xi^{(k)}=\psi(\p{}(\hat \cald^{(k)})\times_{\p1}\p{}(\calc_\sigma)) \subset \pn\) and \(\hat \Upsilon^{(k)}=\psi(\p{}(\hat \calc^{(k)}) \times_{\p1}\p{}(\calc_\sigma)) \subset \pn\). 
We get \(\hat \Upsilon^{(k)} \subset \hat \Xi^{(k)}\). 
Note that \(\psi_*(\varphi^*(\calo_{\p 1}(1))) \simeq \calo_Y(F)\) and that \(F|_{\Upsilon^{(k)}}=0\). Hence, in view of \eqref{hatd} we obtain for each \(k \in \{1,\ldots,s\}\) an exact sequence:
\begin{equation} \label{Xi and Upsilon}
0 \to \calo_{\hat \Xi^{(k-1)}}(-a_k F) \to \calo_{\hat \Xi^{(k)}} \to\calo_{\hat \Upsilon^{(k)}} \to 0.
\end{equation}

With our convention, \(Y=\hat \Xi^{(0)}=\hat \Upsilon^{(0)}\) so for \(k=1\) the leftmost term of the above sequence is \(\calo_{\hat \Xi^{(0)}}(-a_1 F) \simeq \calo_Y(-a_1 F)\).

We have obtained a stratification of \(\Xi_\sigma\) that allows us to compute the desired Ext sheaves.
Indeed, to compute \(\inext_{\pn}^{q+2}(\calo_{\Xi_\sigma},\calo_\pn) \) for \(q>0\) we apply \(\inext_{\pn}^{*}(-,\calo_\pn)\) to \eqref{Xi and Upsilon}  and use induction on \(k \in \{0,\ldots,s\}\) together with twists by \(\calo_Y(tF)\) for suitable \(t \in \Z\).
For \(k=0\) we observe that, since \(Y \subset \pn\) is smooth of codimension \(n-r_1\), for any \(t \in \Z\) we have \(\inext_{\pn}^{q+2}(\calo_Y(t F),\calo_\pn) \ne 0 \) 
if and only if \(q=n-r_1-2\).
For \(k \ge 1\), \(\hat \Upsilon^{(k)}\) is the projectivization of a trivial bundle of rank \(q_k+r_1\) over a subscheme of length \(a_k\) in \(\p1\), so:
\begin{equation} \label{Ext hat Upsilon}
\inext_{\pn}^{q+2}(\calo_{\hat \Upsilon^{(k)}},\calo_\pn) \ne 0
 \qquad \mbox{if and only if  \(q = n-r_1-q_k-1\)},
\end{equation}
and this sheaf is \(\omega_{\hat \Upsilon^{(k)}}(n+1)\) if \(q= n-r_1-q_k-1\).
Again, since these sheaves are line bundles on  \(\hat \Upsilon^{(0)},\ldots,\hat \Upsilon^{(s)}\) and since these subschemes have no common component, we have the vanishing of all the boundary morphisms of the long exact sequence obtained by applying \(\inext_{\pn}^{*}(-,\calo_\pn)\) to \eqref{Xi and Upsilon}.
Therefore, \(\inext_{\pn}^{q+2}(\calo_{\Xi_\sigma},\calo_\pn) \ne 0\) if and only if \(q=n-r_1-p_1-\cdots-p_k-1\) for some \(k\in \{1,\ldots,s\}\) or \(q=n-r_1-2\). This concludes the proof if \(\sigma\) is incompressible.

\medskip

Finally, if \(\sigma\) has compressibility \(m\) with \(1 \le m \le n\), then we set \(\hat n=n-m\) as in Lemma \ref{split compressible} and work with the incompressible pencil of quadrics \(\hat \sigma\) in \(\p {\hat n}\) associated with \(\sigma\).
We obtained already a stratification of \(\Xi_{\hat \sigma}\) by Cohen--Macaulay closed subschemes of \(\p {\hat n}\) which are projective bundles over subschemes of \(\p 1\), or the scroll \(Y\). The equations of these subschemes, viewed in \(\p n\) define cones over such subschemes, which are still Cohen--Macaulay of the same codimension. Therefore, for all \(q >0\), we have \(\inext_{\pn}^q(\calt_\sigma,\calo_\pn) \ne 0\) if and only if 
\(\inext_{\p {\hat n}}^q(\calt_{\hat \sigma},\calo_{\p {\hat n}}) \ne 0\).
This concludes the proof.
\end{proof}

Let us give a couple of explicit examples to show the stratification appearing in the proof of the theorem.

\begin{example} \label{632}
Consider a regular pencil with Segre symbol \([(6^3,3^4,2^3)]\), so that \(\ell=1\), \(s=3\), and \((a_1,a_2,a_3)=(6,3,2)\), \((p_1,p_2,p_3)=(3,4,3)\). We have a torsion sheaf \(\calc_\sigma=\coker(\rho)\) with \(h^0(\calc_\sigma)=18+12+6=36=n+1,\) so \(n=35\) and \(\rho=\rho_6^{\oplus 3} \oplus\rho_4^{\oplus 4} \oplus \rho_2^{\oplus 3}\). 
We have:
\[
\calc_\sigma = \cald^{(3)} = \calo_{\lambda^{(6)}}^{\oplus 3} \oplus \calo_{\lambda^{(3)}}^{\oplus 4} \oplus \calo_{\lambda^{(2)}}^{\oplus 3}.
\]

The Jacobian subscheme \(\Xi_\sigma=\Xi^{(3)}\) is set-theoretically a linear subspace of \(\p{35}\) of dimension \(3+4+3-1=9\). 
We have \(\calc^{(s)} = \calc^{(3)} = \calo^{\oplus 10}_{\lambda^{(2)}}\).
The scheme \(\Xi_\sigma\) contains 
\(\Upsilon^{(3)}=\psi_*(\varphi^*(\calc^{(3)}))\) which is a double structure over \(\p{9} \subset \p{35}\).
We have:
\[
\cald^{(2)} = \calo_{\lambda^{(4)}}^{\oplus 3} \oplus \calo_{\lambda^{(1)}}^{\oplus 4}, \qquad 
\calc^{(2)} = \calo_{\lambda^{(1)}}^{\oplus 7},
\]

Note that \(\Upsilon^{(2)} = \psi_*(\varphi^*(\p{}(\calc^{(2)})))\) and \(\Xi^{(2)} = \psi_*(\varphi^*(\p{}(\cald^{(2)})))\) have dimension 6. The residual subscheme of \(\Upsilon^{(3)}\) in \(\Xi^{(3)}=\Xi_\sigma\) is \(\Xi^{(2)}\). This is  set-theoretically a \(\p6\) and contains \(\Upsilon^{(2)}\) which is a reduced \(\p6\).
We have:
\[
\cald^{(1)} = \calc^{(1)} =  \calo_{\lambda^{(3)}}^{\oplus 3}.
\]
The residual subscheme  of \(\Upsilon^{(2)}\) in \(\Xi^{(2)}\) is \(\Xi^{(1)}=\Upsilon^{(1)}\). This is a triple structure on \(\p2\).
For \(q>0\), we have \(\inext_{\pn}^q(\calt_\sigma,\calo_\pn) \ne 0\) if and only \(q \in \{24, 27, 31\}\).
\end{example}

\begin{example} \label{a big one}
An incompressible pencil with \(r_1=3\), \((c_1,c_2,c_3)=(1,2,2)\) and Segre symbol \(\Sigma=[(3^2,1^4),(4^5,3^2,2^3)]\), hence with \((u,v)=(47,5)\), lives in \(\p{54}\). Its Jacobian scheme consists of a rational normal scroll \(Y\) of dimension \(3\) and degree \(5\) spanning a \(\p 7\) and of two linear spaces \((\Xi_1)_\redu\) and \((\Xi_2)_\redu\) of dimension 5 and 9 meeting \(Y\) along two disjoint projective planes appearing as fibers of the scroll. The subscheme \((\Xi_1)\) contains a simple \(\p 5\) with a double line as residual subscheme. On the other hand the subscheme \((\Xi_2)\) contains a double \(\p 9\), whose residual subscheme still contains a simple \(\p 6\) with a simple \(\p 4\) as residual subscheme.

We have \(\ell = 2\) and \(\inext_{\pn}^q(\calt_\sigma,\calo_\pn) \ne 0\) if and only \(q=54-2-r_1\) or \(q=54-1-r_1-p_{1,1}\) or \(q=54-1-r_1-p_{1,1}-p_{1,2}\)  
or \(q=54-1-r_1-p_{2,1}\)  or \(q=54-1-r_1-p_{2,1}-p_{2,2}\)  or \(q=54-1-r_1-p_{2,1}-p_{2,2}-p_{2,3}\) which gives \( q \in \{40,43,44,45,48,49\}\).
\end{example}

From the proof of the previous theorem, we extract some precise information on the primary components of \(\Xi_\sigma\).
Assume \(\sigma\) is an incompressible pencil of quadrics having Segre symbol \(\Sigma\) and degree vector \(\mathbf{c}\), with:
 \begin{align*}
 &\Sigma = [\Sigma_1,\ldots,\Sigma_\ell], && 
 \Sigma_j = (a_{j,1}^{p_{j,1}}, \ldots, a_{j,s_j}^{p_{j,s_j}}), \\
 &\mathbf{c}=(c_1,\ldots,c_{r_1}), && 
 \end{align*}
for some integers \(r_1,\ell, s_1,\ldots, s_\ell\), \(\{(a_{j,i}, p_{j,i}) \mid j\in \{1,\ldots,\ell\},i \in \{1,\ldots,s_j\}\) with \(a_{j,1} > \cdots > a_{j,s_j}\) for all \(j\in \{1,\ldots,\ell\}\) and \(1 \le c_1 \le \cdots \le c_{r_1}\).
Recall the convention \(a_{j,i}=0\) for \(i>s_j\) and for each \(j\in \{1,\ldots,\ell\}\) set \(q_{j,k}-1=\sum_{i=1}^k p_{j,i}-1\).

\begin{corollary} \label{components of Xi}
Let \(\sigma\) be an incompressible pencil of quadrics.
 \begin{enumerate}[label=\roman*)]
     \item If \(\sigma\) is regular, then the Jacobian scheme \(\Xi_\sigma\) has primary components:
     \[
     \Upsilon_j^{(k)}, \qquad \mbox{for \(j \in \{1,\ldots,\ell\}\) and \(k \in \{1,\ldots,s_j\}\)},
     \]
     where the components \(\Upsilon_j^{(k)}\) are projective spaces of dimension \(q_{j,k}-1\) over subschemes of length \(a_{j,i}-a_{j,k+1}\) of \(\p 1\). We have:
     \[
     h^0(\calo_{\Xi_\sigma})  = \sum_{j=1}^\ell a_{j,1}, \qquad \Upsilon_j^{(k)} \cap \Upsilon_{j'}^{(k')} = \emptyset, \, \mbox{if}\, \,  j \ne j'.
     \]
     \item If \(\sigma\) is irregular, then the Jacobian scheme \(\Xi_\sigma\) consists of a smooth scroll \(Y\) of dimension \(r_1\) and degree \(v=\sum_{i=1}^{r_1}c_i\) and of the primary components:
     \[
     \hat \Upsilon_j^{(k)}, \qquad \mbox{for \(j \in \{1,\ldots,\ell\}\) and \(k \in \{1,\ldots,s_j\}\)},
     \]
     where the components \(\hat \Upsilon_j^{(k)}\) are projective spaces of dimension \(r_1+q_{j,k}-1\) over subschemes of length \(a_{j,i}-a_{j,k+1}\) of \(\p 1\). Also:
     \[
     h^0(\calo_{\Xi_\sigma})  = 1, \qquad \hat \Upsilon_j^{(k)} \cap \hat \Upsilon_{j'}^{(k')} = \emptyset, \, \mbox{if}\, \,  j \ne j'.
     \]
     Finally, setting \(\hat \Xi\) for the residual scheme of \(Y\) in \(\Xi_\sigma\), we have:
     \[
     h^0(\calo_{\hat \Xi})  = \sum_{j=1}^\ell a_{j,1}.
     \]
     \end{enumerate}

\end{corollary}

\begin{proof}
We gave in Lemma \ref{Jacobian union} a set-theoretic description of the Jacobian scheme \(\Xi_\sigma\) which shows that \(\Xi_\sigma\) consists of \(\ell\) pairwise disjoint linear spaces, together with the scroll \(Y\) in case \(\sigma\) is irregular, and in this case we also noticed that \(Y\) has dimension \(r_1\) and degree \(v\).
Also, in the proof of Theorem \ref{extheorem} we gave the structure of each primary component supported at any of the linear spaces mentioned above. Taking the union over all such spaces we get precisely the set \(\{\Upsilon_j^{(k)} \mid j \in \{1,\ldots,\ell\}, \,k  \in \{1,\ldots,s_j\}\}\), or \(\{\hat \Upsilon_j^{(k)} \mid j \in \{1,\ldots,\ell\}, \,k  \in \{1,\ldots,s_j\}\}\) depending on whether \(\sigma\) is regular or not. Note that in the proof of Theorem \ref{extheorem} we also had the component \(\hat \Upsilon^{(0)}\), but this is just the scroll \(Y\) which is already accounted for.

To compute \(h^0(\calo_{\Xi_\sigma})\), note that taking \(h^0\) of the structure sheaf is an additive operation on disjoint primary components, which is invariant under taking projective bundles and takes value \(a\) at \(\lambda_{j}^{(a)}\subset \p1\) for any \(a \in \N^*\). So for regular pencils, we get:
\[
h^0(\calo_{\Xi_\sigma}) = \sum_{j=1}^\ell \sum_{k=1}^{s_j} h^0(\calo_{\Upsilon_j^{k}}) = \sum_{j=1}^\ell \sum_{k=1}^{s_j} (a_{j,i}-a_{j,k+1}) = \sum_{j=1}^\ell a_{j,1}.
\]
For irregular pencils, the Jacobian scheme is connected as \(Y\) meets all the components \(\{\hat \Upsilon_j^{(k)} \mid j \in \{1,\ldots,\ell\}, \,k  \in \{1,\ldots,s_j\}\}\), hence we have \(h^0(\calo_{\Xi_\sigma})=1\). Finally, the primary components of \(\hat \Xi\) are precisely the \(\{\hat \Upsilon_j^{(k)} \mid j \in \{1,\ldots,\ell\}, \,k  \in \{1,\ldots,s_j\}\}\), so the last formula follows as in the regular case.
\end{proof}

\subsection{Applications to projective dimension} \label{sec:pdim}

Theorem \ref{extheorem} allows us to compute the projective dimension \(\pdim(\calt_\sigma)\) of the logarithmic tangent sheaf associated with a pencil of quadrics, namely, the minimal length of a locally free resolution of  \(\calt_\sigma\).

\begin{proposition} \label{pdim regular}
Let \(\sigma\) be an incompressible pencil of quadrics.
\begin{enumerate}[label=\roman*)]
    \item Assume \(\sigma\) is irregular. Then \(\pdim(\calt_\sigma)= n-r_1-2\).
    \item Assume \(\sigma\) is regular and put \(p = \min\{p_{j,1} \mid j \in \{1,\ldots,\ell\}\}\). Then:
    \[\pdim(\calt_\sigma)= n-p-1.
    \]
\end{enumerate}
\end{proposition}

\begin{proof}
By Theorem \ref{extheorem}, we can compute for which values of \(q \ge 1\) one has \(\inext_{\pn}^q(\calt_\sigma,\calo_\pn) \ne 0\).
On the other hand, we have:
\[
\mathrm{pdim}(\calt_\sigma)= \max\{q \in \N \mid \inext_{\pn}^q(\calt_\sigma,\calo_\pn) \ne 0\}.
\]

Fixing \(j \in \{1,\ldots,\ell\}\) and letting \(k\) vary in \(\{1,\ldots,s_j\}\) the maximal value for  \(n-r_1-p_{j,1}-\cdots-p_{j,k}-1\) is attained by choosing \(k=1\).
Such value is thus \(n-r_1-p_{j,1}-1\). Letting \(j\) vary in \(\{1,\ldots,\ell\}\), the maximal value of \(n-r_1-p_{j,1}-1\) is \(n-r_1-p-1\).
The maximum between \(n-r_1-p-1\) and \(n-r_1-2\) is \(n-r_1-2\) because \(p \ge 1\).
This gives the result.
\end{proof}

\begin{example}
A pencil \(\sigma\) as in Example \ref{632} has \(\pdim(\calt_\sigma)=31\). For Example \ref{a big one}, we get \(\pdim(\calt_\sigma)=49\).
\end{example}

\begin{example}
The completely irregular pencil \(\sigma = (x_1x_5+x_3x_4,x_2x_4+x_0x_5)\) showing up at the end of the proof of Theorem \ref{stability-pencils} has \(n=5, u=v=r_1=2\).
So for \(q>0\) we have \(\inext_{\pn}^q(\calt_\sigma,\calo_\pn) \ne 0\) if and only if \(q=1\). We get \(\pdim(\calt_\sigma)=1\).
\end{example}

\begin{corollary} \label{regular locally free}
A regular pencil of quadrics is locally free if and only if \(n \in \{1,2\}\) or its Segre symbol is \([(1,1),(1,1)]\) or \([(2,2)]\).
\end{corollary}

\begin{proof}
A regular pencil \(\sigma\) is locally free if and only if all the integers \(q\) satisfying \eqref{exts} are non-positive. This is always the case for \(n \le 2\) and holds true for the Segre symbols \([(1,1),(1,1)]\) or \([(2,2)]\) so one implication is proved.

Conversely, assume \(n\ge 3\) and \(\sigma\) locally free. Set \(a=\min\{a_{j,1} \mid j \in \{1,\ldots,\ell\}\}\).
For all \(j \in \{1,\ldots,\ell\}\), take \(k=1\) and define \(q\) by \eqref{exts}. The inequality \(q \le 0\) gives \(p_{j,1} \ge n-1\),
which implies, in view of \eqref{just sum}, that:
\[
n+1 \ge \sum_{j=1}^\ell \left(a_{j,1}(n-1) + \sum_{i=2}^{s_j} a_{j,i}p_{j,i}\right).
\]
This gives \(n(a\ell-1) \le a\ell + 1\) and therefore either \(a=\ell=1\), or \(a\ell=2\) and \(n=3\). In the former case \(s_1=1\) so \eqref{just sum} gives \(p_{1,1}=n+1\), which is impossible. In the latter, either \((a,\ell)=(1,2)\) and the Segre symbol is \([1^2,1^2]=[(1,1),(1,1)]\) or \((a,\ell)=(2,1)\) and the Segre symbol is \([2^2]=[(2,2)]\).
\end{proof}

In the same spirit, we have, more generally, the following.

\begin{corollary}
Let \(\sigma\) be an incompressible pencil of quadrics. Then:
\begin{enumerate}[label=\roman*)]
    \item If \(\sigma\) is regular, then \(\pdim(\calt_\sigma) \ge \frac{n-3}{2}\).
    \item If \(\sigma\) is irregular, then \(\pdim(\calt_\sigma) \ge \frac{2n-7}{3}\).
\end{enumerate}
\end{corollary}

\begin{proof}
Assume \(\sigma\) is irregular. Then by Lemma \ref{3r1} we get:
\[
\pdim(\calt_\sigma) = n-r_1-2 \ge n-2-\frac {n+1}3=\frac{2n-7}{3}.
\]

Next, suppose \(\sigma\) is regular.
Again put \(a=\min\{a_{j,1} \mid j \in \{1,\ldots,\ell\}\}\), \(p = \min\{p_{j,1} \mid j \in \{1,\ldots,\ell\}\}\) and use \eqref{just sum} to get \(n+1 \ge \ell ap\).
By Corollary \eqref{pdim regular}, we obtain:
\[
n+1 \ge \ell a (n-\pdim(\calt_\sigma)-1).
\]
Rearranging the terms, this yields:
\[
\pdim(\calt_\sigma) \ge n - \frac{n+1}{\ell a}-1.
\]
We saw in the previous proof that \(a\ell \ge 2\) so this gives \(\pdim(\calt_\sigma) \ge \frac{n-3}{2}\).
\end{proof}

\begin{remark}
The previous bounds are sharp. Indeed, if \(\sigma\) is a completely irregular incompressible pencil, then \(3r_1=n+1\) so \(\pdim(\calt_\sigma) = \frac{n-3}{2}\). 

Also, if \(\sigma\) is a regular pencil with \(n \ge 3\) odd, say \(n+1=2m\), then we may take \(\sigma\) to have Segre symbol \([(1^m),(1^m)]\) or \([(2^m)]\) and we get 
\(\pdim(\calt_\sigma) = \frac{n-3}{2}\).
\end{remark}

\begin{example} \label{irregular P3}
Let us list the possible cases for irregular pencils of quadrics in \(\p3\). We give the possible Segre symbols of the regular part. 
\[
\begin{small}
\begin{array}{c|c|c|c|c|c}
r_1 & \mbox{\((u,v)\) } & h^0(\calc_\t) & \calc_\tf & \mbox{Compressible} & \mbox{Segre} \\
\hline
\hline
1 & (3,0) & 3 & \calo_\p1 & \mbox{yes} & [1,1,1]  \\
1 & (3,0) & 3 & \calo_\p1 & \mbox{yes} & [2,1]  \\
1 & (3,0) & 3 & \calo_\p1 & \mbox{yes} & [3]  \\
1 & (3,0) & 3 & \calo_\p1 & \mbox{yes} & [1^2,1]  \\
1 & (3,0) & 3 & \calo_\p1 & \mbox{yes} & [(2,1)]  \\
\hline
1 & (2,1) & 1 & \calo_\p1(1) &\mbox{no} & [1]\\
\hline
\hline
2 & (2,0) & 2 &  \calo_\p1^{\oplus 2} & \mbox{yes} &[1,1]\\
2 & (2,0) & 2 &  \calo_\p1^{\oplus 2} & \mbox{yes} &[2]\\
\hline
2 & (1,1) & 0 & \calo_\p1 \oplus \calo_\p1(1) & \mbox{yes} & [\emptyset]\\
\end{array}
\end{small}
\]

The pencil with the empty Segre symbol is completely irregular.
We see that only one case gives an incompressible pencil. This one has \(\pdim(\calt_\sigma)=1\).
\end{example}

\subsection{Graded projective dimension}

Let \(n \ge 2\). We call \textit{graded projective dimension} of a torsion-free sheaf \(\cale\) on \(\pn\) :
\[
\gpdim(\cale) = \max\{q \in \{0,\ldots,n-1\} \mid \Ext_{\pn}^q(\cale,\calo_{\pn})_* \ne 0\}.
\]
Here, \(\Ext_{\pn}^q(\cale,\calo_{\pn})_*\) is a shortcut for \(\oplus_{t \in \Z}\Ext_{\pn}^q(\cale,\calo_{\pn}(t))\).
The graded projective dimension is the length of a minimal graded free resolution of the module of global sections of \(\cale\).

\begin{theorem}\label{thm gpdim}
Let \(\sigma\) be a pencil of quadrics in \(\pn\).
\begin{enumerate}[label=\roman*)]
    \item \label{gpdim regular} Assume \(\sigma\) is regular. Then:
    \[\gpdim(\calt_\sigma)=n-2,\] unless the Segre symbol \(\Sigma\) of \(\sigma\) \([1^p,1^q]\) for some \(p \ge q \ge 1\), or \([(2^q,1^p)]\) for some \(p \ge 0\) and \(q \ge 1\), in which case:
    \[\gpdim(\calt_\sigma)=n-q-1.\]
    \item \label{gpdim irregular} Assume \(\sigma\) is irregular of generic corank \(r_1\). Then:
    \[\gpdim(\calt_\sigma)=n-1,\]
    unless \(\sigma\) has degree vector \((1,\ldots,1)\), in which case:
    \[\gpdim(\calt_\sigma)=n-r_1-2.\]
\end{enumerate}
\end{theorem}

We underline that the graded projective dimension of \(\calt_\sigma\) depends on the Segre symbol only if \(\sigma\) is regular; otherwise, $\gpdim(\calt_\sigma)$ only depends on whether or not the degree vector \(\mathbf{c}=(c_1,\ldots,c_{r_1})\) contains a value strictly greater than 1.

\begin{proof}
First of all, we observe that, without loss of generality, we can assume that the pencil \(\sigma\) is incompressible. Indeed, if \(\sigma\) has compressibility \(m>0\), then we may work in a projective space of dimension \(\hat n = n-m\) whose coordinates do occur in the quadrics of \(\sigma\). The minimal resolution obtained over the coordinate ring of this space is a minimal resolution of \(\calt_\sigma/\calo_{\pn}^{\oplus m}\) and thus computes \(\gpdim(\calt_\sigma)\).

\medskip
Next we note that according to the proof of Theorem \ref{stability-pencils}, the sheaf \(\calq_\sigma(-1)\) of an incompressible pencil \(\sigma\) is isomorphic to \(\calo_{\Xi_\sigma}(F)\), where \(F\) is the class of a fiber of the scroll map \(Y \to \p 1\). The divisor \(F\) is trivial on the components \(\{\hat \Upsilon_j^{(k)} \mid j \in \{1,\ldots,\ell\}, \,k  \in \{1,\ldots,s_j\}\}\).
Set \(a=\sum_{j=1}^\ell a_{j,1}\). By Corollary \ref{components of Xi}, we get:
\begin{equation} \label{h0q}
h^0(\calq(-1))=a, \qquad \mbox{if \(\sigma\) is regular}.
\end{equation}
Also, if \(\sigma\) is irregular, denoting again by \(\hat \Xi\) the residual scheme of \(Y\) in \(\Xi_\sigma\), we get an exact sequence:
\begin{equation} \label{filtration Q}
0 \to \calo_Y(F) \to \calq(-1) \to \calo_{\hat \Xi} \to 0.    
\end{equation}

In order to prove the result, we will need two more ingredients, namely two equivalent definitions of the graded projective dimension.
For the first one, for any coherent sheaf \(\cale\) on \(\pn\) and \(q \in \N\), put \(H^q_*(\cale) = \bigoplus_{t \in \Z} H^q(\cale(t))\). Set:
\[
q_0 = \min\{q \in \N^* \mid H^q_*(\calt_\sigma) \ne 0\}.
\]
We have, by Serre duality:
\[
\gpdim(\calt_\sigma)=n-q_0.
\]

The second one is worked out in the framework of graded modules over the polynomial ring \(R=\kappa[x_0,\ldots,x_n]\).
Consider the matrix \(\nabla\sigma\) as a map of graded modules \(R^{n+1} \to R(1)^{2}\) and define the \(R\)-modules \(Q_\sigma\), \(M_\sigma\) and \(T_\sigma\) as the cokernel, image, and kernel of this map, so that sheafifying the graded modules \(Q_\sigma\), \(M_\sigma\) and \(T_\sigma\) we get back \(\calq_\sigma\), \(\calm_\sigma\) and \(\calt_\sigma\).
We write down the exact sequence of graded \(R\)-modules:
\[
0 \to T_\sigma \to R^{n+1} \xrightarrow{\nabla\sigma} R(1)^{2} \to Q_\sigma \to 0.
\]
The Auslander--Buchsbaum formula gives: 
\[
\gpdim(\calt_\sigma)=n+1-\depth(T_\sigma).
\]
We compute the depth of \(T_\sigma\) by the relation:
\[
\depth(T_\sigma) = \min\{q \in \N \mid \Ext_R^q(\kappa,T_\sigma) \ne 0\},
\]
where \(\kappa\) is the residual field, namely \(\kappa= R/(x_0,\ldots,x_1)\).
\bigskip

Having set up all this, we are in a position to prove \ref{gpdim regular}. So assume \(\sigma\) is regular.
First we compute \(\gpdim(\calt_\sigma)\) when \(\Sigma=[1^p,1^q]\) or \(\Sigma=[(2^q,1^p)]\). 
If \(\Sigma=[1^p,1^q]\) with \(p \ge q \ge 1\) then \(p+q=n+1\) and the generators \((f,g)\) of \(\sigma\) can be chosen to be \(f=x_0^2+\cdots+x_{p-1}^2\) and \(g=x_p^2+\cdots+x_n^2\). Set \(L=\VV(x_0,\ldots,x_{p-1})\) and \(M=\VV(x_p,\ldots,x_n)\). Looking at \(\nabla\sigma\), we see that:
\[
\calt_\sigma \simeq \calr_L(1) \oplus \calr_M(1),
\]
where \(\calr_L\) and \(\calr_M\) are the Koszul syzygies of \(L\) and \(M\), see \S \ref{destab koszul}.
Now \(\gpdim(\calr_L(1))=p-2 \ge q-2 = \gpdim(\calr_M(1))\), so:
\[\gpdim(\calt_\sigma)=p-2=n-q-1.\]
Next, we deal with \(\Sigma=[(2^q,1^p)]\), for \(q \ge 1\) and \(p \ge 0\). Note that \(p+2q=n+1\), so 
\(
n-q-1 \ge q
\).
When \(p=0\) we have \(\calq_\sigma(-1) \simeq \calo_{\Upsilon^{(1)}}\), where \(\Upsilon^{(1)}\) is a projective space \(\p {n-q}\) over a length-2 subscheme of \(\p 1\), while for \(p>0\) we have a filtration:
\[
0 \to \calo_{\Upsilon^{(1)}} \to \calq_\sigma(-1) \to \calo_{\Upsilon^{(2)}} \to 0,
\]
where \(\Upsilon^{(1)}\) is a reduced \(\p {n-q}\) and \(\Upsilon^{(2)}\) is a reduced \(\p {q-1} \subset \p {n-q}\).
In both cases, since the coordinate rings of the subschemes \(\Upsilon^{(k)}\) are graded Cohen--Macaulay rings, we have \(\gpdim(\calo_{\Upsilon^{(k)}})=\codim(\Upsilon^{(k)})\). Therefore \(\gpdim(\calt_\sigma)+2\) is the
maximum of the \(\gpdim(\calo_{\Upsilon^{(k)}})\) for different values of \(k\).
Since \(
n-q-1 \ge q
\), in both cases we obtain the equality:
\[\gpdim(\calt_\sigma)=n-q-1.\]

Let us show that, unless \(\Sigma=[1^p,1^q]\) or \(\Sigma=[(2^q,1^p)]\), we have 
\(H^2(\calt_\sigma(-1)) \ne 0\), which implies \(\gpdim(\calt_\sigma) \ge n-2\).
It suffices to show \(H^1(\calm_\sigma(-1)) \ne 0\), which in turn holds true if 
\(h^0(\calq_\sigma(-1)) > 2\).
But by \eqref{h0q}, we have \(h^0(\calq_\sigma(-1)) > 2\) unless \(\ell=2\) and \(s_1=s_2=a_{1,1}=a_{2,1}=1\), or  \(\ell=1\), \(s_1 \in \{1,2\}\), \(a_{1,1}=2\). Since these two cases correspond to the Segre symbols \(\Sigma=[1^p,1^q]\) for some \(p,q \ge 1\) or \(\Sigma=[(2^q,1^p)]\) for some \(q \ge 1, p\ge 0\), we get \(\gpdim(\calt_\sigma) \ge n-2\) except in these cases.
\medskip

Now we prove that, for regular pencils, we have \(\gpdim(\calt_\sigma) \le n-2\). 
It suffices to show that the module \(Q_\sigma\) contains no copy of the residual field \(\kappa\).
Indeed, otherwise there would be a non-zero element of \(Q_\sigma\), represented by \((h,k) \in R_1^2\), whose annihilator contains the maximal ideal \((x_0,\ldots,x_n)\). 
Up to switching the factors, we may assume \(h \ne 0\). 
Also, since the pencil \(\sigma\) is regular,
we may choose the generators \((f,g)\) of the pencil to be both associated with smooth quadrics. Also, we may select coordinates \(x_0,\ldots,x_n\) of \(\pn\) so that \(2f=x_0^2+\cdots+x_n^2\).

Then, for the pair \((h,k) \in R_1^2\) with \(h \ne 0\), there must be a matrix \((a_{i,j})_{0 \le i,j \le n}\), with \(a_{i,j} \in R\) for all \(0 \le i,j \le n\) such that:
\begin{equation}
\label{matrix relation}    
\begin{pmatrix}
      h \\
      k
\end{pmatrix} 
\begin{pmatrix}
      x_0 & \cdots & x_n
      \end{pmatrix} =
      \begin{pmatrix}
      x_0 & \cdots & x_n \\
      g_0 & \cdots & g_n 
      \end{pmatrix} 
      \begin{pmatrix}
      a_{0,0} & \cdots & a_{0,n} \\
      \vdots & & \vdots \\
      a_{n,0} & \cdots & a_{n,n} 
      \end{pmatrix},
      \end{equation}
where \(\sigma=(f,g)\) and we wrote \(g_i=\partial g/\partial x_i\), for all \(i \in \{0,\ldots,n\}\). We used here  \(2f=x_0^2+\cdots+x_n^2\).
Note that, by the symmetric role of \(f\) and \(g\), we may assume \(h \ne 0\).
Hence, \eqref{matrix relation} implies that the matrix \(A=(a_{i,j})_{0 \le i,j \le n}\) is \(h \III_{n+1}\) and is thus invertible in \(\kappa(x_0,\ldots,x_n)\). Hence, we may rewrite \eqref{matrix relation} in \(\kappa(x_0,\ldots,x_n)\) as:
\[
\begin{pmatrix}
      h \\
      k
\end{pmatrix} 
\begin{pmatrix}
      x_0 & \cdots & x_n
      \end{pmatrix} A^{-1} =       \begin{pmatrix}
      x_0 & \cdots & x_n \\
      g_0 & \cdots & g_n 
      \end{pmatrix}.
\]
Therefore \(\nabla\sigma\) should have generic rank 1, which is impossible by the Euler relation since \(f,g\) are not proportional.

Summing up, we have proved \(\Hom_R(\kappa,Q_\sigma)=0\).
Therefore, applying \(\Ext_R^*(\kappa,-)\) to the above sequence we get \(\Ext_R^q(\kappa,T_\sigma)=0\) for \(q \le 2\).
Hence \(\depth(T_\sigma) \ge 3\) and finally 
\(\gpdim(\calt_\sigma) \le n-2\). This concludes the proof for regular pencils.

\bigskip
It remains to carry out the proof if \(\sigma\) is irregular. 
In view of the filtration \eqref{filtration Q} and since the coordinate rings of the primary components of \(\hat \Xi\) are graded Cohen--Macaulay rings, we get: \[
\gpdim(\calo_{\hat \Xi})=\codim(\hat \Xi) \le \codim(Y)=n-r_1.
\]

Now, if \(c_i=1\) for all \(i \in \{1,\ldots,r_1\}\), the sheaf \(\calo_Y(F)\) has a minimal Buchsbaum--Rim resolution of length equal to \(\codim(Y)\). This is obtained by the Buchsbaum--Rim resolution of \(\calo_Y(F)\) in the linear span \(L = \p {v-1} \subset \pn\) of \(Y\) seen as the cokernel of a matrix of linear forms of size \(2 \times v\) over \(L\), combined with the Koszul complex of \(L\) in \(\pn\); we refer to \cite[Theorem A2.10, Exercise A2.19]{E} for Buchsbaum--Rim complexes and matrices associated with scrolls.

We deduce \(\gpdim(\calo_Y(F))=\codim(Y)\), which in turn yields:
\[
\gpdim(\calt_\sigma)=\max(\gpdim(\calo_{\hat \Xi}),\gpdim(\calo_Y(F))-2=n-r_1-2.
\]
This proves the last part of \ref{gpdim irregular}. 

\medskip
To conclude the proof, let us assume that there is \(i \in \{1,\ldots,r_1\}\) such that \(c_i \ge 2\) and show that \(\gpdim(\calt_\sigma)=n-1\).
It is enough to show that \(H^1(\calt_\sigma) \ne 0\). 
Note that \(H^0(\calt_\sigma)=0\) by incompressibility of \(\sigma\), hence:
\[
h^1(\calt_\sigma) \ge 2(n+1)-h^0(\calq_\sigma)-(n+1) \ge n+1-h^0(\calq_\sigma),
\]
so it suffices to check \(h^0(\calq_\sigma)<n+1\).

To show this inequality we recall the notation \((u,v)\) for the splitting type of \(\sigma\) and take cohomology of \eqref{filtration Q} twisted by \(\calo_{\pn}(1)\).
Since the linear span of the residual subscheme \(\hat \Xi\) of \(Y\) in \(\Xi_\sigma\) is \(\p{}(H^0(\calc_\t))\) and since \(h^k(\calo_Y(F+H))=0\) for \(k >0\) and \(h^0(\calo_Y(tF+H))=v+tr_1\), for all \(t \ge 0\), we obtain:
\begin{align*}
h^0(\calq_\sigma)&=h^0(\calo_Y(F+H))+h^0(\calo_{\hat \Xi}(1)) = \\ 
 & =\sum_{i=1}^{r_1} (c_i+2) + h^0(\calc_\t)=(v+2r_1)+(u-v).
\end{align*}
We are reduced to show \(u+2r_1 < n+1\) and, since \(n+1=u+v+r_1\), this amounts to \(v > r_1\).
But the inequality \(v=\sum_{i=1}^{r_1} c_i > r_1\) takes place precisely if there is \(i \in \{1,\ldots,r_1\}\) such that \(c_i \ge 2\), so the non-vanishing \(H^1(\calt_\sigma) \ne 0\) is established.
The proof of the theorem is achieved.
\end{proof}

\begin{example}
Let \(\sigma\) be a regular pencil of quadrics with \(r_0=1\). Then the module of global sections \(T_\sigma\) of \(\calt_\sigma\) has the Buchsbaum--Rim type resolution:
\[
0 \leftarrow T_\sigma \leftarrow R(-2)^{\binom{n+1}3}\leftarrow \cdot  \cdot \cdot \leftarrow R(-1-j)^{j\binom{n+1}{j+2}} \leftarrow \cdot  \cdot \cdot
\leftarrow R(-n)^{n-1} \leftarrow 0,
\]
with \(j \in \{1,\ldots,n-1\}\).
This resolution is minimal and linear of length \(n-2\).

For a regular pencil of quadrics, the Buchsbaum--Rim complex in the above display is exact if and only if \(\dim(\calq_\sigma)=0\). This happens if and only if \(r_0=1\).
\end{example}

\section{Pencils of quadrics in dimension 3} \label{another}

We can arrive at a full classification for pencils of quadrics in \(\p 3\) over an algebraically closed field \(\kappa\) of characteristic different from 2. The result is the following.

\begin{theorem}\label{pencil dim 3}
 Let \(\sigma\) be a pencil of quadrics in \(\p 3\). Then the following holds.
 \begin{enumerate}[label=\roman*)]
     \item The pencil \(\sigma\) is free if and only if it is locally free. This happens:
     \begin{enumerate}[label=\alph*)]
         \item If \(\sigma\) has Segre symbol \([(1,1),(1,1)]\), in which case \(\calt_\sigma \simeq \calo_{\p 3}(-1)^{\oplus 2}\);
         \item If \(\sigma\) has Segre symbol \([(2,2)]\), in which case \(\calt_\sigma \simeq \calo_{\p 3}(-1)^{\oplus 2}\);
         \item If \(\sigma\) is irregular and incompressible, in which case \(\calt_\sigma \simeq \calo_{\p 3}(-1)^{\oplus 2}\);
         \item If \(\sigma\) is compressible, in which case \(\calt_\sigma \simeq \calo_{\p 3} \oplus \calo_{\p 3}(e-2) \), where \(e \in \{0,1,2\}\) is the total degree of the singular divisors of $\sigma$.
     \end{enumerate}
     \item In all other cases \(\sigma\) is regular, \(\pdim(\calt_\sigma)=1\) and the sheafified minimal graded free resolution of \(\calt_\sigma\) reads:
      \begin{enumerate}[label=\alph*)]
         \item If \(r_0=1\):
         \[
        0 \to \calo_{\p 3}(-3)^{\oplus 2} \to \calo_{\p 3}(-2)^{\oplus 4} \to \calt_\sigma \to 0.
        \]
        \item If \(r_0=2\) and the Segre symbol is not \([(1,1),(1,1)]\) or \([(2,2)]\):
         \[
        0 \to \calo_{\p 3}(-3) \to \calo_{\p 3}(-2)^{\oplus 2} \oplus \calo_{\p 3}(-1) \to \calt_\sigma \to 0.
        \]  
     \end{enumerate}
 \end{enumerate}
\end{theorem}

The proof of the theorem is by inspection of the different Segre symbols. It follows from the analysis appearing in the next subsubsections.

\subsection{Regular pencils}\label{subsec:regular pencils}
Let us write the table of possible Segre symbols of regular pencils, together with the description of \(\Xi_\sigma\) arising from the previous sections.

\[
\begin{small}
\begin{array}{c|c||c|c|c|c|c}
\mbox{Segre} & \ell & \Xi_\sigma & r_0 & \mbox{Chern} & \mbox{stable} & \pdim\\
\hline
\hline
 {[1,1,1,1]} & 4 & 
\mbox{4 simple points} & 1 & (-2,3,4) & \mbox{s} & 1\\
\hline
{[2,1,1]} & 3 & \mbox{\begin{tabular}{c}
     double point \&
     2 simple points 
\end{tabular}} & 1 & (-2,3,4) & \mbox{s} & 1\\
\hline
{[2,2]} & 2 & 
    \mbox{2 double points}
 & 1 & (-2,3,4) & \mbox{s} & 1 \\
\hline
{[3,1]} & 2 & \mbox{\begin{tabular}{c}
     triple point \&
     simple point 
\end{tabular}} &  1 & (-2,3,4) & \mbox{s}& 1\\
\hline
{[4]} & 3 &  \mbox{quadruple point} &  1 & (-2,3,4) & \mbox{s} & 1\\
\hline
\hline
{[1^2,1,1]} & 2 & \mbox{\begin{tabular}{c}
     line \&
     2 simple points 
\end{tabular}} & 2 & (-2,2,2) & \mbox{sss} & 1\\
\hline
{[1^2,2]} & 2 & \mbox{\begin{tabular}{c}
     line \&
     double point 
\end{tabular}} & 2 & (-2,2,2) & \mbox{sss} & 1\\
\hline
{[(2,1),1]} & 2 & \mbox{\begin{tabular}{c}
     line \(\Upsilon^{(2)}\) \& \\
     simple point \(\Upsilon^{(1)}\) \& \\
     simple point
\end{tabular}} & 2 & (-2,2,2) & \mbox{sss} & 1\\
\hline
{[(3,1)]} & 1 & \mbox{\begin{tabular}{c}
     line \(\Upsilon^{(2)}\) \& \\
     double point \(\Upsilon^{(1)}\) 
\end{tabular}} & 2 & (-2,2,2) & \mbox{sss} & 1\\
\hline
\hline
{[1^2,1^2]} & 2 & \mbox{\begin{tabular}{c}
     2 disjoint lines 
        \end{tabular}} & 2 & (-2,1,0) & \mbox{free} & 0\\
\hline
{[2^2]} & 1 & \mbox{\begin{tabular}{c}
     double line \(\Upsilon^{(1)}\) \\
\end{tabular}} & 2 & (-2,1,0) & \mbox{free} & 0\\
\hline
\hline
{[1^3,1]} & 2 & \mbox{\begin{tabular}{c}
     plane \& 
     simple point
\end{tabular}} & 3 & (-1,1,1) & \mbox{s} & 1\\
\hline
{[(2,1^2)]} & 1 & \mbox{\begin{tabular}{c}
     plane \(\Upsilon^{(2)}\) \& \\ 
     simple point \(\Upsilon^{(1)}\)
\end{tabular}} & 3 & (-1,1,1) & \mbox{s} & 1\\
\end{array}
\end{small}
\]

In the column labeled \textit{stable}, we wrote \textit{s} or \textit{sss} according to whether \(\calt_\sigma\) is stable or strictly semistable (in the sense of the slope), and \textit{free} when \(\calt_\sigma\) is split. In the description of \(\Xi_\sigma\), we let the subschemes \(\Upsilon^{(k)}\) show up when a primary component of \(\Xi_\sigma\) has a non-trivial filtration as in the proof of Theorem \ref{extheorem}. In the column labelled \textit{Chern} we write the triple \((c_1(\calt_\sigma), c_2(\calt_\sigma) , c_3(\calt_\sigma))\).

\bigskip

Some comments are in order. 
\begin{enumerate}[label=\roman*)]
    \item When \(\sigma\) is free (and regular), we have \(\calt_\sigma \simeq \calo_{\p 3}(-1)^{\oplus 2}\). 
    \item When \(r_0=1\), the sheaf \(\calq_\sigma\) has resolution of type Buchsbaum--Rim that induces a sheafified minimal graded free resolution:
    \[
    0 \to \calo_{\p 3}(-3)^{\oplus 2} \to \calo_{\p 3}(-2)^{\oplus 4} \to \calt_\sigma \to 0.
    \]
    This gives the Chern classes of \(\calt_\sigma\) when \(r_0=1\).
    \item When \(r_0=2\), there are lines \(M,L \subset \p 3\), not necessarily distinct, with \(L \subset \Xi_\sigma\), and a finite length subscheme \(W \subset M\), such that \(\calt_\sigma\) fits into \eqref{4terms}. Note that, since \(r_0=2\), there is a quadric in the pencil, say \(f_2\), which is a rank-2 quadric in the coordinates \(x_0,x_1\), up to homography. So, setting \(L=\VV(x_0,x_1)\) and composing the Jacobian matrix with the projection \(\calo_{\p3}^{\oplus 2}(1) \to \calo_{\p3}(1)\) onto the second factor and with the obvious quotient \(\calo_{\p3}(1) \to \calo_L(1)\) we obtain explicitly the morphism \(\calq_\sigma \to \calo_L(1)\) required to get \eqref{4terms}, so Lemma \ref{ii and iii} holds also for \((n,\varrho)=(3,2)\).
    
    We have \(\calr_M \simeq \calr_L \simeq \calo_{\p 3}(-2)\) and the length of \(W\) is either 2 or 0, according to whether \(\pdim(\calt_\sigma)\) is 1 or 0.
    So \(\calt_\sigma\) is polystable in the free case, otherwise, it is strictly slope semistable and Gieseker-unstable.

    In the latter case, we have \(\cali_{W/M}\simeq \calo_M(-2)\) and the morphism \(\calr_L(1) \to \cali_{W/M}(1)\) of \eqref{4terms} is the natural surjection 
    \(\calo_{\p 3}(-1) \to \calo_M(-1)\). Therefore the term \(\calg_2\) appearing in the sequence \eqref{3terms} fits into:
    \[
    0 \to \calo_{\p 3}(-3) \to   \calo_{\p 3}(-2)^{\oplus 2} \to \calg_2 \to 0.
    \]
    Then we have a sheafified minimal graded free resolution of \(\calt_\sigma\) of the form:
    \[
    0 \to \calo_{\p 3}(-3) \to \calo_{\p 3}(-2)^{\oplus 2} \oplus \calo_{\p 3}(-1) \to \calt_\sigma \to 0.
    \]
    This gives the Chern classes of \(\calt_\sigma\) when \(r_0=2\). 
    \item Stability of \(\calt_\sigma\) for \(r_0=3\) follows from Proposition \ref{double planes}.
    \item One can put two quadrics of \(\sigma\) in normal form. This is done in \cite{FMS}.
\end{enumerate}

\subsection{Irregular pencils}

We gave in Example \ref{irregular P3} the list of numerical invariants of irregular pencils in \(\p 3\). We observed that there is only one irregular incompressible pencil in \(\p 3\).
Note that in any case the number of points in the support of \(\calc_\t\) is at most 3, so irregular pencils have a normal form \((f_1,f_2)\) which is completely determined up to \(\SL_2(\kappa)\)-action as \(\SL_2(\kappa)\) is 3-transitive on \(\p 1\).
Note that we can assume that this support is contained in \(\{(1:0),(0:1),(1:1)\}\). 
 We are going to see that, for irregular pencils of quadrics \(\sigma\) in \(\p 3\), the sheaf \(\calt_\sigma\) is always free.
 For details on this, see Theorem \ref{classification of free}.


\subsubsection{Irregular incompressible pencils} \label{funny case}

The unique irregular incompressible pencil on \(\p 3\) has \(r_1=1\) so according to Proposition \ref{pdim regular} we have \(\calt_\sigma\) locally free.

The splitting type of \(\sigma\) is \((2,1)\) and the regular part of \(\sigma\) vanishes at a single point \(\lambda \in \p 1\) which gives a single quadric of corank \(2\) in the pencil, so \(r_0=2\). This gives a component \(\Xi^{(1)} \subset \Xi_\sigma\) which is a reduced line.
The component \(Y\) of \(\Xi_\sigma\) is a line which meets \(\Xi^{(1)}\) at \(\lambda\).
In the normalized form appearing in the proof of Theorem \ref{extheorem}, we have \(\lambda=(0:1)\) and the matrix \(\rho\) reads:
\[
\left(
\begin{array}{ccc|c}
0 & z_1 & z_2 & 0 \\
z_1 & 0 & 0 & 0 \\
z_2 & 0 & 0 & 0 \\
\hline
0 & 0 & 0 & z_1 
\end{array}
\right).
\]

The associated pencil is \((x_0 x_2, 2x_0 x_1+x_3^2)\) and, up to dividing by 2, the Jacobian matrix reads:
\[
\left(\begin{array}{cccc}
     x_2 & 0 & x_0 & 0 \\
     x_1 & x_0 & 0 & x_3
\end{array}\right)
      \]
The kernel of this matrix is \(\calt_\sigma \simeq \calo_{\p 3}(-1)^{\oplus 2}\), the syzygy map being:
\[
\left(\begin{array}{cccc}
     x_0 & 0 \\
     -x_1 & -x_3 \\
     -x_2 & 0 \\     
     0 & x_0 \\
\end{array}\right).
\]

\subsubsection{Compressible pencils}

Assume \(\sigma\) is compressible. Then the sheaf \(\calt_\sigma\) contains a copy of the trivial sheaf \(\calo_{\p 3}\) which is thus a direct summand of \(\calt_\sigma\). Therefore \(\calt_\sigma / \calo_{\p 3}\) is a reflexive sheaf of rank one and thus isomorphic to \(\calo_{\p 3}(-c_1(\calt_\sigma))\).
For compressible pencils, we get the following.
\begin{enumerate}[label=\roman*)]
    \item If \(\sigma\) has no singular divisor \(\calt_\sigma \simeq \calo_{\p 3} \oplus \calo_{\p 3}(-2) \).
    \item If \(\sigma\) has a singular plane, then \(\calt_\sigma \simeq \calo_{\p 3} \oplus \calo_{\p 3}(-1) \),
    cf. \ref{case 1} of Lemma \ref{c1=0,1}
    \item If \(\sigma\) has a singular divisor of degree 2, then \(\calt_\sigma \simeq \calo_{\p 3} \oplus \calo_{\p 3} \),
    cf. item \ref{case 0} of Lemma \ref{c1=0,1}.
\end{enumerate}


\section{Locally free pencils of quadrics} \label{locally free pencils}

Let us conclude the analysis of freeness and local freeness of pencils of quadrics over an algebraically closed field \(\kappa\) of characteristic different from 2 on $\p n$ with $n \ge 2$. In the next table, the column labeled \(\Sigma\) displays the Segre symbol of the regular part of \(\sigma\).
In the column labeled \textit{exponents} we write the sequence of degrees of the line bundles which are direct summands of \(\calt_\sigma\), for instance, the sequence \((0^{n-3},-1^2)\) means that \(\calt_\sigma \simeq \calo_{\pn}^{\oplus(n-3)} \oplus \calo_{\pn}(-1)^{\oplus 2}\). Recall from Section \ref{webs} that $\hat{n}=n-h^0(\calt_\sigma)$. When \(n = 2\) one should not consider the first three lines.

\begin{theorem} \label{classification of free}
 Let \(n\ge 2\). A pencil of quadrics \(\sigma\) is free if and only if \(\sigma\) is locally free. This happens if and only if, up to homography, \(\sigma=(f_1,f_2)\) is:
\begin{equation} \label{table of free}
 \begin{array}{c|c|c|c|c|c|c|c}
    f_1 & f_2 & \mbox{{\rm exponents}} & \hat n & r_0 & r_1 & (u,v)& \Sigma \\
    \hline
    \hline
    x_2^2+x_3^2 & x_0^2+x_1^2 & (0^{n-3},-1^2) & 3 & n-2 & n-3 & (4,0) & [1^2,1^2] \\
    x_0^2+x_2^2 & x_0x_1 + x_2x_3 & (0^{n-3},-1^2) & 3 & n-2 & n-3 & (4,0) & [2^2] \\
    x_0 x_2 & 2x_0x_1+x_3^2 &(0^{n-3},-1^2)& 3 & n-1 & n-2 & 
    (2,1)  & [1] \\
    \hline
    \hline
    x_0^2+x_2^2 & x_1^2-x_2^2 & 
    (0^{n-2},-2) & 2 & n-1 & n-2 & (3,0) & [1,1,1] \\
    x_0^2+x_2^2 & x_0x_1 & 
    (0^{n-2},-2) & 2 & n-1 & n-2 & (3,0) & [2,1] \\
    x_0x_1 & 2x_0x_2+x_1^2 &
    (0^{n-2},-2) & 2 & n-1 & n-2 & (3,0) & [3] \\    
    \hline
    x_0^2 & 2x_0x_1+x_2^2 & 
    (0^{n-2},-1)& 2 & n & n-2 & (3,0) & [(2,1)] \\
    x_2^2 & x_0^2+x_1^2 & 
    (0^{n-2},-1) & 2 & n & n-2 & (3,0) & [1^2,1] \\
    x_0x_2 & x_0x_1 & (0^{n-2},-1) & 2 & n-1 & n-1 & (1,1) & \emptyset\\
    \hline
    \hline
    x_1^2  & x_0^2 & (0^{n-1}) & 1 & n & n-1 & (2,0) & [1,1] \\
    x_0^2 & x_0 x_1 & (0^{n-1}) & 1 &  n & n-1 & (2,0) & [2]
 \end{array}
\end{equation}
\end{theorem}

\begin{proof}
Let \(\sigma\) be a locally free pencil of quadrics. Following the notation of Section \ref{webs}, set \(m\) for the compressibility of \(\sigma\) and \(\hat n=n-m\). By Lemma \ref{split compressible}, the sheaf \(\calt_\sigma\) decomposes as \(\calo_{\pn}^{\oplus m} \oplus \cale\), thus $\cale$ must be locally free. In addition, the associated incompressible pencil \(\hat \sigma\) is also locally free, since \( \calt_{\hat\sigma}\) coincides with the restriction of $\cale$ to some $\hat n$-dimensional linear space.

If \(\hat \sigma\) is regular, then Corollary \ref{regular locally free} says that \(\hat n \le 2\) or \(\hat n = 3\) and \(\hat \sigma\) has Segre symbol \([1^2,1^2]\) or \([2^2]\).
In the latter case the normal form of the quadrics of \(\hat \sigma\) obtained as in the proof of Theorem \ref{extheorem} is the one displayed in the first two lines of the table in display \eqref{table of free}. Since \(\sigma\) depends only on \(x_0,\ldots,x_{\hat n}\), this is actually the normal form of the quadrics of \(\sigma\).

On the other hand, if \(\hat \sigma\) is irregular, then, since \(\hat \sigma\) is locally free, setting \(\hat r_1\) for the generic rank of \(\hat \sigma\), we must have \(\hat r_1 \ge \hat n-2\) by Corollary \ref{pdim regular}. Combining this with Lemma \ref{3r1} gives \(\hat n+1 \ge 3 \hat n-6\) which implies \(\hat n \le 3\). If \(\hat n=3\), we are in the situation of Subsection \ref{funny case}, and we obtain the third line of the above table.
In all these cases we have seen already that $\sigma$ is free with the desired exponents.

\medskip

It remains to treat the cases \(\hat n \le 2\). Let us assume \(\hat n = 2\). Since \(\hat \sigma\) is incompressible, \(\calt_{\hat \sigma}\) is a reflexive sheaf of rank 1 with determinant equal to \(e-2\) where \(e\) the total degree of the singular divisors of \(\hat \sigma\). Note that \(e \in \{0,1\}\) as \(\hat \sigma\) is incompressible. 

If $\hat \sigma$ is irregular (and incompressible), then $\hat r_1 = \hat c_1=1$ (with obvious notation) so $\hat \sigma$ is completely irregular and we get the third line from the bottom. 

Still with $\hat n=3$, assuming now $\hat \sigma$ regular, we get \(e=0\) if and only if \(r_0=n-1\) and this corresponds to the Segre symbols \([1,1,1]\), \([2,1]\) and \([3]\), while \(e=1\) takes place when \(r_0=n\) and the Segre symbol is \([(2,1)]\), \([1^2,1]\). 

These Segre symbols are associated with unique pencils up to homography since \(\rho_{\hat \sigma}\) defines at most 3 distinct points in support of \(\calc_\t\) and \(\mathrm{PGL_2}(\kappa)\) acts transitively on triplets of points of \(\p1\).

\medskip

Finally, assume \(\hat n = 1\), so that \(\calt_\sigma \simeq \calo_{\pn}\). Then there are only two possible Segre symbols, whose normal forms give the pencils \((x_0^2,x_1^2)\) and \((x_0x_1,x_0^2)\). Note that this is case \ref{case 0} of Lemma \ref{c1=0,1}.
\end{proof}


\section{Locally free pencils of higher degree} \label{locally free higher}

In this section, \(\kappa\) is any field of characteristic 0.
In contrast with the case of pencils of quadrics seen in the previous section, we will now show that there are locally free pencils of higher degree that are not free.

Before stepping into the general case, let us take a look at the case of pencils of cubics in detail.

Over the complex projective space \(\p3\), there are, according to \cite{BW}, two non-normal cubic surfaces up to homography. In the homogeneous variables \((x_0,\ldots,x_3)\), the equations of these surfaces are:
\[
f = {x}_{1}^{3}+{x}_{0}^{2}{x}_{2}+{x}_{1}^{2}{x}_{3}, ~~,~~ 
g = {x}_{1}^{3}+{x}_{0}^{2}{x}_{2}+{x}_{0}{x}_{1}{x}_{3}.
\]
Both surfaces are singular along the line \(L=\VV(x_0,x_1)\). The Jacobian matrix of the pencil of cubics \(\sigma=(f,g)\) reads:
\[ \nabla\sigma = 
\begin{pmatrix}
2x_0x_2 & 3x_1^2+2x_1x_3  & x_0^2 & x_1^2 \\
2x_0x_2+x_1x_3 & 3x_1^2+x_0x_3 & x_0^2 & x_0x_1
\end{pmatrix}.
\]
The sheaf \(\calq_\sigma\) has rank two over \(L\) and admits no zero-dimensional subsheaf. The first part of Lemma \ref{lf no codim 3} implies that $\calt_\sigma$ is locally free.

However, the scheme-theoretic locus where \(\calq_\sigma\) has rank two has an embedded point at \(p=(0:0:1:0)\). In fact, the Jacobian scheme \(\Xi_\sigma\) has 4 primary components \(P_1,\ldots,P_4\) described by the next table.
\[
\begin{tabular}{c|c|c}
    Dimension & degree & radical ideal \\
\hline
\hline
    1 & 5 & \((x_0,x_1)\) \\ 
    1 & 1 & \((x_1,x_3)\) \\ 
    1 & 1 & \((x_0-x_1,x_3)\) \\ 
    0 & 20 & \((x_0,x_1,x_3)\).
\end{tabular}
\]
Note that $(\Xi_\sigma)_{\rm red}$ consists of the union of the 3 lines $\VV(x_0,x_1)$, $\VV(x_1,x_3)$ and $\VV(x_0-x_1,x_3)$; the first one appears with a multiple structure of degree 5.

In this example, \(\calt_\sigma(2)\) has \(c_1(\calt_\sigma(2))=0\)  and \(c_2(\calt_\sigma(2))=1\), \(c_3(\calt_\sigma(2))=0\). Also, we have \(H^0(\calt_\sigma(2))=0\). Therefore \(\calt_\sigma(2)\) is a null correlation bundle.

This example is generalized to degree $k+3$, for any \(k \ge 0\), in our next result.

\begin{theorem}\label{high deg}
For any $k\ge0$, define the pencil $\sigma=(f,g)$ as:
$$ f= x_0x_1^{k+2}+x_2^{k+3}+x_2^{k+2}x_3,~~
g = x_2x_3(x_2^{k+1}-x_1^{k+1}).
$$
Then we have \(\calt_\sigma \simeq N(-k-2)\), where \(N\) is a null correlation bundle.
\end{theorem}

\begin{proof}
Set \(f=x_0x_1^{k+2}+x_2^{k+3}+x_2^{k+2}x_3\) and \(g=x_2x_3(x_2^k-x_1^k)\).
Observe that, in the algebraic closure of \(\kappa\), the divisor \(\VV(g)\) is an arrangement of planes consisting of \(k+2\) planes \(H_1,\ldots,H_{k+2}\) passing  through the line \(L=\VV(x_1,x_2)\) together with an extra plane \(H\) not containing \(L\), namely the plane \(\VV(x_3)\). This arrangement is free, more precisely, we have :
\[
\calt_{g} \simeq \calo_{\p 3} \oplus\calo_{\p 3}(-1) \oplus \calo_{\p 3}(-k-1). 
\]

Factoring out the trivial summand \(\calo_{\p 3}\) of \(\calt_g\), we write explicitly the syzygy \(\phi : \calo_{\p 3}(-1) \oplus \calo_{\p 3}(-k-1) \to \calo_{\p 3}^{\oplus 3}\),  independently on whether \(\kappa\) is closed or not:
\[\phi =
\begin{pmatrix}
      {x}_{1}&{x}_{2}^{k+1}\\
      {x}_{2}&{x}_{1}^{k}{x}_{2}\\
      -(k+2)\,{x}_{3}&-{x}_{1}^{k}{x}_{3}\end{pmatrix}.
      \]
      The
the Jacobian matrix \(\nabla\sigma : \calo_{\p 3}^{\oplus 4} \to \calo_{\p 3}^{\oplus 2}(k+2)\) reads:
\[
\nabla\sigma = \begin{pmatrix}
x_1^{k+2} & (k+2)x_0x_1^{k+1} & (k+3)x_2^{k+2}+(k+2)x_2^{k+1}x_3 & x_2^{k+2} \\
0 & (k+1)x_1^kx_2x_3 & x_1^{k+1}x_3-(k+2)x_2^{k+1}x_3 & x_1^{k+1}x_2-x_2^{k+2}
\end{pmatrix}.
\]

Note that this matrix has a vanishing entry at the bottom left corner and that the vanishing of this entry corresponds to the trivial summand of \(\calt_g\).
Therefore, projecting onto the last three factors of \(\calo_{\p 3}^{\oplus 4}\) and onto the second factor of \(\calo_{\p 3}^{\oplus 2}(k+2)\) we get a commutative diagram, which is essentially a particular case of the diagram in display \eqref{sub tf sqc}:
\begin{equation} \label{towards NC}
\begin{split} \xymatrix@-1.5ex{
&  & \calo_{\p 3} \ar[d] \ar[r] &  \calo_{\p 3}(k+2) \ar[d] \\
0 \ar[r] & \calt_\sigma \ar[r] \ar[d] & \calo_{\p 3}^{\oplus 4} \ar[d] \ar^-{\nabla\sigma}[r] & \calo_{\p 3}^{\oplus 2}(k+2) \ar[d] \\
0 \ar[r] & \calo_{\p 3}(-1) \oplus \calo_{\p 3}(-k-1) \ar[r] & \calo_{\p 3}^{\oplus 3} \ar^-{\nabla_{g}}[r] & \calo_{\p 3}(k+2)
} \end{split}
\end{equation}

The top arrow defines a surface \(D \subset \p 3\) of degree \(k+2\), whose equation must sit in the top left corner of \(\nabla\sigma\). In other words, \(D=\VV(x_1^{k+2})\) is the \((k+2)\)-tuple structure over the plane \(\VV(x_1)\).

Also, we observe that the image of \(\nabla_{g}\) is \(\cali_{C/\p 3}(k+2)\), where the curve \(C\) is the scheme-theoretic singular locus of \(\VV(g)\) and is defined by the 3 minors of order 2 of \(\phi\). 
Incidentally, over the algebraic closure of \(\kappa\), the curve \(C\) consists of \(k+2\) reduced lines \(L_1,\ldots,L_{k+2}\), with \(L_i=H \cap H_i\) for all \(i \in \{1,\ldots,k+2\}\), together with a \((k+1)\)-tuple complete intersection structure over \(L\) of degree \((k+1)^2\) defined by \(\VV(x_0^{k}x_1,x_0^{k+1}-(k+2)x_1^{k+1})\).

The rightmost column of diagram \eqref{towards NC} gives a surjection \(\calm_\sigma \to \cali_{C/\p 3}(k+2)\), whose kernel is a torsion-free sheaf of rank 1, isomorphic to \(\cali_{B/\p 3}(k+2)\), where the subscheme \(B \subset D \subset \p3\) is defined in \(\p 3\) by the homogeneous polynomials \(h\) of the form \(h= a_0f_0+\cdots+a_3f_3\) with \(a_i \in R=\kappa[x_0,\ldots,x_3]\) and satisfying \(a_1g_1+a_2g_2+a_3g_3=0\), where we put \(f_i = \partial f/\partial x_i\) and \(g_i = \partial g/\partial x_i\), for \(i \in \{0,1,2,3\}\).
Since \(\phi\) accounts for all relations of the homogeneous ideal of \(C\), the homogeneous ideal of \(B\) is thus generated by \((f_0,f_1\phi_{1,1}+f_2\phi_{2,1}+f_3\phi_{3,1},f_1\phi_{1,2}+f_2\phi_{2,2}+f_3\phi_{3,2})\) and the matrix of these generators is:
\[
\begin{pmatrix}
      x_1^{k+2} \\
      (k+2)x_0x_1^{k+2}+(k+3)x_2^{k+3}\\
      (k+2)x_0x_1^{k+1}
      x_2^{k+1}+(k+3)x_1^{k}x_2^{k+3}+(k+1)x_1^{k}x_2^{k+2}x_3
      \end{pmatrix}.
\]
Therefore, the homogeneous ideal \(I_{B/D}\) of \(B\) in \(D=\VV(x_1^{k+2})\) is :
\begin{equation} \label{ideal of B}
\left(x_2^{k+3},(k+2)x_0x_1^{k+1}x_2^{k+1}+(k+1)x_1^{k}x_2^{k+2}x_3\right).    
\end{equation}
We have an exact sequence:
\[
0 \to \cali_{B/\p 3}(k+2) \to \calm_\sigma \to \cali_{C/\p 3}(k+2) \to 0,
\]
and thus, from the leftmost column of diagram \eqref{towards NC}:
\[
0 \to \calt_\sigma \to \calo_{\p 3}(-1) \oplus \calo_{\p 3}(-k-1) \to \cali_{B/D}(k+2) \to 0.
\]

The morphism \(\calo_{\p 3}(-1) \to \cali_{B/D}(k+2)\) is given by the generator \(x_2^{k+3}\) of the ideal of \(B\) in \(D\) as in \eqref{ideal of B}. It defines a curve section \(A\) of \(D\) of degree \(k+3\) which contains \(B\). 
We get an exact sequence:
\[
0 \to \cali_{E/\p 3}(-k-1) \to \calo_{\p 3}(-1-k) \to \cali_{B/A}(k+2) \to 0,
\]
where the curve \(E \subset \p 3\) is defined by the sequence and is cut in \(D\) as the residual scheme of \(B\) with respect to the complete intersection \(A = \VV(x_1^{k+2},x_2^{k+3})\). 
From the diagram \eqref{towards NC}, using the snake lemma we also get:
\begin{equation} \label{the double line}
0 \to \calo_{\p 3}(-k-3) \to \calt_{\sigma} \to \cali_{E/\p 3}(-1-k) \to 0.
\end{equation}
We compute the equations of \(E\) from \eqref{ideal of B} as \(\left(I_{B/D} : (x_2^{k+3})\right)\) and get:
\[
E = \VV\left(x_1^2,x_1x_2,x_2^2,(k+2)x_0x_1-(k+1)x_2x_3\right).
\]
Therefore, the curve \(E\) is a double structure of arithmetic genus \(-1\) over the line \(L\). We conclude from \eqref{the double line} that \(\calt_\sigma(k+2)\) is a null correlation bundle.
\end{proof}

\begin{remark}
By the previous theorem, for any degree \(k+3\) there is a pencil \(\sigma\) which is locally free but not free. Also, we have \(\gpdim(\calt_\sigma)=2\). More precisely, the sheafified minimal graded free resolution of \(\calt_\sigma\) reads:
\[
0 \to \calo_{\p 3}(-k-5) \to \calo_{\p 3}(-k-4)^{\oplus 4} \to \calo_{\p 3}(-k-3)^{\oplus 5} \to \calt_\sigma \to 0.
\]
This is in contrast with the case of pencils of quadrics, where local freeness is equivalent to freeness and where \(\gpdim(\calt_\sigma) \le n-2\).
\end{remark}


\section{Algebraically independent sequences of length 2 and rational foliations}\label{sec:ratfols}

We complete this paper by looking at arbitrary algebraically independent sequences of length 2 and showing how these are related to rational 1-forms, which we now introduce.
Part of the material about foliations, in particular in relation to $k$-webs with $k \ge 3$, is discussed in Appendix A below.

Let $f_1$ and $f_2$ be homogeneous polynomials in $\kappa[x_0,\dots,x_n]$ with no common factors of degree $d_1+1$ and $d_2+1$, respectively; remark that $\sigma:=(f_1,f_2)$ is an algebraically independent sequence. Let also $a$ and $b$ be relatively prime integers such that $(d_1+1)b=(d_2+1)a$. The 1-form
$$ \omega = af_1\cdot \mathrm{d}f_2 - bf_2\cdot \mathrm{d}f_1 \in H^0(\Omega_{\pn}^1(d_1+d_2+2)) $$
is called a \emph{rational 1-form of type $(d_1+1,d_2+1)$}.

Regarding $\omega$ as an element of $\Hom_{\pn}(\tpn,\opn(d+2))$, we set $\calk_\omega:=\ker(\omega)$. Since $\omega$ vanishes along the complete intersection scheme $C:=\VV(\sigma)$, the image of the morphism $\omega:\tpn\to\opn(d+2)$ is actually contained in the ideal sheaf $\cali_C(d+2)$. Applying the functor $\Hom_{\pn}(\opn(1)^{\oplus n+1},-)$ to the resolution of $\cali_C(d+2)$
$$ 0 \lra \opn \stackrel{\tilde\eta}{\lra} \opn(d_1+1)\oplus\opn(d_2+1) \lra \cali_C(d+2) \lra 0, $$
where $\tilde\eta=((d_1+1)f_1 \hspace{3mm} (d_2+1)f_2)^{\rm t}$, we check that the composed morphism
$$ \opn(1)^{\oplus n+1}\onto\tpn\to\cali_C(d+2) $$
lifts to a unique morphism $\mu:\opn(1)^{\oplus n+1}\to\opn(d_1+1)\oplus\opn(d_2+1)$, since
$$ \Hom_{\pn}(\opn(1)^{\oplus n+1},\opn) = \Ext_{\pn}^1(\opn(1)^{\oplus n+1},\opn) = 0 . $$

Therefore we obtain the commutative diagram:
\begin{equation}\label{diag6}
\begin{split} \xymatrix@-1.5ex{ 
& & 0 \ar[d] & 0 \ar[d] & \\
& & \opn\ar[d]^{\eta}\ar@{=}[r] & \opn \ar[d]^{\tilde\eta} \\
0\ar[r] & \calk_\omega \ar[r]\ar@{=}[d] & \opn(1)^{\oplus n+1} \ar[r]^-{\mu}\ar[d] & \opn(d_1+1)\oplus\opn(d_2+1) \ar[d]  \\
0\ar[r] & \calk_\omega \ar[r] & \tpn \ar[r]^{\omega}\ar[d] & \cali_C(d+2) \ar[d] \\
& & 0 & 0 
} \end{split}
\end{equation}
This proves that $\calk_\omega\simeq\ker(\mu)$. We argue that $\mu=\nabla\sigma$, thus in fact $\calk_\omega\simeq\calt_\sigma(1)$.
Indeed, note that 
$$ \omega = \sum_{i=0}^n ~ (pf_1\partial_if_2 - qf_2\partial_if_1) \cdot dx_i,$$
which means that the entries of the morphism $\alpha:\opn(1)^{\oplus n+1}\to\opn(d+2)$ given by the composition
$$ \opn(1)^{\oplus n+1} \onto \tpn \to \cali_C(d+2) \into \opn(d+2) $$
are precisely $\alpha_i=p\cdot f_1\partial_if_2-q\cdot f_2\partial_if_1$. Since, on the other hand, $\alpha=(-qf_2 \hspace{3mm} pf_1)\circ\mu$, we conclude that $\mu=\nabla\sigma$, as desired.

Conversely, given an algebraically independent $\sigma=(f_1,f_2)$ with $\deg(f_i)=d_i+1$, we follow the proof of Lemma \ref{lem:dist} in Section \ref{sec:sqc-dist} and consider the associated codimension 1 distribution $\sD_\sigma$ as presented in display \eqref{distribution}; in the case at hand, this simplifies to (setting $d=d_1+d_2$)
\begin{equation}\label{diag2'}
0\lra \calt_\sigma(1) \lra \tpn \stackrel{\omega}{\lra}  \cali_{\Gamma_\sigma}(d-l+2) \lra 0,
\end{equation}
where $\Gamma_\sigma\subset\pn$ is a (possibly not pure) 2-codimensional subscheme of $\pn$, and $l=c_1(\calq_\sigma)$; this is precisely the codimension one distribution associated with the (possibly non-saturated) twisted rational 1-form
$$ \omega = (d_1+1) f_1\cdot \mathrm{d}f_2 - (d_2+1)f_2\cdot \mathrm{d}f_1 \in \Hom_{\pn}(\tpn,\cali_\Gamma(d-l+2)) \subset H^0(\Omega^1_{\pn}(d+2)). $$
Moreover, the bottom line of the diagram in display \eqref{diag3} yields the following description for the singular scheme $\Gamma_\sigma$ of $\omega$:
\begin{equation} \label{singsch}
0 \lra \cali_{\Gamma_\sigma}(d-l+1) \lra \cali_{C}(d+1) \lra \calq_\sigma \lra 0.
\end{equation}
In particular, we have that
$$ \deg(\Gamma_\sigma)=\deg(\calq_\sigma)+\deg(C)=\deg(\calq_\sigma)+(d_1+1)(d_2+1). $$

Summarizing, we have established a natural 1-1 correspondence between algebraically independent sequences of length 2 and rational codimension one foliations as follows.

\begin{lemma}\label{rat-split}
There exists a 1-1 correspondence between algebraically independent sequences $\sigma=(f_1,f_2)$ on $R$ and rational codimension one foliations $\sD$ on $\pn$ of type $(\deg(f_1),\deg(f_2))$ such that $\Ts(1)=\calt_\sD$ and $\sing(\sD)=\Gamma_\sigma$.
\end{lemma}

The previous statement has the following two important applications when $n=3$. First, as an immediate consequence of \cite[Theorem 6.3]{CCJ1}, we obtain the following stability result for logarithmic sheaves associated with algebraically independent sequences of length 2 on $\p3$.

\begin{corollary}\label{cor:stable}
Let $\sigma=(f_1,f_2)$ be an algebraically independent sequence in $\kappa[x_0,x_1,x_2,x_3]$ and let $d_i:=\deg(f_i)-1$; assume that $d_1+d_2>0$ and $c_1(\calq_\sigma)=0$.
\begin{enumerate}
\item If $d_1+d_2$ is even, then
\begin{itemize}
\item if $\deg(\calq_\sigma)<(d_1^2+d_2^2-d_1-d_2-2)/2$, then $\Ts$ is slope-stable;
\item if $\deg(\calq_\sigma)<(d_1^2+d_2^2+d_1+d_2)/2$, then $\Ts$ is slope-semistable;
\end{itemize}
\item If $d_1+d_2$ is odd and $\deg(\calq_\sigma)<(d_1^2+d_2^2-1)/2$,  then $\Ts$ is slope-stable.
\end{enumerate}
In particular, if the Jacobian scheme is 0-dimensional, then $\calt_\sigma$ is slope-stable.
\end{corollary}

We remark that the previous result is not sharp, and it is not hard to find examples of algebraically independent sequences with slope-stable logarithmic sheaves whose degrees do not satisfy the numerical inequalities above. Indeed, if $\sigma$ corresponds to a pencil of quadrics, so that $d_1=d_2=1$, with $\dim\calg_\sigma=0$, then Corollary \ref{cor:stable} only implies that $\Ts$ is slope-semistable; however, as we have seen in Section \ref{subsec:regular pencils}, $\Ts$ is actually slope-stable in this case. Note that the case $d_1=d_2=1$ is the only one for which the right-hand sides of the inequalities are not positive.

In addition, the higher degree pencils provided in Theorem \ref{high deg} yield yet another set of examples showing that the converse of Corollary \ref{cor:stable} does not hold. 

Finally, as a second application, we give a negative answer to a problem posed by Calvo-Andrade, Cerveau, Giraldo, and Lins Neto, see \cite[Problem 2]{CACGLN}. To be precise, these authors asked whether the tangent sheaf of a codimension one foliation on $\p3$ splits as a sum of line bundles whenever it is locally free. Indeed, in light of the proof of Lemma \ref{rat-split}, the pencils presented in Theorem \ref{high deg} provide examples, for each $k\ge0$, of rational foliations of type $(k+3,k+3)$ on $\p3$ whose tangent sheaves are slope-stable locally free sheaves.


\appendix \label{appendix}
\section{Algebraically independent sequences and foliations (by Alan Muniz)}

The aim of this appendix is to elucidate the relation between logarithmic tangent sheaves and foliations. Precisely, we show that the distribution introduced in Subsection \ref{sec:sqc-dist} is integrable and even more, it is algebraically integrable, i.e., its leaves are contained in the fibers of a rational map. In particular, Proposition \ref{prop:log-fol} below generalizes Lemma \ref{rat-split}.

Fix an algebraically independent sequence $\sigma = (f_1, \dots, f_k)$ with $\deg f_i  = d_i + 1$. We assume, without loss of generality, that each $f_i$ is reduced. Define the integers $e_i = \lcm( d_1 + 1, \dots, d_k+1)/(d_i+1)$ for $i \in \{1,\ldots,k\}$.
Then we consider the rational map 
\[
\phi \colon \pn \dashrightarrow \p{k-1} ; \quad x \mapsto \left(f_1^{e_1}(x): \dots : f_k^{e_k}(x)\right),
\]
which is dominant since $\sigma$ is an algebraically independent sequence. The connected components of its fibers define a codimension $k-1$ foliation on $\pn$ that we denote by $\mathcal{F}_\sigma$. 

In terms of sheaves, $\mathcal{F}_\sigma$ is given locally by the vector fields tangent to the fibers of $\phi$. To be precise, on the affine open subset $V_j = \{f_j \neq 0\}$ we have 
\[
\phi = \left(\frac{f_1^{e_1}}{f_j^{e_j}}, \dots,\widehat{\frac{f_j^{e_j}}{f_j^{e_j}}} ,\dots , \frac{f_k^{e_k}}{f_j^{e_j}} \right)
\]
and a vector field $v$ defined on some open subset $U\subset V_j$ is tangent to $\mathcal{F}_\sigma$ if 
\begin{equation}\label{eq:deffol}
    v\left(\frac{f_i^{e_i}}{f_j^{e_j}}\right) = \iota_v d\!\left(\frac{f_i^{e_i}}{f_j^{e_j}}\right) = 0
\end{equation}
for every $i$, here $\iota_v$ denotes the contraction morphism. This is equivalent to $\mathcal{F}_\sigma$ being defined by the homogeneous $(k-1)$-form $\omega$ satisfying 
\begin{equation}\label{eq:omega}
    g\omega = \iota_\epsilon(df_1 \wedge \dots \wedge df_k)
\end{equation}
where $\epsilon = \sum_j x_j\frac{\partial}{\partial x_j}$ is the Euler radial vector field, and $g$ is the greatest common divisor of the coefficients of the differential form on the right-hand side. Indeed, on $V_j$ we have 

\[
g\omega|_{V_j} = \frac{c f_j^{ke_j}}{f_1^{e_1-1}\dots f_k^{e_k-1}} d\!\left(\frac{f_1^{e_1}}{f_j^{e_j}}\right) \wedge \dots \wedge d\!\left( \frac{f_k^{e_k}}{f_j^{e_j}} \right),
\]
where $c\in \Z_{>0}$. If $v\!\left(\frac{f_i^{e_i}}{f_j^{e_j}}\right) = 0$ then $\iota_v \omega =0$. Conversely,
\begin{equation}\label{eq:resform}
    0 = \iota_v \omega = F \sum_{i = 1} (-1)^{i+1}v\!\left(\frac{f_i^{e_i}}{f_j^{e_j}}\right) d\!\left(\frac{f_1^{e_1}}{f_j^{e_j}}\right) \wedge \dots\wedge \widehat{d\!\left(\frac{f_i^{e_i}}{f_j^{e_j}}\right) } \wedge \dots \wedge d\!\left( \frac{f_k^{e_k}}{f_j^{e_j}} \right),
\end{equation}
where $F= \frac{cf_j^{ke_j}}{gf_1^{e_1-1}\dots f_k^{e_k-1}}$. Thus $v\!\left(\frac{f_i^{e_i}}{f_j^{e_j}}\right) = 0$, for every $i$, on $U \setminus \Delta$, where $\Delta$ is the degeneracy locus of the Jacobian matrix of $\phi$. Since we are in characteristic $0$, $U \setminus \Delta$ is nonempty and open, hence $v\!\left(\frac{f_i^{e_i}}{f_j^{e_j}}\right) = 0$ everywhere on $U$. 


\begin{remark}
When $g=1$, a foliation as above is called a \emph{rational foliation}, see \cite{CPV}. This is a generalization of the codimension one rational foliations given by pencils of hypersurfaces. 

We also remark that $g$ may be non-constant. For instance, take $f_1 = x_1x_3-\frac{1}{2}x_1^2$ and $f_2 = x_2x_3^2 -x_0x_1x_3 +\frac{1}{3}x_1^3$. The foliation associated with $(f_1,f_2)$ is the so called \emph{exceptional foliation} on $\p3$, see \cite[Example 6]{CLn}.
\end{remark}

Next we are interested in computing the tangent sheaf $T\mathcal{F}_\sigma$ of $\mathcal{F}_\sigma$, showing that it coincides with $\mathcal{T}_\sigma(1)$ where $\mathcal{T}_\sigma$ is the logarithmic tangent sheaf associated with $\sigma$. In particular, we show that the distributions $\cald_\sigma$ constructed in Subsection \ref{sec:sqc-dist} are actually integrable.

\begin{proposition}\label{prop:log-fol}
Let $\mathcal{T}_\sigma$ be logarithmic tangent sheaf associated with $\sigma$. Then $\mathcal{T}_\sigma(1)$ is isomorphic to the tangent sheaf $T\mathcal{F}_\sigma$ of $\mathcal{F}_\sigma$.  
 \end{proposition}
 
\begin{proof}

Since $\mathcal{T}_\sigma(1)$ and $T\mathcal{F}_\sigma$ are both reflexive, we only need to show they are isomorphic away from a codimension $2$ set. Let $X = \VV(f_1, \dots, f_k)$ be the algebraic set defined by $\sigma$. Then $\pn \setminus X = \bigcup_l V_l$ where $V_l = \{f_l \neq 0\}$, and we only need to show $\mathcal{T}_\sigma(1)|_{V_l} \simeq T\mathcal{F}_\sigma|_{V_l}$ for every $l$.

Let $U = V_l \cap \{x_0 \neq 0\}$ and fix the natural local coordinates by setting $x_0 = 1$. On $\{x_0 \neq 0\}$, hence on $U$, the projection $\opn^{\oplus n+1}(1)|_U \rightarrow \tpn|_U$ is given by 
\[
v = (a_0, \dots, a_n) \mapsto   [v] = \sum_{j=1}^n (a_j -a_0x_j)\frac{\partial}{\partial x_j}.
\]
Denote by $\pi$ the restriction of the projection above to $\mathcal{T}_\sigma(1)|_{U} \subset \opn^{\oplus n+1}(1)|_U$. We will show that $\pi$ establishes the desired isomorphism, on $U$.

It follows that $v \in \mathcal{T}_\sigma(1)|_{U}$ if and only if $[v] = \pi(v)$ satisfies, for every $i \in \{1,\ldots,k\}$,
\[
[v](f_i) \stackrel{\text{def}}{=} \sum_{j=1}^n (a_j -a_0x_j)\frac{\partial f_i}{\partial x_j} = -a_0(d_i+1)f_i,
\]
here $f_i = f_i|_U$ by abuse of notation. In general, we have 
\begin{equation}\label{eq:derv}
    [v]\!\left(\frac{f_i^{e_i}}{f_l^{e_l}}\right)  = \frac{f_i^{e_i}}{f_l^{e_l}}\left( \frac{e_i[v](f_i)}{f_i} - \frac{e_l[v](f_l)}{f_l} \right).
\end{equation}
For $v \in \mathcal{T}_\sigma(1)|_{U}$ we get $[v]\!\left(\frac{f_i^{p_i}}{f_l^{p_l}}\right) = 0$ and, by definition, $[v] \in T\mathcal{F}_\sigma|_{U}$. That is, $\im \pi \subset T\mathcal{F}_\sigma|_{U}$.

Conversely, for $[v] = \sum_{j=1}^n g_j\frac{\partial }{\partial x_j}  \in T\mathcal{F}_\sigma|_{U}$ we have that $[v]\!\left(\frac{f_i^{e_i}}{f_l^{e_l}}\right) = 0$ for every $i$. Then \eqref{eq:derv} implies that there exists (a unique) $g_0 \in \opn|_U$ such that
\[
[v](f_i)  =  -(d_i+1)g_0 f_i, \text{ for every } i= 1, \dots, k.
\]
We have then a well defined local section $ v= (g_0,g_1 + g_0x_1, \dots, g_n+ g_0x_n) \in \mathcal{T}_\sigma(1)|_{U}$. 

Therefore $\pi \colon \mathcal{T}_\sigma(1)|_{U} \rightarrow T\mathcal{F}_\sigma|_{U}$ is an isomorphism. Repeating this construction for other choices of the open set $U$, yields an isomorphism $\mathcal{T}_\sigma(1)|_{\pn\setminus X} \simeq T\mathcal{F}_\sigma|_{\pn \setminus X}$ hence $\mathcal{T}_\sigma(1) \simeq T\mathcal{F}_\sigma$.

\end{proof}


In view of our present discussion, we remark on the behavior of algebraically independent sequences related to the associated foliations. From $\sigma = (f_1, \dots, f_k)$ we can define a $k$-web $(f_1^{e_1}, \dots, f_k^{e_k})$, as in Subsection \ref{webs}. There, it was introduced the notion of compressibility which is equivalent, due to Lemma \ref{compressibility}, to $H^0(\mathcal{T}_\sigma) \neq 0$. On the other hand, $H^0(\mathcal{T}_\sigma) = H^0(T\mathcal{F}(-1))$, due to Proposition \ref{prop:log-fol}, and this is equivalent to $\mathcal{F}_\sigma$ being a linear pullback foliation. Then, alternatively, Lemma \ref{compressibility} follows from the following result.

\begin{lemma}
Let $\mathcal{F}$ be a codimension $k$ foliation on $\pn$. Then $h^0(T\mathcal{F}(-1)) =: r >0$ if and only if there exists $\pi\colon \pn \dashrightarrow \p{n-r}$ linear and a integrable $k$-form $\eta$ on $\p{n-r}$ such that $\omega = \pi^\ast \eta$ defines $\mathcal{F}$.
\end{lemma}

We suspect that this fact is well-known but we include a proof due to the lack of a precise reference.

\begin{proof}
Let $\omega$ be a homogeneous $k$-form defining $\mathcal{F}$. If $\mathcal{F}$ is a linear pullback $\omega = \pi^\ast \eta$, where $\pi\colon \pn \dashrightarrow \p{q}$ is linear. Then up to linear change in coordinates we assume that $\pi(x_0:\dots:x_n) = (x_0:\dots:x_{q})$ so that $\omega$ does not depend on the variables $x_{q+1}, \dots , x_{n}$. Hence the vector fields $\frac{\partial}{\partial x_j}\in H^0(\tpn(-1))$, $j=q+1, \dots, n$, define elements of $H^0(T\mathcal{F}(-1))$.

Conversely, let $v\in H^0(T\mathcal{F}(-1)) \subset H^0(\tpn(-1))$. Then we claim that the integrability condition and $\codim(\sing \mathcal{F}) \geq 2$ together imply that the Lie derivative vanishes:
\[
\mathcal{L}_v \omega = \iota_v d\omega + d(\iota_v\omega) = 0. 
\]
On the other hand, up to the change in coordinates, we may assume that $\{ \frac{\partial}{\partial x_{n-r+1}}, \dots, \frac{\partial}{\partial x_{n}}\}$ is a basis for $H^0(T\mathcal{F}(-1))$. It is straightforward to verify that $\mathcal{L}_v\omega =0$ with $v =\frac{\partial}{\partial x_j}$ if and only if $\omega$ does not depend on the variable $x_{j}$. Therefore $\omega$ is a linear pullback via the map given by $(x_0: \dots: x_n) \mapsto (x_0:\dots : x_{n-r})$.

Now we have to prove our claim. Let $\omega$ be an integrable homogeneous $k$-form, i.e., for every constant multivector field $y \in \bigwedge^{k-1}\kappa^{n+1}$ we have
    \[
    (\iota_y\omega)\wedge \omega  = (\iota_y\omega) \wedge d\omega = 0,
    \]
and suppose further that $\omega$ does not vanish in codimension one. Assume that $\iota_{v}\omega = 0$, then $\mathcal{L}_{v}\omega = \iota_{v}d\omega$ and, due to integrability, 
\[
(\iota_y\omega)\wedge \mathcal{L}_{v}\omega = \iota_{v}((\iota_y\omega) \wedge d\omega ) = 0
\]
for every $y$. Now let  $u = u_1\wedge \dots \wedge u_k \in \bigwedge^{k}\kappa^{n+1}$  be such that $\iota_u\omega =: f_u \neq 0$, it must exist since $\omega \neq 0$. Define $y_j = (-1)^{j-1}u_1\wedge \dots \wedge \widehat{u_j} \wedge \dots \wedge u_k $ and the $1$-forms $\eta_j= \iota_{y_j}\omega$ so that $\iota_{u_j}\eta_j = \iota_u \omega = f_u$. Applying de Rham--Saito division lemma \cite{SAITO} over the localization $\kappa[x_0, \dots, x_n]_{f_u}$, we see that there exists a polynomial $g_u$ such that
\[
\mathcal{L}_{v}\omega = \frac{g_u}{f^{r+k-1}_u}\eta_1 \wedge \dots \wedge \eta_k = \frac{g_u}{f^{r}_u} \omega 
\]
for some $r\geq 0$. Since $\codim(\sing \omega) \geq 2$, there exists another  $u'\in \bigwedge^{k}\kappa^{n+1}$, also decomposable, such that $\gcd(f_u,f_{u'}) = 1$. From the above equation, we define a rational function 
\[
G= \frac{g_u}{f^{r}_u} = \frac{g_{u'}}{f^{r'}_{u'}} \in \kappa[x_0, \dots, x_n]_{f_u} \cap \kappa[x_0, \dots, x_n]_{f_{u'}} \subset \kappa(x_0 , \dots, x_n)
\]
It follows from the algebraic Hartogs's Lemma \cite[p. 320]{Vakil} and $\gcd(f_u,f_{u'}) = 1$ that
\[
\kappa[x_0, \dots, x_n]_{f_u} \cap \kappa[x_0, \dots, x_n]_{f_{u'}} = \kappa[x_0, \dots, x_n]
\]
hence $G$ is a polynomial. Finally, computing degrees on both sides of the equation $\mathcal{L}_{v}\omega = G\omega$ shows that $G=0$. 
\end{proof}

\bibliographystyle{amsalpha}

\end{document}